\documentclass[letterpaper,11pt]{article}

\usepackage[letterpaper, total={6.75in,8.5in}, centering]{geometry}
\usepackage[utf8]{inputenc}
\usepackage[T1]{fontenc}
\usepackage{graphicx}
\DeclareGraphicsExtensions{.png,.pdf}
\usepackage{xcolor}
\usepackage{authblk}

\usepackage{amsmath}
\usepackage{amsthm}
\usepackage{amssymb}
\usepackage{amsfonts}
\usepackage{mathrsfs}
\usepackage{bm}
\usepackage{mathtools}
\mathtoolsset{centercolon=true}
\DeclareFontFamily{OT1}{pzc}{}
\DeclareFontShape{OT1}{pzc}{m}{it}{<-> s * [1.10] pzcmi7t}{}
\DeclareMathAlphabet{\mathpzc}{OT1}{pzc}{m}{it}
\DeclareFontFamily{U}{BOONDOX-calo}{\skewchar\font=45 }
\DeclareFontShape{U}{BOONDOX-calo}{m}{n}{<-> s*[1.05] BOONDOX-r-calo}{}
\DeclareFontShape{U}{BOONDOX-calo}{b}{n}{<-> s*[1.05] BOONDOX-b-calo}{}
\DeclareMathAlphabet{\mathcalboondox}{U}{BOONDOX-calo}{m}{n}
\SetMathAlphabet{\mathcalboondox}{bold}{U}{BOONDOX-calo}{b}{n}
\DeclareMathAlphabet{\mathbcalboondox}{U}{BOONDOX-calo}{b}{n}
\usepackage{cases}

\usepackage{etoolbox}
\AtBeginEnvironment{algorithmic}{\setstretch{1}}
\usepackage{newfloat}
\usepackage[caption=false]{subfig}
\usepackage{algorithm}
\usepackage{algpseudocode}

\usepackage{setspace}
\usepackage[inline]{enumitem}
\usepackage{xspace}
\usepackage{textpos}
\usepackage{multirow}
\usepackage[title]{appendix}

\newcommand{\Integer}{\mathbb{Z}}
\newcommand{\IntegerP}{\mathbb{Z}_{\geq 0}}
\newcommand{\IntegerPP}{\mathbb{Z}_{>0}}
\newcommand{\Real}{\mathbb{R}}
\newcommand{\RealP}{\mathbb{R}_{\geq 0}}
\newcommand{\RealPP}{\mathbb{R}_{>0}}

\newcommand\given{{\mathbin{}\mid\mathbin{}}}
\newcommand\vect[1]{\mathbf{#1}}

\newcommand\SetSymbol[1][]{\nonscript\:#1\vert\nonscript\:\allowbreak}
\DeclarePairedDelimiterX\Set[1]\{\}{%
  \renewcommand\given{\SetSymbol[\delimsize]}#1} 

\DeclarePairedDelimiterX\innerp[2]{\langle}{\rangle}{#1
  \mathop{}\delimsize|\mathop{} #2} 
\DeclarePairedDelimiterX\norm[1]\lVert\rVert{\ifblank{#1}{\:\cdot\:}{
    #1}}

\DeclareMathOperator{\ri}{ri}
\DeclareMathOperator{\sri}{sri}

\DeclareMathOperator{\domain}{dom}
\DeclareMathOperator{\linspan}{span}

\DeclareMathOperator{\Fix}{Fix}

\DeclareMathOperator{\graph}{gra}

\DeclareMathOperator{\prox}{Prox}

\DeclareMathOperator{\Id}{Id}

\DeclareMathOperator*{\Argmin}{arg\,min}
\DeclareMathOperator*{\ARGMIN}{Arg\,min}

\DeclareMathOperator{\range}{ran}

\newlist{thmlist}{enumerate}{1}
\setlist[thmlist]{label=\textbf{(\roman{thmlisti})},
                  ref=\thetheorem(\roman{thmlisti}),noitemsep}  
\newlist{lemlist}{enumerate}{1}
\setlist[lemlist]{label=\textbf{(\roman{lemlisti})},
                  ref=\thetheorem(\roman{lemlisti}),noitemsep} 
\newlist{exlist}{enumerate}{1}
\setlist[exlist]{label=\textbf{(\roman{exlisti})},
                 ref=\thetheorem(\roman{exlisti}),noitemsep} 
\newlist{factlist}{enumerate}{1}
\setlist[factlist]{label=\textbf{(\roman{factlisti})},
                   ref=\thetheorem(\roman{factlisti}),noitemsep} 
\newlist{proplist}{enumerate}{1}
\setlist[proplist]{label=\textbf{(\roman{proplisti})},
                   ref=\thetheorem(\roman{proplisti}),noitemsep} 
\newlist{asslist}{enumerate}{1}
\setlist[asslist]{label=\textbf{(\roman{asslisti})},
                  ref=\thetheorem(\roman{asslisti}),noitemsep} 
\newlist{deflist}{enumerate}{1}
\setlist[deflist]{label=\textbf{(\roman{deflisti})},
                  ref=\thetheorem(\roman{deflisti}),noitemsep} 
\newlist{algolist}{enumerate}{1}
\setlist[algolist]{label=\textbf{(\roman{algolisti})},
                   ref=\thetheorem(\roman{algolisti}),noitemsep}
\newlist{corlist}{enumerate}{1}      
\setlist[corlist]{label=\textbf{(\roman{corlisti})},
                  ref=\thetheorem(\roman{corlisti}),noitemsep}

\newlist{corlistlist}{enumerate}{1}      
\setlist[corlistlist]{label=\textbf{(\alph{corlistlisti})},
                      ref=\thetheorem(\roman{corlisti})(\alph{corlistlisti}),noitemsep}

\usepackage[capitalize]{cleveref}

\Crefname{theorem}{Theorem}{Theorems}
\Crefname{lemma}{Lemma}{Lemmas}
\Crefname{cor}{Corollary}{Corollaries}
\Crefname{prop}{Proposition}{Propositions}
\Crefname{assumption}{Assumption}{Assumptions}
\Crefname{definition}{Definition}{Definitions}
\Crefname{example}{Example}{Examples}
\Crefname{algo}{Algorithm}{Algorithms}
\Crefname{fact}{Fact}{Facts}

\Crefname{theoremlisti}{Theorem}{Theorems}
\Crefname{lemlisti}{Lemma}{Lemmas}
\Crefname{proplisti}{Proposition}{Propositions}
\Crefname{corlisti}{Corollary}{Corollaries}
\Crefname{corlistlisti}{Corollary}{Corollaries}
\Crefname{asslisti}{Assumption}{Assumptions}
\Crefname{deflisti}{Definition}{Definitions}
\Crefname{exlisti}{Example}{Examples}
\Crefname{algolisti}{Algorithm}{Algorithms}
\Crefname{factlisti}{Fact}{Facts}

\crefname{appsec}{Appendix}{Appendices}

\theoremstyle{plain}% Theorem-like structures provided by amsthm.sty
\newtheorem{theorem}{Theorem}[section]
\newtheorem{lemma}[theorem]{Lemma}
\newtheorem{corollary}[theorem]{Corollary}
\newtheorem{prop}[theorem]{Proposition}

\theoremstyle{definition}
\newtheorem{definition}[theorem]{Definition}
\newtheorem{example}[theorem]{Example}
\newtheorem{problem}[theorem]{Problem}
\newtheorem{fact}[theorem]{Fact}
\newtheorem{assumption}[theorem]{Assumption}

\theoremstyle{remark}

\makeatletter
\newcommand*{\ie}{%
  \@ifnextchar{,}%
  {\textit{i.e.}}%
  {\textit{i.e.,}\@\xspace}%
}
\newcommand*{\eg}{%
  \@ifnextchar{,}%
  {\textit{e.g.}}%
  {\textit{e.g.,}\@\xspace}%
}
\newcommand*{\etc}{%
  \@ifnextchar{.}%
  {\textit{etc}}%
  {\textit{etc.}\@\xspace}%
}
\newcommand*{\cf}{%
  \@ifnextchar{.}%
  {\textit{cf}}%
  {\textit{cf.}\@\xspace}%
}
\newcommand*{\aka}{%
  \@ifnextchar{,}%
  {\textit{a.k.a.}}%
  {\textit{a.k.a.}\@\xspace}%
}
\makeatother

\allowdisplaybreaks

\begin{document}
\sloppy

\title{Fej\'{e}r-Monotone Hybrid Steepest Descent Method for\\ Affinely Constrained and Composite
  Convex Minimization Tasks\footnote{Preliminary parts of this study can be found in
    \cite{Slavakis.ICCOPT.16, Slavakis.ICASSP.17}.}}

\author[1]{Konstantinos~Slavakis}
\author[2]{Isao~Yamada}

\affil[1]{University at Buffalo, The State University of
  New~York,\protect\\ Dept.\ of Electrical Eng., Buffalo, NY 14260-2500,
  USA.\protect\\ Email: \texttt{kslavaki@buffalo.edu}}
\affil[2]{Tokyo Institute of Technology,\protect\\ Dept.\ of Information and
  Communications Eng., Tokyo 152-8550, Japan.\protect\\ Email: \texttt{isao@sp.ce.titech.ac.jp}}

%\date{\today}

\maketitle

\begin{abstract}
  This paper introduces the \textit{Fej\'{e}r-monotone hybrid steepest descent method}\/ (FM-HSDM),
  a new member to the HSDM family of algorithms, for solving affinely constrained minimization tasks
  in real Hilbert spaces, where convex smooth and non-smooth losses compose the objective
  function. FM-HSDM offers sequences of estimates which converge weakly and, under certain
  hypotheses, strongly to solutions of the task at hand. In contrast to its HSDM's precursors,
  FM-HSDM enjoys Fej\'{e}r monotonicity, the step-size parameter stays constant across iterations to
  promote convergence speed-ups of the sequence of estimates to a minimizer, while only Lipschitzian
  continuity, and not strong monotonicity, of the derivative of the smooth-loss function is needed
  to ensure convergence. FM-HSDM utilizes fixed-point theory, variational inequalities and
  affine-nonexpansive mappings to accommodate affine constraints in a more versatile way than
  state-of-the-art primal-dual techniques and the alternating direction method of multipliers
  do. Recursions can be tuned to score low computational footprints, well-suited for large-scale
  optimization tasks, without compromising convergence guarantees. Results on the rate of
  convergence to an optimal point are also presented. Finally, numerical tests on synthetic data are
  used to validate the theoretical findings.
\end{abstract}



\section{Introduction}\label{sec:intro}

\subsection{Problem and notation}\label{sec:the.problem}

\begin{problem}
  This paper considers the following composite convex minimization task:
  \begin{align}
    \min\nolimits_{x\in\mathcal{A}\subset \mathcal{X}} f(x) + g(x)\,,
    \label{the.problem}
  \end{align}
  where $\mathcal{X}$ is a real Hilbert space, the loss functions $f, g$ belong to the class
  $\Gamma_0(\mathcal{X})$ of all convex, proper, and lower-semicontinuous functions from
  $\mathcal{X}$ to $(-\infty, +\infty]$~\cite[p.~132]{Bauschke.Combettes.book}, $f$ is everywhere
  (Fr\'{e}chet) differentiable with $L$-Lipschitz-continuous derivative $\nabla f$, \ie, there
  exists an $L\in\RealPP$ such that (s.t.)
  $\norm{\nabla f(x_1) - \nabla f(x_2)} \leq L\norm{x_1 - x_2}$, $\forall x_1, x_2\in\mathcal{X}$,
  and $\mathcal{A}$ is a closed affine subset of $\mathcal{X}$. Throughout the manuscript, it is
  assumed that \eqref{the.problem} possesses a solution.
\end{problem}

Symbols $\Integer$ and $\Real$ stand for sets of all integer and real numbers,
respectively. Moreover, $\IntegerPP := \Set{1,2, \ldots} \subset \Set{0,1,2,\ldots} =: \IntegerP$,
while $\RealPP:= (0,+\infty)$. The algorithms of this paper are built on a real Hilbert space
$\mathcal{X}$, equipped with an inner product $\innerp{\cdot}{\cdot}$, with vectors denoted by
lower-case letters, \eg, $x$. In the special case where $\mathcal{X}$ is finite dimensional, \ie,
Euclidean, vectors of $\mathcal{X}$ are denoted by boldfaced lower-case letters, \eg, $\vect{x}$,
while boldfaced upper-case letters are reserved for matrices, \eg, $\vect{Q}$. Symbol $\Id$ denotes
the identity mapping in $\mathcal{X}$, \ie, $\Id x = x$, $\forall x\in \mathcal{X}$. In the special
case where $\mathcal{X}$ is Euclidean, $\Id$ boils down to the identity matrix, denoted by
$\vect{I}$. Vector/matrix transposition is denoted by the superscript $\top$. For
$g\in \Gamma_0(\mathcal{X})$, $\partial g$ denotes the set-valued subdifferential operator which is
defined as
$x\mapsto \partial g(x) := \Set{\xi\in\mathcal{X} \given g(x) + \innerp{x'-x}{\xi} \leq g(x'),
  \forall x'\in\mathcal{X}}$.

Let $\mathfrak{B}(\mathcal{X}, \mathcal{X}')$ denote all bounded linear operators from $\mathcal{X}$
to $\mathcal{X}'$~\cite{Kreyszig}, and
$\mathfrak{B}(\mathcal{X}) := \mathfrak{B}(\mathcal{X}, \mathcal{X})$. For
$Q\in \mathfrak{B}(\mathcal{X}, \mathcal{X}')$, $\norm{Q} < \infty$ stands for the norm of
$Q$. Mapping $Q^*\in \mathfrak{B}(\mathcal{X}', \mathcal{X})$ stands for the adjoint of
$Q\in \mathfrak{B}(\mathcal{X}, \mathcal{X}')$~\cite{Kreyszig}. In the case of matrices, the adjoint
of a mapping $\vect{Q}$ is nothing but the transpose $\vect{Q}^{\top}$. Mapping
$Q\in \mathfrak{B}(\mathcal{X})$ is called self adjoint if $Q^*= Q$.
% Further, $\spectrum(Q)\neq \emptyset$ stands for the
% spectrum of $Q$~\cite[Thm.~7.5-4, p.~390]{Kreyszig}, and
In the case of a symmetric matrix $\vect{Q}$, $\lambda(\vect{Q})$
denotes an eigenvalue of $\vect{Q}$. Further,
$\norm{\vect{Q}} = \sigma_{\max}(\vect{Q}) :=
\lambda_{\max}^{1/2}(\vect{Q}^{\top} \vect{Q})$ stands for the
(spectral) norm of $\vect{Q}$, where
$\sigma_{\max}(\cdot)\in \RealPP$ denotes the maximum singular
value and $\lambda_{\max}(\cdot)$ the maximum eigenvalue of a
matrix.

\subsection{Background and contributions}\label{sec:background}

\subsubsection{The hybrid steepest descent method}\label{sec:HSDM}

To solve \eqref{the.problem}, this paper extrapolates the paths established by the hybrid steepest
descent method (HSDM), which was originally introduced to solve a variational-inequality problem of
a strongly-monotone operator over the fixed-point set of a nonexpansive
mapping~\cite{Yamada.HSDM.2001} (see also, \eg, \cite{YO.04, YYY.11, YY.17} and references therein,
for a wider applicability of HSDM in other scenarios). In the context of \eqref{the.problem}, a
version of HSDM solves
\begin{align}
  \min\nolimits_{x\in\Fix T}\ f(x)\,, \label{HSDM.problem} 
\end{align}
where $f$ is a strongly convex function and $\Fix T\subset \mathcal{X}$ denotes the fixed-point set
of a nonexpansive mapping $T:\mathcal{X}\to\mathcal{X}$ (\cf~Sec.~\ref{sec:preliminaries}). For an
arbitrarily fixed starting point $x_0$, HSDM generates the sequence
\begin{align}
  x_{n+1} := Tx_n - \lambda_n \nabla
  f(Tx_n)\,, \label{original.HSDM} 
\end{align}
which strongly converges to the \textit{unique}\/ minimizer of \eqref{HSDM.problem}. To secure
strong convergence, the step sizes $(\lambda_n)_{n\in\IntegerP} \subset\RealP$ satisfy
\begin{enumerate*}[label=\textbf{(\roman*)}]
\item $\sum_{n\in\IntegerP}\lambda_n= +\infty$,
\item $\lim_{n\to\infty}\lambda_n=0$, and
\item $\sum_{n\in\IntegerP} \lvert \lambda_{n+1} - \lambda_{n}
\rvert < +\infty$.
\end{enumerate*}
Further, in the case where $\mathcal{X}$ is Euclidean, $f$ is not necessarily strongly
convex, and $T$ is attracting nonexpansive~\cite{OguraYamadaNonStrictly, HB.Borwein.POCS.96} with
bounded $\Fix T$, the requirements on $(\lambda_n)_{n\in\IntegerP}$ can be relaxed to
\begin{enumerate*}[label=\textbf{(\roman*)}]
\item $\sum_{n\in\IntegerP}\lambda_n= +\infty$,
\item $\sum_{n\in\IntegerP}\lambda_n^2< +\infty$
\end{enumerate*}
for achieving $\lim_{n\to\infty} d_{\mathcal{X}}(\vect{x}_n, \ARGMIN_{\Fix T} f) = 0$, where
$d_{\mathcal{X}}(\vect{x}_n, \ARGMIN_{\Fix T} f)$ stands for the (metric) distance of point
$\vect{x}_n$ from the set of minimizers of $f$ over $\Fix T$~\cite{OguraYamadaNonStrictly}. To
speed up HSDM's convergence rate, conjugate-gradient-based variants were introduced in
\cite{Iiduka.Yamada.09, Iiduka.AMC.11, Iiduka.Math.Prog.15}. For example, for an arbitrarily fixed
starting point $x_0\in\mathcal{X}$, and $d_0:= -\nabla f(x_0)$, the following recursions
\begin{enumerate*}[label=\textbf{(\roman*)}]
\item $x_{n+1} := T(x_n + \mu\lambda_nd_n)$;
\item $d_{n+1} := -\nabla f(x_{n+1}) + \beta_{n+1} d_n$,
\end{enumerate*}
with $\mu>0$, $\lambda_n\in(0,1]$, $\beta_n\in[0,\infty)$, were introduced
in~\cite{Iiduka.Yamada.09}. If $\mu\in (0,2\eta/L^2)$, $\lim_{n\to\infty}\beta_n=0$,
$(\nabla f(x_n))_{n\in\IntegerP}$ is bounded, and
\begin{enumerate*}[label=\textbf{(\roman*)}]
\item $\sum_{n\in\IntegerP}\lambda_n= +\infty$,
\item $\lim_{n\to\infty}\lambda_n=0$,
\item $\sum_{n\in\IntegerP} \lvert \lambda_{n+1} - \lambda_{n} \rvert <
+\infty$, 
\item $\lambda_n/\lambda_{n+1}\leq \sigma$, $(\sigma\geq 1)$,
\end{enumerate*} 
then $(x_n)_{n\in\IntegerP}$ converges strongly to the unique minimizer of \eqref{HSDM.problem}.

\subsubsection{Prior art}\label{sec:PriorArt}

To demonstrate the connections of \eqref{the.problem} with state-of-the-art methods, it is helpful
to notice that the concise description \eqref{the.problem} can be unfolded in several ways to
describe a large variety of convex composite minimization tasks, \eg,
\begin{align}
  \min\nolimits_{x\in\mathcalboondox{A}} \mathcalboondox{f}(x) +
  \sum\nolimits_{j=1}^{J} \mathcalboondox{g}_j(H_jx-r_j)
  \,, \label{problem.multiple.g} 
\end{align}
where $\Set{\mathcalboondox{X}_j}_{j=0}^J$ are real Hilbert spaces,
$\mathcalboondox{f}\in \Gamma_0(\mathcalboondox{X}_0)$,
$\mathcalboondox{g}_j \in \Gamma_0(\mathcalboondox{X}_j)$,
$H_j\in \mathfrak{B}(\mathcalboondox{X}_0, \mathcalboondox{X}_j)$ and
$r_j \in \mathcalboondox{X}_j$, $j\in\Set{1, \ldots, J}$. Moreover, $\nabla \mathcalboondox{f}$ is
$L$-Lipschitz continuous and $\mathcalboondox{A}$ is a closed affine subset of
$\mathcalboondox{X}_0$. Indeed, it can be verified that \eqref{problem.multiple.g} can be recast as
\eqref{the.problem} via
$\mathcal{X} := \mathcalboondox{X}_0 \times \mathcalboondox{X}_1 \times \cdots \times
\mathcalboondox{X}_J = \Set{x:= (x^{(0)}, x^{(1)}, \ldots, x^{(J)}) \given x^{(j)}\in
  \mathcalboondox{X}_j, \forall j\in \Set{0,1, \ldots, J}}$, $f(x):= \mathcalboondox{f}(x^{(0)})$,
$g(x):= \sum_{j=1}^J \mathcalboondox{g}_j(x^{(j)})$, and the closed affine set
$\mathcal{A}:= \Set{x\in \mathcal{X} \given x^{(0)}\in \mathcalboondox{A}, x^{(j)} = H_j
  x^{(0)}-r_j, \forall j\in \Set{1, \ldots, J}}$. Task \eqref{problem.multiple.g}, in the case where
$J=2$, $\mathcalboondox{X} := \mathcalboondox{X}_0 = \mathcalboondox{X}_1$, $H_1 = \Id$,
$r_1 = r_2 = 0$, and $\mathcalboondox{A}:= \mathcalboondox{X}$, \ie,
$\min_{x\in \mathcalboondox{X}} [\mathcalboondox{f}(x) + \mathcalboondox{g}_1(x) +
\mathcalboondox{g}_2(H_2x)]$, has been already studied, \eg, via the primal-dual algorithmic
framework~\cite{Chambolle.Pock.11, Condat.JOTA.13, PLC.Pesquet.composite.12, Vu.13}. Gradient
$\nabla \mathcalboondox{f}$, proximal mappings (\cf~\cref{def:prox}) $\prox_{\mathcalboondox{g}_1}$
and $\prox_{\mathcalboondox{g}^*_2} = \Id - \prox_{\mathcalboondox{g}_2}$~\cite[Rem.~14.4,
p.~198]{Bauschke.Combettes.book}, where $\mathcalboondox{g}^*_2$ stands for the (Fenchel) conjugate
of $\mathcalboondox{g}_2$, as well as adjoint $H_2^*$ are utilized in a computationally efficient
way to generate a sequence $(x_n)_{n\in\IntegerP} \subset \mathcalboondox{X}$, which converges
weakly (and under certain hypotheses, strongly) to a solution of the previous minimization
task. Moreover, task \eqref{problem.multiple.g}, in the case where $J=2$,
$\mathcalboondox{X} := \mathcalboondox{X}_0 = \mathcalboondox{X}_1 = \mathcalboondox{X}_2$,
$H_1 = H_2 = \Id$, $r_1 = r_2 = 0$, and $\mathcalboondox{A}:= \mathcal{X}$, \ie,
$\min_{x\in \mathcalboondox{X}} [\mathcalboondox{f}(x) + \mathcalboondox{g}_1(x) +
\mathcalboondox{g}_2(x)]$, has also attracted attention in the context of the ``three-term operator
splitting'' framework~\cite{Davis.Yin.15, Cevher.Vu.16}. As in \cite{Chambolle.Pock.11,
  Condat.JOTA.13, PLC.Pesquet.composite.12}, $\nabla \mathcalboondox{f}$,
$\prox_{\mathcalboondox{g}_1}$ and $\prox_{\mathcalboondox{g}_2}$ are employed via computationally
efficient recursions in \cite{Davis.Yin.15, Cevher.Vu.16} to generate a sequence which converges
weakly (and under certain hypotheses, strongly) to a solution of the minimization task at hand. All
studies in \cite{Chambolle.Pock.11, Condat.JOTA.13, PLC.Pesquet.composite.12, Davis.Yin.15,
  Cevher.Vu.16} set $\mathcalboondox{A} := \mathcalboondox{X}$. In the case of
$\mathcalboondox{A} \subsetneq \mathcalboondox{X}$, one can accommodate the affine constraint
$\mathcalboondox{A}$ via the use of the indicator function $\iota_{\mathcalboondox{A}}$
[$\iota_{\mathcalboondox{A}}(x):= 0$, if $x\in \mathcalboondox{A}$, and
$\iota_{\mathcalboondox{A}}(x):= +\infty$, if $x\notin \mathcalboondox{A}$] and the additional loss
$\mathcalboondox{g}_3 := \iota_{\mathcalboondox{A}}$. According to the previous discussion, such an
accommodation entails the use of $\prox_{\iota_{\mathcalboondox{A}}} = P_{\mathcalboondox{A}}$,
where $P_{\mathcalboondox{A}}$ denotes the metric projection mapping onto
$\mathcalboondox{A}$. Mapping $P_{\mathcalboondox{A}}$ may become computationally demanding, \eg, in
the case where $\mathcalboondox{X}$ is a Euclidean space and the affine constraints are described by
a matrix of large dimensions (\cf~Fact~\ref{fact:Rockafellar}), since computing
$P_{\mathcalboondox{A}}$ necessitates the costly singular value decomposition of the matrix under
query (\cf~\cref{ex:aff.constraint.LS}). Task \eqref{the.problem} in the case where $\mathcal{X}$ is
a Euclidean space and
$\mathcal{A} := \Set{\vect{x}\in \mathcal{X} \given \vect{a}^{\top} \vect{x}=0}$, for some
$\vect{a}\in \mathcal{X}\setminus\Set{\vect{0}}$, was treated, within a stochastic setting, in
\cite{Necoara.random.14}.

The celebrated alternating direction method of multipliers (ADMM) \cite{glowinski.marrocco.75,
  gabay.mercier.76, Boyd.admm, Bredies.Sun.DR.15, Sun.ADMM.16} deals with the task
\begin{subequations}\label{admm.problem}
  \begin{align}
    \min\nolimits_{(x^{(1)}, x^{(2)})\in\mathcalboondox{X}_1
    \times \mathcalboondox{X}_2}\
    & \mathcalboondox{g}_1(x^{(1)}) +
      \mathcalboondox{g}_2(x^{(2)})\\ 
    \text{s.to}\
    & H_1x^{(1)} + H_2x^{(2)} = r\,, 
  \end{align}
\end{subequations}
where $H_j \in \mathfrak{B}(\mathcalboondox{X}_j, \mathcalboondox{X}_0)$ and
$r\in\mathcalboondox{X}_0$. Again, \eqref{admm.problem} can be recast as \eqref{the.problem} under
the following setting:
$\mathcal{X} := \mathcalboondox{X}_1 \times \mathcalboondox{X}_2 = \Set{x:= (x^{(1)}, x^{(2)})
  \given x^{(1)} \in \mathcalboondox{X}_1, x^{(2)} \in \mathcalboondox{X}_2}$, $f(x) := 0$,
$g(x) := \mathcalboondox{g}_1(x^{(1)}) + \mathcalboondox{g}_2(x^{(2)})$, and
$\mathcal{A} := \Set{x\in \mathcal{X} \given H_1x^{(1)} + H_2x^{(2)} = r}$. Provided that the
inverse mappings $(\lambda H_1^*H_1 + \partial \mathcalboondox{g}_1)^{-1}$ and
$(\lambda H_2^*H_2 + \partial \mathcalboondox{g}_2)^{-1}$ exist, the recursive application of
$(\lambda H_1^*H_1 + \partial \mathcalboondox{g}_1)^{-1}$ and
$(\lambda H_2^*H_2 + \partial \mathcalboondox{g}_2)^{-1}$ generates a sequence which converges
weakly to a solution of \eqref{admm.problem}~\cite{Bredies.Sun.DR.15, Sun.ADMM.16}. ADMM enjoys
extremely wide popularity for minimization problems in Euclidean spaces~\cite{Boyd.admm}, at the
expense of the computation of $(\lambda H_1^*H_1 + \partial \mathcalboondox{g}_1)^{-1}$ and
$(\lambda H_2^*H_2 + \partial \mathcalboondox{g}_2)^{-1}$: there may be cases where computing the
previous inverse mappings entails the costly task of solving a convex minimization sub-problem.

The motivation for the present paper is the algorithmic solution given in the distributed
minimization context of \cite{extra.14, extra.15}: for a Euclidean $\mathcalboondox{X}$, and a
collection of loss functions
$\Set{\mathcalboondox{f}_j, \mathcalboondox{g}_j\in\Gamma_0(\mathcalboondox{X})}_{j=1}^J$, where
$\mathcalboondox{f}_j$ is everywhere differentiable with an $L_j$-Lipschitz continuous
$\nabla \mathcalboondox{f}_j$, $\forall j\in\Set{1, \ldots, J}$, nodes $\mathpzc{N}$
($\lvert \mathpzc{N}\rvert = J$), connected by edges $\mathpzc{E}$ within a network/graph
$\mathpzc{G} := (\mathpzc{N}, \mathpzc{E})$, operate in parallel and cooperate to solve
\begin{subequations}\label{extra.problem}
  \begin{align} 
    \min\nolimits_{(\vect{x}^{(1)},
    \ldots, \vect{x}^{(J)})\in \mathcalboondox{X}^J} 
    & \sum\nolimits_{j=1}^J
      \mathcalboondox{f}_j(\vect{x}^{(j)}) +  
      \sum\nolimits_{j=1}^J
      \mathcalboondox{g}_j(\vect{x}^{(j)}) \\
    \text{s.to}\quad 
    & \vect{x}^{(1)} = \ldots = \vect{x}^{(J)}
      \,, \label{extra.problem.consensus}
  \end{align}
\end{subequations} 
Each node $j\in\mathpzc{N}$ operates only on the pair $(\mathcalboondox{f}_j, \mathcalboondox{g}_j)$
and communicates the information regarding its updates to its neighboring nodes to cooperatively
solve \eqref{extra.problem}, under the consensus constraint of \eqref{extra.problem.consensus}. Once
again, \eqref{extra.problem} can be seen as a special case of \eqref{the.problem} under the
following considerations: $\mathcal{X} := \mathcalboondox{X}^J$,
$f(\vect{x}^{(1)}, \ldots, \vect{x}^{(J)}) := \sum_{j=1}^J \mathcalboondox{f}(\vect{x}^{(j)})$,
$g(\vect{x}^{(1)}, \ldots, \vect{x}^{(J)}) := \sum_{j=1}^J \mathcalboondox{g}(\vect{x}^{(j)})$, and
$\mathcal{A}:= \Set{(\vect{x}^{(1)}, \ldots, \vect{x}^{(J)})\in \mathcal{X} \given \vect{x}^{(1)} =
  \ldots = \vect{x}^{(J)}}$. Upon defining the $J\times J$ mixing matrices $\vect{W} = [w_{ij}]$,
$\tilde{\vect{W}} = [\tilde{w}_{ij}]$, \cite{extra.15} introduced the following recursions to solve
\eqref{extra.problem}: for an arbitrarily fixed starting-point $J\times \dim\mathcalboondox{X}$
matrix $\vect{X}_0$, as well as $\vect{X}_{1/2} := \vect{WX}_0 - \lambda\nabla f(\vect{X}_0)$ and
$\vect{X}_1 := \prox_{\lambda g}(\vect{X}_{1/2})$, repeat for all $n\in\IntegerP$,
\begin{enumerate*}[label=\textbf{(\roman*)}]
\item $\vect{X}_{n+3/2} := \vect{X}_{n+1/2} +
  \vect{W}\vect{X}_{n+1} - \tilde{\vect{W}}\vect{X}_n - \lambda
  [\nabla f(\vect{X}_{n+1}) - \nabla f(\vect{X}_{n})]$;
\item $\vect{X}_{n+2} := \prox_{\lambda g} (\vect{X}_{n+3/2})$.
\end{enumerate*}
If
\begin{enumerate*}[label=\textbf{(\roman*)}]
\item $(i,j)\notin\mathpzc{E}\Rightarrow w_{ij} = \tilde{w}_{ij}
= 0$,
\item $\vect{W}^{\top} = \vect{W}$, $\tilde{\vect{W}}^{\top} =
\tilde{\vect{W}}$,
\item $\ker(\vect{W} - \tilde{\vect{W}}) = \linspan\bm{1} \subset
\ker(\vect{I}- \tilde{\vect{W}})$,
\item $\tilde{\vect{W}} \succ \vect{0}$,
\item $(1/2) (\vect{I} +
\tilde{\vect{W}})\succeq\tilde{\vect{W}}\succeq\vect{W}$, and
\item $\lambda\in (0,{2
\lambda_{\min}(\tilde{\vect{W}})}/{\max_iL_i})$,
\end{enumerate*} 
then the sequence $(\vect{X}_n)_{n\in\mathbb{N}}$ converges to a matrix whose rows provide a
solution to \eqref{extra.problem}.

\subsubsection{Contributions}\label{sec:Contributions}

Driven by the similarity between the algorithmic solution of \cite{extra.14, extra.15} and HSDM, and
aiming at solving \eqref{the.problem}, this study introduces a new member to the HSDM family of
algorithms: the \textit{Fej\'{e}r-monotone}\/ (FM-)HSDM. Building around the simple recursion of
\eqref{original.HSDM} and the concept of a nonexpansive mapping, FM-HSDM's recursions offer
sequences which converge weakly and, under certain hypotheses (uniform convexity of loss functions),
strongly to a solution of \eqref{the.problem}; \cf~Theorems~\ref{thm:basic} and
\ref{thm:plain.vanilla.HSDM}. Fixed-point theory, variational inequalities and affine-nonexpansive
mappings are utilized to accommodate the affine constraint $\mathcal{A}$ in a more flexible way (see,
\eg, \cref{prop:affine.maps} and \cref{ex:aff.constraint.LS}) than the usage of the indicator
function and its associated metric-projection mapping that methods \cite{Condat.JOTA.13,
  PLC.Pesquet.composite.12, Davis.Yin.15, Cevher.Vu.16} promote. Such flexibility is combined with
the first-order information of $f$ and the proximal mapping of $g$ to build recursions of tunable
complexity that can score low-computational-complexity footprints, well-suited for large-scale
minimization tasks. FM-HSDM enjoys Fej\'{e}r monotonicity, and in contrast to \eqref{original.HSDM}
as well as its conjugate-gradient-based variants~\cite{Iiduka.Yamada.09, Iiduka.AMC.11,
  Iiduka.Math.Prog.15}, only Lipschitzian continuity, and not strong monotonicity, of the derivative
of the smooth-part loss is needed to establish convergence of the sequence of estimates. Further, a
constant step-size parameter is utilized to effect convergence speed-ups. Finally, as opposed to
\cite{Iiduka.Yamada.09, Iiduka.AMC.11, Iiduka.Math.Prog.15}, the advocated scheme needs no
boundedness assumptions on estimates or gradients to establish weak (or even strong) convergence of
the sequence of estimates to a solution of \eqref{the.problem}. Results on the rate of convergence
to an optimal point are also presented. Numerical tests on synthetic data are used to validate the
theoretical findings.

\section{Affine nonexpansive mappings and variational inequalities}\label{sec:preliminaries}

\subsection{Nonexpansive mappings and fixed-point sets}

\begin{definition}\label{def:positive.op}
  A self-adjoint mapping $Q\in \mathfrak{B}(\mathcal{X})$ is called \textit{positive}\/ if
  $\innerp{Qx}{x}\geq 0$, $\forall x\in\mathcal{X}$~\cite[\S~9.3]{Kreyszig}. Moreover, the
  self adjoint $\Pi\in \mathfrak{B}(\mathcal{X})$ is called \textit{strongly positive}\/ if there
  exists $\delta\in \RealPP$ s.t.\ $\innerp{\Pi x}{x}\geq \delta\norm{x}^2$,
  $\forall x\in\mathcal{X}$. In the context of matrices, $\vect{Q}$ is positive iff $\vect{Q}$ is
  positive semidefinite, \ie, $\vect{Q}\succeq \vect{0}$. Moreover, $\bm{\Pi}$ is strongly positive
  iff $\bm{\Pi}$ is positive definite, \ie, $\bm{\Pi}\succ \vect{0}$, and $\delta$ in the previous
  definition can be taken to be $\lambda_{\min}(\bm{\Pi})$.
\end{definition}

For a strongly positive $\Pi$, $\innerp{\cdot}{\cdot}_{\Pi}$ stands for the inner product
$\innerp{x}{x'}_{\Pi} := \innerp{x}{\Pi x'}$, $\forall (x, x')\in \mathcal{X}^2$. For a function
$\varphi: \mathcal{X}\to \Real$, $\nabla\varphi$ and $\nabla\varphi(x)$ stand for the
(G\^{a}teaux/Fr\'{e}chet) derivative and gradient at $x\in\mathcal{X}$, respectively~\cite[\S~2.6,
p.~37]{Bauschke.Combettes.book}. Given $Q\in\mathfrak{B}(\mathcal{X})$, $\ker Q$ stands for the
linear subspace $\ker Q := \Set{x \in \mathcal{X} \given Qx = 0}$. Moreover, $\range Q$ denotes the
linear subspace $\range Q := Q\mathcal{X} := \Set{Qx \given x\in\mathcal{X}}$. For the case of a
matrix $\vect{Q}$, $\range \vect{Q}$ is the linear subspace spanned by the columns of
$\vect{Q}$. Finally, the orthogonal complement of a linear subspace is denoted by the superscript
$\perp$.

\begin{definition} The \textit{fixed-point set}\/ of a mapping $T:\mathcal{X} \to\mathcal{X}$ is
  defined as the set $\Fix T := \{x\in \mathcal{X} \given Tx = x\}$.
\end{definition}

\begin{definition} Mapping $T:\mathcal{X}\to\mathcal{X}$ is called
  \begin{deflist}
  \item \textit{Nonexpansive,} if $\norm{Tx - Tx'} \leq \norm{x - x'}$,
    $\forall (x, x')\in \mathcal{X}^2$.

  \item\label{def:fne} \textit{Firmly nonexpansive,} if
    $\norm{Tx - Tx'}^2 \leq \innerp{x - x'}{Tx - Tx'}$, $\forall (x, x')\in \mathcal{X}^2$. Any
    firmly nonexpansive mapping is nonexpansive~\cite[\S4.1]{Bauschke.Combettes.book}.

  \item \textit{$\alpha$-averaged (nonexpansive),} if there exist an $\alpha\in (0,1)$ and a
    nonexpansive mapping $R:\mathcal{X}\to \mathcal{X}$ s.t.\ $T = \alpha R + (1-\alpha)\Id$. It can
    be easily verified that $T$ is nonexpansive with $\Fix R= \Fix T$.

  \end{deflist}
\end{definition}

\begin{fact}[\protect{\cite[Cor.~4.15, p.~63]{Bauschke.Combettes.book}}] 
  The fixed-point set $\Fix T$ of a nonexpansive mapping $T$ is closed and convex.
\end{fact}

\begin{definition}\label{def:prox} Given 
  $f\in\Gamma_0(\mathcal{X})$ and $\gamma\in\RealPP$, the \textit{proximal}\/ mapping
  $\prox_{\gamma f}$ is defined as
  $\prox_{\gamma f} : \mathcal{X}\to \mathcal{X}: x\mapsto \Argmin\nolimits_{z\in \mathcal{X}}
  (\gamma f(z) + \tfrac{1}{2} \norm{x - z}^2)$.
\end{definition}

\begin{example}\label{ex:nonexp.maps}\mbox{}

  \begin{exlist}

  \item\label{ex:projection} \cite[Prop.~4.8, p.~61]{Bauschke.Combettes.book} Given a non-empty
    closed convex set $\mathcal{C}\subset\mathcal{X}$, the \textit{metric projection mapping onto
      $\mathcal{C}$,} defined as
    $P_{\mathcal{C}}: \mathcal{X}\to \mathcal{C}: x\mapsto P_{\mathcal{C}}x$, with
    $P_{\mathcal{C}}x$ being the unique minimizer of $\min_{z\in\mathcal{C}} \norm{x-z}$, is firmly
    nonexpansive with $\Fix P_{\mathcal{C}} = \mathcal{C}$.

  \item\label{ex:prox} \cite[Prop.~12.27, p.~176]{Bauschke.Combettes.book} Given
    $f\in\Gamma_0(\mathcal{X})$ and $\gamma\in\RealPP$, the proximal mapping $\prox_{\gamma f}$ is
    firmly nonexpansive with $\Fix \prox_{\gamma f} = \Argmin f$.

  \item\label{ex:firmly.nonexp} \cite[Prop.~4.2, p.~60]{Bauschke.Combettes.book} $T$ is firmly
    nonexpansive iff $\Id-T$ is firmly nonexpansive iff $T$ is $(1/2)$-averaged iff $2T-\Id$ is
    nonexpansive.

  \item\label{ex:convex.comb.maps} \cite[Prop.~2.2]{PLC.Isao.JMAA.15},
    \cite[Thm.~3(b)]{OguraYamadaNonStrictly}. Let $\Set{T_j}_{j=1}^J$ be a finite family
    ($J\in\IntegerPP$) of nonexpansive mappings from $\mathcal{X}$ to $\mathcal{X}$, and
    $\Set{\omega_j}_{j=1}^J$ be real numbers in $(0,1]$ s.t.\ $\sum_{j=1}^J \omega_j = 1$. Then,
    $T:= \sum_{j=1}^J \omega_j T_j$ is nonexpansive. If $\cap_{j=1}^J \Fix T_j\neq \emptyset$, then
    $\Fix T = \cap_{j=1}^J \Fix T_j$. Further, consider real numbers
    $\Set{\alpha_j}_{j=1}^J\subset (0,1)$ s.t.\ $T_j$ is $\alpha_j$-averaged, $\forall j$. Define
    $\alpha := \sum_{j=1}^J \omega_j\alpha_j$. Then, $T$ is $\alpha$-averaged. Hence, if each $T_j$
    is firmly nonexpansive, \ie, $(1/2)$-averaged, then $T$ is also firmly nonexpansive.

  \item\label{ex:composition.maps} \cite[Prop.~2.5]{PLC.Isao.JMAA.15},
    \cite[Thm.~3(b)]{OguraYamadaNonStrictly} Let $\Set{T_j}_{j=1}^J$ be a finite family
    ($J\in\IntegerPP$) of nonexpansive mappings from $\mathcal{X}$ to $\mathcal{X}$. Then, mapping
    $T:= T_1T_2\cdots T_J$ is nonexpansive. If $\cap_{j=1}^J \Fix T_j\neq \emptyset$, then
    $\Fix T = \cap_{j=1}^J \Fix T_j$. Further, consider real numbers
    $\Set{\alpha_j}_{j=1}^J\subset (0,1)$ s.t.\ $T_j$ is $\alpha_j$-averaged, $\forall j$. Define
    \begin{align*} \alpha := \frac{1}{1+ \frac{1}{\sum_{j=1}^J
          \frac{\alpha_j}{1-\alpha_j}}} \,.
    \end{align*}
    Then, $T$ is
    $\alpha$-averaged.

  \end{exlist}

\end{example}

In what follows, function $f\in\Gamma_0(\mathcal{X})$ is considered to have an $L$-Lipschitz
continuous $\nabla f$ with $\domain \nabla f = \mathcal{X}$. By \cite[Prop.~16.3(i),
p.~224]{Bauschke.Combettes.book}, the previous condition leads to $\domain f = \mathcal{X}$, which
further implies by \cite[Cor.~16.38(iii), p.~234]{Bauschke.Combettes.book} that
$\partial(f+g) = \nabla f + \partial g$.

\subsection{Affine nonexpansive mappings}\label{sec:AffineMaps}

\begin{definition}[\protect{\cite[p.~3]{Bauschke.Combettes.book}}] \label{def:affine.map} A mapping
  $T:\mathcal{X}\to\mathcal{X}$ is called \textit{affine}\/ if there exist a linear mapping
  $Q:\mathcal{X}\to\mathcal{X}$ and a $\pi\in\mathcal{X}$ s.t.\ $Tx = Qx + \pi$,
  $\forall x \in \mathcal{X}$.
\end{definition}

\begin{fact}[\protect{\cite[Ex.~4.4,
    p.~72]{Bauschke.Combettes.book}}]\label{fact:affine.nonexp.T}
  Consider the affine mapping $Tx = Qx +\pi$, $\forall x\in\mathcal{X}$, with $Q$ being linear and
  $\pi\in\mathcal{X}$. Then, $T$ is nonexpansive iff $\norm{Q}\leq 1$.
\end{fact}

Define now the following special class of affine-nonexpansive mappings:
\begin{align}\label{T.family}
  \mathfrak{T} := \Set*{T:\mathcal{X}\to
  \mathcal{X} 
  \given \begin{aligned}
    & Tx = Qx + \pi, \forall x\in\mathcal{X};\\
    & Q\in \mathfrak{B}(\mathcal{X}); \pi\in\mathcal{X};\\
    & \norm{Q}\leq 1, Q\ \text{is positive} \\
  \end{aligned}}\,.
\end{align}

As the following proposition highlights, $T$ is nothing but the class of affine firmly nonexpansive
mappings.

\begin{prop}\label{prop:ClassT.fne}
  $T\in \mathfrak{T}$ iff $T = Q + \pi$, where $Q\in\mathfrak{B}(\mathcal{X})$ is self-adjoint,
  $\pi\in\mathcal{X}$, and $T$ is firmly nonexpansive.
\end{prop}

\begin{proof}
  First, consider $T\in \mathfrak{T}$. Since $Q$ is positive, let $Q^{1/2}$ be the \textit{positive
    square root}\/ of $Q$, \ie, the (unique) positive operator which satisfies $Q^{1/2} Q^{1/2} =
  Q$~\cite[Thm.~9.4-2, p.~476]{Kreyszig}. The positivity of $Q$ yields
  $\norm{Q} = \sup_{x\in\mathcal{X} \setminus\Set{0}} |\innerp{Qx}{x}|/\innerp{x}{x} =
  \sup_{x\in\mathcal{X} \setminus\Set{0}} \innerp{Qx}{x}/\innerp{x}{x}$, according to \cite[Thm.~9.2-2,
  p.~466]{Kreyszig}. Then, $\forall (x,x')\in \mathcal{X}^2$,
  \begin{align*}
    \norm{Tx-Tx'}^2
    & = \norm{Qx-Qx'}^2 = \norm{Q(x-x')}^2 = \innerp{Q(x-x')}{Q(x-x')} \\
    & = \innerp{Q^{1/2}(x-x')}{QQ^{1/2}(x-x')} \leq \norm{Q} \innerp{Q^{1/2}(x-x')}{Q^{1/2}(x-x')}
    \\
    & \leq \innerp{Q^{1/2}(x-x')}{Q^{1/2}(x-x')} = \innerp{x-x'}{Q(x-x')} \\
    & = \innerp{x-x'}{Tx-Tx')}\,,
  \end{align*}
  which suggests that $T$ is firmly nonexpansive. 

  Now, let $T = Q + \pi$, for a self-adjoint $Q\in\mathfrak{B}(\mathcal{X})$, $\pi\in\mathcal{X}$.
  Let also $T$ be firmly nonexpansive. Then, $\forall x\in\mathcal{X}$,
  $\innerp{x}{Qx} = \innerp{x-0}{Q(x-0)} = \innerp{x-0}{Tx-T0)} \geq \norm{Tx-T0}^2 \geq 0$; thus
  $Q$ is positive. By the fact that a firmly nonexpansive mapping is nonexpansive [\cref{def:fne}]
  and \cref{fact:affine.nonexp.T}, $\norm{Q}\leq 1$. In summary, $T\in\mathfrak{T}$.
\end{proof}

\begin{prop}\label{prop:affine.maps} Let $J\in\IntegerPP$.

  \begin{proplist}

  \item\label{prop:convex.comb.affine} Consider a family $\{T_j\}_{j=1}^J$ of members of
    $\mathfrak{T}$. For any set of weights $\Set{\omega_j}_{j=1}^J$ s.t.\ $\omega_j\in (0,1]$ and
    $\sum_{j=1}^J \omega_j = 1$, mapping $\sum_{j=1}^J \omega_jT_jx \in\mathfrak{T}$.

  \item\label{prop:compose.affine} Consider $T_0:= Q_0 + \pi_0 \in\mathfrak{T}$. Moreover, let the
    self adjoint $Q_j\in \mathfrak{B}(\mathcal{X})$, with $\norm{Q_j}\leq 1$, and
    $\pi_j\in\mathcal{X}$, $\forall j\in\{1, \ldots, J\}$. Let now the family
    $\{T_j:= Q_j + \pi_j\}_{j=1}^J$ of affine nonexpansive mappings, where each $T_j$ does not
    necessarily belong to $\mathfrak{T}$, \ie, $\{Q_j\}_{j=1}^J$ might not be positive according to
    \cref{prop:ClassT.fne}. Then, the composition
    \begin{alignat*}{2}
      T_JT_{J-1}\cdots T_{1}T_0T_{1}\cdots T_{J-1}T_Jx
      && \,\mathbin{=}\, & Q_J Q_{J-1} \cdots Q_{1} Q_0 Q_{1} \cdots Q_{J-1} Q_J x \\
      &&& + \sum\nolimits_{j=1}^J Q_J Q_{J-1}\cdots Q_{1} Q_0 Q_{1}\cdots Q_{j-1} \pi_j\\
      &&& + \sum\nolimits_{j=1}^J Q_J Q_{J-1}\cdots Q_j \pi_{j-1} + \pi_J\,,
      \quad\forall x \in\mathcal{X}\,,
    \end{alignat*}
    satisfies
    $T_JT_{J-1}\cdots T_{1}T_0T_{1}\cdots T_{J-1}T_J \in \mathfrak{T}$.

  \end{proplist}
\end{prop}

\begin{proof} The proof of \cref{prop:convex.comb.affine} follows easily from
  \cref{ex:convex.comb.maps} and \cref{prop:ClassT.fne}. The formula appearing in
  \cref{prop:compose.affine} can be deduced by mathematical induction on $J$. Further,
  $Q_J Q_{J-1} \cdots Q_{1} Q_0 Q_{1} \cdots Q_{J-1} Q_J$ is self adjoint, and its positivity
  follows from the fundamental observation that $\forall x\in\mathcal{X}$,
  $\innerp{Q_J Q_{J-1} \cdots Q_{1} Q_0 Q_{1} \cdots Q_{J-1} Q_Jx}{x} = \innerp{Q_0 (Q_{1} \cdots
    Q_{J-1} Q_Jx)}{Q_{1} \cdots Q_{J-1} Q_Jx}\geq 0$, due to the positivity of $Q_0$. Finally, the
  claim of \cref{prop:compose.affine} is established by
  $\norm{Q_J\cdots Q_{1} Q_0 Q_{1}\cdots Q_J} \leq \norm{Q_0}\prod_{j=1}^J \norm{Q_j}^2 \leq 1$.
\end{proof}

\begin{prop}\label{prop:T(A).family}
  Given the closed affine set $\mathcal{A}\subset \mathcal{X}$,
  define the following family of mappings:
  \begin{align}\label{T(A).family}
    \mathfrak{T}_{\mathcal{A}} := \Set*{T\in\mathfrak{T}
    \given \Fix T = \mathcal{A}}\,.
  \end{align}
  Then, $\mathfrak{T}_{\mathcal{A}}$ is non-empty.
\end{prop}

\begin{proof}
  The metric projection mapping $P_{\mathcal{A}}$ onto $\mathcal{A}$ is not only firmly nonexpansive
  with $\Fix P_{\mathcal{A}} = \mathcal{A}$ [\cf~\cref{ex:projection}], but also affine, according
  also to \cite[Cor.~3.20(ii), p.~48]{Bauschke.Combettes.book}. Hence, by virtue of
  \cref{prop:ClassT.fne}, $P_{\mathcal{A}}\in \mathfrak{T}_{\mathcal{A}}\neq \emptyset$.
\end{proof}

It can be verified that the fixed-point set $\Fix T$ of an affine mapping $T$ is affine. However,
more can be said about the members of $\mathfrak{T}_{\mathcal{A}}$.

\begin{prop}\label{prop:FixT} For any $T\in
  \mathfrak{T}_{\mathcal{A}}$,
  \begin{align*} 
    \mathcal{A} = \Fix T = \ker(\Id-Q) + w_* = \ker U + w_* \,,
  \end{align*}
  where $w_*$ is any vector of $\mathcal{A}$, and $U$ is the positive square root of $\Id-Q$, \ie,
  the (unique) positive operator which satisfies $U^2 = \Id - Q$~\cite[Thm.~9.4-2,
  p.~476]{Kreyszig}.
\end{prop}

\begin{proof}
  Since
  $\norm{Q} = \sup_{x\in\mathcal{X} \setminus\Set{0}} |\innerp{Qx}{x}|/\norm{x}^2$~\cite[Thm.~9.2-2,
  p.~466]{Kreyszig} and $\norm{Q}\leq 1$, it can be easily verified that $\forall x\in \mathcal{X}$,
  $\innerp{(\Id-Q)x}{x} = \norm{x}^2 - \innerp{Qx}{x} \geq \norm{x}^2 - \norm{Q} \cdot
  \norm{x}^2\geq \norm{x}^2 - \norm{x}^2=0$, \ie, $\Id-Q$ is positive. Interestingly, the
  positivity of $Q$ suggests that $\forall x\in \mathcal{X}$,
  $\innerp{(\Id-Q)x}{x} = \norm{x}^2 - \innerp{Qx}{x}\leq \norm{x}^2$, which implies, via
  \cite[Thm.~9.2-2, p.~466]{Kreyszig}, that $\norm{\Id-Q} \leq 1$. Moreover, by the definition of
  $T$, it follows that for any arbitrarily fixed $w_*\in \Fix T$,
  \begin{align*}
    \Fix T
    & = \Set*{x \given Tx = x} = \Set*{x \given (\Id-T)x = 0} \\
    & = \Set*{x \given (\Id - Q) x =
      \pi} = \Set*{x \given (\Id - Q)x = (\Id - Q)w_*} \\
    & = \Set*{x \given
      (\Id - Q)(x-w_*)= 0} = \Set*{x' + w_* \given (\Id - Q) x' = 
      0} \\
    & = \ker(\Id-Q) + w_* \,.
  \end{align*}
  Finally, the characterization $\Fix T = \ker U + w_*$ follows from the previous arguments and
  $x'\in\ker(\Id-Q) \Leftrightarrow (\Id - Q)x' = 0 \Rightarrow U^2x' = 0 \Rightarrow U^*Ux' = 0
  \Rightarrow \innerp{x'}{U^*Ux'} = \innerp{Ux'}{Ux'} = \norm{Ux'}^2 = 0 \Rightarrow Ux' = 0
  \Leftrightarrow x'\in\ker U \Rightarrow U^2x' = 0 \Rightarrow (\Id - Q)x' = 0 \Leftrightarrow
  x'\in\ker(\Id-Q)$, which establishes $\ker(\Id-Q) = \ker U$.
\end{proof}

Several examples of $\mathfrak{T}_{\mathcal{A}}$ members playing important roles in convex
minimization tasks can be found in \cref{sec:appendix}.

\subsection{Variational-inequality problems}\label{sec:VIP}

\begin{definition}[Variational-inequality problem]\label{def:VIP}
  For a nonexpansive mapping $T:\mathcal{X}\to\mathcal{X}$, point $x_*\in\Fix T$ is said to solve
  the variational-inequality problem $\text{VIP}(\nabla f + \partial g, \Fix T)$ if there exists
  $\xi_*\in \partial g(x_*)$ s.t.\ $\forall y\in \Fix T$,
  $\innerp{y- x_*}{\nabla f(x_*) + \xi_*} \geq 0$.
\end{definition}

\begin{fact}[\protect{\cite[Prop.~26.5(vi),
    p.~383]{Bauschke.Combettes.book}}]
  Consider a mapping $T\in\mathfrak{T}_{\mathcal{A}}$ (recall $\Fix T = \mathcal{A}$), and assume
  that one of the following holds:
  \begin{enumerate}
  \item $0\in \sri (\mathcal{A}-\domain (f+g))$ (\cf\cite[Prop.~6.19, p.~95]{Bauschke.Combettes.book}
    for special cases);
  \item $\mathcal{X}$ is Euclidean and $\mathcal{A} \cap \ri\,
    [\domain(f+g)] \neq \emptyset$.
  \end{enumerate}
  Then, point $x_*$ solves $\text{VIP}(\nabla f + \partial g, \Fix T)$ iff
  $x_* \in \Argmin_{x\in\Fix T} [f(x) + g(x)]$.
\end{fact}

\begin{prop}\label{prop:O*}
  \begin{subequations}
    Given the closed affine set $\mathcal{A}\subset \mathcal{X}$, consider any
    $T\in \mathfrak{T}_{\mathcal{A}}$ (\cf~\cref{prop:FixT}). If $U$ stands for the square root of
    the linear operator $\Id-Q$ in the description of $T$ (\cf~\cref{def:affine.map}), let
    $\mathop{\overline{\range}} U$ denote the closure (in the strong topology) of the range of
    $U$. Then,
    \begin{align}
      & x_*\ \text{solves}\ \text{VIP}(\nabla f + \partial g, \Fix
      T) \notag\\
      & \qquad \Leftrightarrow 
        x_*\in \mathcal{A}_* := \Set*{x\in\Fix T \given \left[\nabla
      f(x) + \partial g(x) \right] \cap \mathop{\overline{\range}} U \neq
        \emptyset}\,. \label{solve.VIP.inf.dim}
    \end{align}
    Moreover, for an arbitrarily fixed $\lambda\in\Real\setminus\Set{0}$, define the subset
    \begin{align}
      \Upsilon_*^{(\lambda)} := 
      \Set*{(x,v) \in\Fix T\times \mathcal{X}\given -
      \tfrac{1}{\lambda} Uv \in \nabla f(x) + \partial
      g(x)}\,. \label{O.lambda}
    \end{align}
    Then,
    \begin{align}
      (x_*, v_*)\in \Upsilon_*^{(\lambda)} \Rightarrow x_*\
      \text{solves}\ \text{VIP}(\nabla f + \partial g, \Fix T)
      \,. \label{O.lambda.means.Gamma*}
    \end{align}
    Further, in the case where $\mathcal{X}$ is finite dimensional,
    \begin{align}
      \vect{x}_*\ \text{solves}\
      \text{VIP}(\nabla f + \partial g, \Fix T)
      \Leftrightarrow 
      \exists\vect{v}_*\in\mathcal{X}\ \text{s.t.}\
      (\vect{x}_*, \vect{v}_*)\in \Upsilon_*^{(\lambda)} \,. \label{O.lambda.finite.dim}
      \end{align}
  \end{subequations}
\end{prop}

\begin{proof}
  First, recall that
  $(\ker U)^{\perp} = \mathop{\overline{\range}} U^* = \mathop{\overline{\range}}
  U$~\cite[Fact~2.18(iii), p.~32]{Bauschke.Combettes.book}. According to \cref{def:VIP},
  \begin{subequations}
    \begin{align}
      & x_*\ \text{solves}\ \text{VIP}(\nabla f
        + \partial g, \Fix T) \notag\\
      & \Leftrightarrow\ x_*\in\Fix T\
        \text{and}\ \exists \xi_*\in \partial g(x_*)\
        \text{s.t.}\ \forall y\in\Fix T,\ \innerp{y -
        x_*}{\nabla f(x_*) + \xi_*} \geq 0 \notag\\
      & \Leftrightarrow\ x_*\in\Fix T\ \text{and}\ \exists 
        \xi_*\in \partial g(x_*)\ \text{s.t.}\
        \forall z\in \ker U,\ \innerp{z}{\nabla
        f(x_*) + \xi_*} \geq 0 \label{use.FixT.description}\\
      & \Leftrightarrow\ x_*\in\Fix T\ \text{and}\ \exists 
        \xi_*\in \partial g(x_*)\ \text{s.t.}\
        \forall z\in \ker U,\ \innerp{z}{\nabla
        f(x_*) + \xi_*} \leq 0 \label{lin.vector.space.ineq}\\
      & \Leftrightarrow\
        x_*\in\Fix T\ \text{and}\ \exists \xi_*\in \partial
        g(x_*)\ \text{s.t.}\ \forall z\in \ker U,\
        \innerp{z}{\nabla f(x_*) + \xi_*} = 0 \notag\\
      & \Leftrightarrow\
        x_*\in\Fix T\ \text{and}\ \exists \xi_*\in \partial
        g(x_*)\ \text{s.t.}\ \nabla f(x_*) +
        \xi_*\in (\ker U)^{\perp} = \mathop{\overline{\range}} U
        \notag \\
      & \Leftrightarrow\
        x_*\in\Fix T\ \text{and}\ \left[\nabla f(x_*) + \partial
        g(x_*) \right] \cap \mathop{\overline{\range}} U \neq
        \emptyset \notag\\
      & \Leftrightarrow\ x_*\in\mathcal{A}_*\,,
    \end{align}
  \end{subequations}
  which establishes \eqref{solve.VIP.inf.dim}. Notice that \cref{prop:FixT} is used in
  \eqref{use.FixT.description} and $z\in \ker U \Leftrightarrow -z\in \ker U$ in
  \eqref{lin.vector.space.ineq}.

  Moreover,%
  \begin{subequations}%
    \begin{align}
      & (x_*, v_*)\in \Upsilon_*^{(\lambda)} \notag\\
      & \Leftrightarrow\ x_*\in\Fix T\ \text{and}\
        U \left(-\tfrac{v_*}{\lambda}\right) \in \nabla f(x_*) +
        \partial g(x_*) \notag\\
      & \Leftrightarrow\ x_*\in\Fix T\ \text{and}\ \exists
        v_*'\in\mathcal{X}\ \text{s.t.}\ Uv_*' \in \left[
        \nabla f(x_*) + \partial g(x_*) \right] \cap \range U
        \quad \left( v_*' = -\tfrac{v_*}{\lambda}\right)
        \notag\\
      & \Leftrightarrow\ x_*\in\Fix T\ \text{and}\
        \left[\nabla f(x_*) + \partial g(x_*) \right] \cap \range
        U \neq \emptyset \notag\\
      & \Rightarrow\ x_*\in\Fix T\ \text{and}\
        \left[\nabla f(x_*) + \partial g(x_*) \right] \cap
        \mathop{\overline{\range}} U \neq \emptyset \label{->}\\
      & \Leftrightarrow\ x_*\in \mathcal{A}_*\,, \notag
    \end{align}
  \end{subequations}%
  which establishes \eqref{O.lambda.means.Gamma*} via \eqref{solve.VIP.inf.dim}.

  In the case where $\mathcal{X}$ is Euclidean, \eqref{O.lambda.finite.dim} is established by the
  well-known fact $\mathop{\overline{\range}} U = \range U$~\cite[Thm.~2.4-3, p.~74]{Kreyszig},
  which turns ``$\Rightarrow$'' into ``$\Leftrightarrow$'' in \eqref{->}.
\end{proof}

\section{Algorithm and convergence analysis}\label{sec:algo} 

For any $T\in\mathfrak{T}_{\mathcal{A}}$ and any $\alpha\in (0,1)$, define the $\alpha$-averaged
mapping
\begin{align} 
  T_{\alpha}x := [\alpha T + (1-\alpha)
  \Id] x = Q_{\alpha}x +
  \alpha \pi\,, \label{def.Talpha}
\end{align} where $Q_{\alpha}:= \alpha Q +
(1-\alpha)\Id$.

\begin{theorem}\label{thm:basic}
  Consider $f,g\in \Gamma_0(\mathcal{X})$, with $L$ being the Lipschitz-continuity constant of
  $\nabla f$. Moreover, given the closed affine set $\mathcal{A}$, consider any
  $T\in\mathfrak{T}_{\mathcal{A}}$. For $\lambda\in\RealPP$, an arbitrarily fixed
  $x_0\in\mathcal{X}$, and for all $n\in\IntegerP$, the \textit{Fej\'{e}r-monotone hybrid steepest
    descent method}\/ (FM-HSDM) is stated as follows:
  \begin{subequations}\label{FM-HSDM}
    \begin{align}
      x_{1/2}
      & := T_{\alpha}x_0 - \lambda \nabla
        f(x_0)\,, \label{FM-HSDM.half}\\ 
      x_{1}
      & := \prox_{\lambda g}
        (x_{1/2})\,, \label{FM-HSDM.one}\\ 
      x_{n+3/2}
      & := x_{n+1/2} - \left[ T_{\alpha}x_n - \lambda \nabla
        f(x_n) \right] + \left[ Tx_{n+1} - \lambda \nabla
        f(x_{n+1})\right] \,, \label{FM-HSDM.n+one+half}\\
      x_{n+2}
      & := \prox_{\lambda g} (x_{n+3/2})\,. \label{FM-HSDM.n+two}
    \end{align}
  \end{subequations}
  Consider also $\alpha\in [0.5,1)$ and $\lambda\in (0,2(1-\alpha)/L)$. Then, the following hold
  true.
  
  \begin{thmlist}

  \item There exist a sequence $(v_n)_{n\in\IntegerP} \subset \mathcal{X}$ and a strongly positive
    operator $\Theta: \mathcal{X}^2 \to \mathcal{X}^2$ s.t.\ sequence
    $(y_n:= (x_n,v_n))_{n\in\IntegerPP\setminus \Set{1}}$ is Fej\'{e}r monotone~\cite[Def.~5.1,
    p.~75]{Bauschke.Combettes.book} w.r.t.\ $\Upsilon_*^{(\lambda)}$ of \cref{prop:O*} in the Hilbert
    space $(\mathcal{X}^2, \innerp{\cdot}{\cdot}_{\Theta})$, \ie,
    \begin{align*}
      \norm{(x_{n+1}, v_{n+1}) - (x_*, v_*)}_{\Theta} \leq
      \norm{(x_{n}, v_{n}) - (x_*, v_*)}_{\Theta}\,, \quad
      \forall (x_*,v_*)\in \Upsilon_*^{(\lambda)}\,.
    \end{align*}

  \item Sequence $(x_n)_{n\in\IntegerP}$ of \eqref{FM-HSDM} converges weakly to a point that solves
    $\text{VIP}(\nabla f + \partial g, \Fix T)$.

  \end{thmlist}
\end{theorem}

\begin{proof}
  \textbf{(i)} By \eqref{FM-HSDM.n+one+half},
  \begin{align}
    x_{n+3/2} -
    x_{n+1/2} = Tx_{n+1} - T_{\alpha}x_n
    - \lambda \left[\nabla f(x_{n+1}) - \nabla f(x_{n})
    \right] \,.  \label{diff.halves}
  \end{align}
  Since
  $z = \prox_{\lambda g}(y) \Leftrightarrow (\exists \xi\in\partial g(z)\ \text{s.t.}\ z + \lambda
  \xi = y)$, then
  \begin{align}
    \exists\xi_{n+2}\in \partial g(x_{n+2}) \label{existence.xi}
  \end{align}    
  s.t.\ $x_{n+3/2} = x_{n+2} + \lambda \xi_{n+2}$ and thus $\exists\xi_{n+1}\in \partial g(x_{n+1})$
  s.t.\ $x_{n+1/2} = x_{n+1} + \lambda \xi_{n+1}$. Incorporating the previous equations in
  \eqref{diff.halves} yields that $\forall n\in\IntegerP$,
  \begin{align}
    x_1
    & = T_{\alpha}x_0 - \lambda
      \left[\nabla f(x_0) + \xi_1\right]\,, \notag\\
    x_{n+2} - x_{n+1}
    & = Tx_{n+1} - T_{\alpha}x_n - \lambda \left[\nabla
      f(x_{n+1}) + \xi_{n+2} \right] + \lambda \left[\nabla
      f(x_{n}) + \xi_{n+1} \right] \,. \label{diff.consec.x}
  \end{align}
  Moreover, adding consecutive equations of \eqref{diff.consec.x} results into the following fact:
  \begin{alignat*}{2}
    x_{n+1} 
    & = &&\ Tx_n - \sum\nolimits_{\nu=1}^{n-1}
    (T_{\alpha}-T)x_{\nu} -
    \lambda \left[\nabla f(x_n) + \xi_{n+1}\right] \\
    & = &&\ Tx_n - \sum\nolimits_{\nu=1}^{n+1}
    (T_{\alpha}-T)x_{\nu} + (T_{\alpha}-T)x_{n} +
    (T_{\alpha}-T)x_{n+1} - \lambda \left[\nabla f(x_n)
      + \xi_{n+1}\right] \\
    & = &&\ 2T_{\alpha} x_{n+1} - Tx_{n+1} +
    (T_{\alpha}x_n - T_{\alpha}x_{n+1}) -
    \sum\nolimits_{\nu=1}^{n+1} (T_{\alpha}-T)x_{\nu} \\
    &&& - \lambda \left[\nabla f(x_n) + \xi_{n+1}\right] \,,
  \end{alignat*}
  where the last equality holds true $\forall n\in\IntegerP$. Consequently,
  \begin{align}
    (\Id+T-2T_{\alpha})x_{n+1}
    & + (T_{\alpha}x_{n+1} - T_{\alpha}x_n) \notag\\
    & = (1-2\alpha)(T-\Id) x_{n+1} +
      Q_{\alpha} (x_{n+1} - x_n) \notag\\
    & = - \sum\nolimits_{\nu=1}^{n+1} (T_{\alpha}-T)x_{\nu} -
      \lambda \left[\nabla f(x_n) +
      \xi_{n+1}\right]\,, \label{recursion.n.n+1}
  \end{align}
  where the first equation is due to \eqref{def.Talpha}.
  
  Choose arbitrarily a $w_*\in\Fix T$, \ie, $(\Id-T)w_* = 0$. Then,
  \begin{align*}
    (T_{\alpha} - T)x_\nu
    & = (1-\alpha) (\Id-T)x_{\nu}\\
    & = (1-\alpha)\left[(\Id-T)x_{\nu}
      - (\Id-T)w_* \right]\\
    & = (1-\alpha) (\Id-Q) (x_{\nu} - w_*)\,.
  \end{align*}
  Define also
  \begin{align*}
    v_{n+1} := (1-\alpha)
    \sum\nolimits_{\nu=1}^{n+1} U(x_{\nu} - w_*)\,.
  \end{align*}
  Point $v_{n+1}$ does not depend on the choice of the fixed point $w_*$. Indeed, by
  \cref{prop:FixT}, it can be verified that for any $w_{\#} \in\Fix T$, $w_{\#} - w_* \in \ker U$,
  and that
  \begin{align}
    v_{n+1}
    & = (1-\alpha)
      \sum\nolimits_{\nu=1}^{n+1} U(x_{\nu} - w_{\#}
      + w_{\#} - w_*) \notag\\
    & = (1-\alpha) \sum\nolimits_{\nu=1}^{n+1}
      \left[U(x_{\nu} - w_{\#}) +
      U(w_{\#} - w_*)\right] \notag\\
    & = (1-\alpha) \sum\nolimits_{\nu=1}^{n+1}
      U(x_{\nu} - w_{\#})\,. \label{define.v}
  \end{align}
  Moreover,
  \begin{align}
    v_{n+1} - v_{n}
    & = (1-\alpha)
      \sum\nolimits_{\nu=1}^{n+1}U(x_{\nu} -
      w_*) - (1-\alpha)
      \sum\nolimits_{\nu=1}^{n} U(x_{\nu} - w_*) \notag\\
    & = (1-\alpha) U(x_{n+1} - w_*)\,, \quad\forall w_*\in\Fix
      T\,, \label{consecutive.v}
  \end{align} and
  \begin{align}
    - \sum\nolimits_{\nu=1}^{n+1}
    (T_{\alpha}-T)x_{\nu}
    & = -(1-\alpha)
      \sum\nolimits_{\nu=1}^{n+1} (\Id-Q) (x_{\nu} - 
      w_*) \notag\\
    & = - U (1-\alpha)
      \sum\nolimits_{\nu=1}^{n+1} U (x_{\nu} -
      w_*)\notag\\
    & = -Uv_{n+1} \,. \label{sum.T.and.v}
  \end{align}
  Under the previous considerations, \eqref{recursion.n.n+1} becomes
  \begin{align}
    (1-2\alpha)(T-\Id)x_{n+1} +
    Q_{\alpha} (x_{n+1} - x_n)
    + Uv_{n+1} = - \lambda \left[\nabla f(x_n) +
    \xi_{n+1}\right] \,. \label{basic.recursion.n.n+1}
  \end{align}

  Recall now \cref{prop:O*}, and consider \textit{any}\/ $(x_*, v_*)\in \Upsilon_*^{(\lambda)}$. By
  the definition of $\Upsilon_*^{(\lambda)}$, $(\Id- T)x_* = 0$ and there exists
  $\xi_*\in \partial g(x_*)$ s.t.\ $Uv_* + \lambda [\nabla f(x_*) + \xi_*] = 0$. These arguments,
  \eqref{basic.recursion.n.n+1} and $(T-\Id)x_{n+1} - (T-\Id)x_* = (Q-\Id)(x_{n+1}-x_*)$ yield 
  \begin{alignat}{2}
    % & -\vect{U}\vect{v}_{n+1} - \lambda
    % \left[\nabla f(\vect{x}_n) + \bm{\xi}_{n+1}\right] + \vect{Uv}_*
    % + \lambda (\nabla f(\vect{x}_*) + \bm{\xi}_*)
    % && \notag\\
    % &&& \hspace{-250pt} = (1-2\alpha)(T-\Id)\vect{x}_{n+1} -
    % (1-2\alpha)(T-\Id)\vect{x}_* + \vect{Q}_{\alpha} (\vect{x}_{n+1}
    % - \vect{x}_n) \notag\\
    & \lambda [\nabla f(x_n) - \nabla f(x_*)] + \lambda (\xi_{n+1}
    - \xi_*) && \notag\\
    &&& \hspace{-150pt}= -(1-2\alpha)(Q-\Id)(x_{n+1} - x_*)
    -Q_{\alpha} (x_{n+1} - x_n) -U(v_{n+1} - v_*)
    \,. \label{describe.grads}
  \end{alignat}
  The Baillon-Haddad theorem~\cite{Baillon.Haddad}, \cite[Cor.~18.16,
  p.~270]{Bauschke.Combettes.book} states that the $L$-Lipschitz continuous $\nabla f$ is
  $(1/L)$-inverse strongly monotone, \ie, $\forall(x, x') \in \mathcal{X}^2$,
  $\innerp{x - x'}{\nabla f(x) - \nabla f(x')} \geq (1/L) \norm{\nabla f(x) - \nabla f(x')}^2$. This
  property, the fact that $\partial g$ is monotone~\cite[Example~20.3,
  p.~294]{Bauschke.Combettes.book}, \ie, $\forall x, x', \xi, \xi'$ s.t.\ $\xi\in\partial g(x)$ and
  $\xi'\in\partial g(x')$, $\innerp{x-x'}{\xi - \xi'} \geq 0$, and the fact that $U$ is self adjoint
  imply
  \begin{subequations}\label{Baillon.Haddad}
    \begin{align}
      \tfrac{2\lambda}{L}
      & \norm{\nabla
        f(x_n) - \nabla f(x_*)}^2 \notag\\
      & \mathbin{\leq} 2\lambda
        \innerp{x_n - x_*}{\nabla f(x_n) - \nabla
        f(x_*)} \notag\\
      & \mathbin{\leq} 2\lambda
        \innerp{x_{n+1} - x_*}{\nabla f(x_n) -
        \nabla f(x_*)} + 2\lambda \innerp{x_{n} -
        x_{n+1}}{\nabla f(x_n) - \nabla
        f(x_*)}\notag\\
      & \mathbin{\hphantom{\leq}} + 2\lambda
        \innerp{x_{n+1} - 
        x_*}{\xi_{n+1} - \xi_*} \notag\\
      & \mathbin{=}
        2\innerp{x_{n+1} - x_*}{\lambda [\nabla
        f(x_n) - \nabla f(x_*)] + \lambda (\xi_{n+1} -
        \xi_*)} \notag\\
      & \mathbin{\hphantom{\leq}} + 2\lambda \innerp{x_{n} -
        x_{n+1}}{\nabla f(x_n) - \nabla f(x_*)}
        \notag\\
      & \mathbin{=} -2(1-2\alpha) \innerp{x_{n+1} -
        x_*}{(Q-\Id)(x_{n+1} - x_*)} -
        2\innerp{x_{n+1} - x_*}{Q_{\alpha}
        (x_{n+1} - x_n)} \notag\\
      & \mathbin{\hphantom{\leq}} -
        2\innerp{x_{n+1} - x_*}{U(v_{n+1} -
        v_*)} + 2\lambda \innerp{x_{n} -
        x_{n+1}}{\nabla f(x_n) - \nabla
        f(x_*)} \label{replace.grads}\\
      & \mathbin{=} -2(1-2\alpha)
        \innerp{x_{n+1} - x_*}{(Q-\Id)(x_{n+1} - x_*)} -
        2\innerp{x_{n+1} - x_*}{Q_{\alpha}
        (x_{n+1} - x_n)}\notag\\
      & \mathbin{\hphantom{\leq}} -
        2\innerp{U(x_{n+1} - x_*)}{v_{n+1} -
        v_*} + 2\lambda \innerp{x_{n} -
        x_{n+1}}{\nabla f(x_n) - \nabla
        f(x_*)}\notag\\
      & \mathbin{\leq} -2(1-2\alpha)
        \innerp{x_{n+1} - x_*}{(Q-\Id)(x_{n+1} - x_*)} -
        2\innerp{x_{n+1} - x_*}{Q_{\alpha}
        (x_{n+1} - x_n)} \notag\\
      & \mathbin{\hphantom{\leq}} - \tfrac{2}{1-\alpha}
        \innerp{v_{n+1} - v_n}{v_{n+1} - v_*}
        + \tfrac{\lambda L}{2}\norm{x_n - x_{n+1}}^2
        \notag\\
      & \mathbin{\hphantom{\leq}} + \tfrac{2\lambda}{L}
        \norm{\nabla f(x_n) - \nabla f(x_*)}^2
        \,,  \label{apply.young.ineq} 
    \end{align}
  \end{subequations}
  where \eqref{describe.grads} is used in \eqref{replace.grads}, and \eqref{consecutive.v} as well
  as
  \begin{align}
    2\innerp*{\tfrac{a}{\sqrt{\eta}}}{\sqrt{\eta}\,
    b}_{\Pi} \leq \tfrac{1}{\eta}
    \norm{a}^2_{\Pi} + \eta
    \norm{b}^2_{\Pi}\,, \quad 
    \left\{
    \begin{aligned}
      &\forall (a, b)\in \mathcal{X}^2\,,
      \forall\eta\in\RealPP\,, \\
      & \forall\ \text{strongly positive}\
      \Pi\in\mathfrak{B}(\mathcal{X})\,,  
      \end{aligned}
        \right. \label{Young.ineq}
  \end{align}
  with $\eta := 2/L$, $a:= x_n - x_{n+1}$, $b:= \nabla f(x_n) - \nabla f(x_*)$, $\Pi := \Id$, were
  used in \eqref{apply.young.ineq}.

  Recall \eqref{def.Talpha} to verify that the positivity of $Q$ implies that for any
  $x\in\mathcal{X}$,
  \begin{align}
    \innerp{Q_{\alpha}x}{x} = \alpha \innerp{Qx}{x} +
    (1-\alpha)\norm{x}^2 \geq (1-\alpha)\norm{x}^2\,,
    \label{Qa.is.PD}
  \end{align}
  \ie, $Q_{\alpha}$ is strongly positive.  Hence, upon defining the linear mapping
  $\Theta: \mathcal{X}^2 \to \mathcal{X}^2: (x,v) \mapsto (Q_{\alpha}x, v/(1-\alpha))$, it can be
  easily seen that $\Theta$ is strongly positive, under the standard inner product
  $\innerp{(x,v)}{(x',v')} := \innerp{x}{x'} + \innerp{v}{v'}$,
  $\forall (x,v), (x',v')\in \mathcal{X}^2$, due to the fact that both $Q_{\alpha}$ and
  $\Id/(1-\alpha)$ are strongly positive. Consequently,
  $(\mathcal{X}^2, \innerp{\cdot}{\cdot}_{\Theta})$ can be considered to be a Hilbert space equipped
  with the inner product $\innerp{\cdot}{\cdot}_{\Theta}$.

  Notation $y:=(x,v)$, $\alpha\geq 1/2$ as well as the positivity of $\Id-Q$ in
  \eqref{Baillon.Haddad} yield
  \begin{alignat*}{2}
    0
    & \,\mathbin{\leq}\, &&
    2 \innerp{(x_{n+1} - x_n, v_{n+1} - v_n)}{\Theta(x_* -
      x_{n+1}, v_* - v_{n+1})} \\
    &&& - 2(2\alpha-1) \innerp{x_{n+1} - x_*}{(\Id
      - Q)(x_{n+1} - x_*)} \\
    &&& + \tfrac{\lambda L}{2}\norm{x_n - x_{n+1}}^2 \\
    & \,\mathbin{=}\, && 2 \innerp{y_{n+1} - y_{n}}{\Theta (y_*-
      y_{n+1})} - 2(2\alpha-1) \innerp{x_{n+1} - x_*}{(\Id -
      Q)(x_{n+1} - x_*)} \\
    &&& + \tfrac{\lambda
      L}{2}\norm{x_n - x_{n+1}}^2 \\
    & \,\mathbin{\leq}\, && 2 \innerp{y_{n+1} - y_{n}}{y_*-
      y_{n+1}}_{\Theta} + \tfrac{\lambda
      L}{2}\norm{x_n - x_{n+1}}^2\\
    & \,\mathbin{=}\, && \norm{y_n - y_*}_{\Theta}^2 - \norm{y_{n+1} -
      y_*}_{\Theta}^2 - \norm{y_{n+1} - y_n}_{\Theta}^2 + \tfrac{\lambda
      L}{2}\norm{x_n - x_{n+1}}^2\,.
  \end{alignat*}
  Hence,
  \begin{align}
    \norm{y_n - y_*}_{\Theta}^2 -
    \norm{y_{n+1} - y_*}_{\Theta}^2
    & \geq \norm{y_{n+1} - y_n}_{\Theta}^2 -
      \tfrac{\lambda L}{2}\norm{x_n -
      x_{n+1}}^2 \,.\label{pre.Fejer}
    % & = \begin{bmatrix} \vect{x}_{n+1} - \vect{x}_n\\
    %   \vect{v}_{n+1} - \vect{v}_n
    % \end{bmatrix}^{\top}
    % \underbrace{\begin{bmatrix}
    %     \vect{Q}_{\alpha} - \tfrac{\lambda L}{2}\vect{I}
    %     & \vect{0} \\
    %     \vect{0}
    %     & \tfrac{1}{1-\alpha}\vect{I}
    %   \end{bmatrix}}_{=:\,\bm{\Pi}_{\alpha,
    %       \lambda}}
    %       \begin{bmatrix} \vect{x}_{n+1} - \vect{x}_n\\
    %         \vect{v}_{n+1} - \vect{v}_n 
    %       \end{bmatrix} \notag \\
    % & = \innerp{\vect{y}_{n+1} -
    %   \vect{y}_n}{\bm{\Pi}_{\alpha, \lambda}(\vect{y}_{n+1} -
    %   \vect{y}_n)} \,.\notag
  \end{align}
  Since $\lambda < 2(1-\alpha)/L$, choose any $\zeta\in (\lambda L/[2(1-\alpha)],1)$. Then, by
  \eqref{Qa.is.PD}, $\forall y := (x, v)$,
  \begin{align*}
    \tfrac{\lambda L}{2} \norm{x}^2
     <
      \zeta(1-\alpha) \norm{x}^2 \leq
      \zeta \innerp{x}{Q_{\alpha}x}
     \leq
      \zeta \innerp{x}{Q_{\alpha}x} +
      \zeta\tfrac{1}{1-\alpha}\norm{v}^2 = \zeta
      \norm{y}^2_{\Theta} \,,
  \end{align*}
  and by \eqref{pre.Fejer},
  \begin{align}
    \norm{y_n - y_*}_{\Theta}^2 -
    \norm{y_{n+1} - y_*}_{\Theta}^2
    & \geq \norm{y_{n+1} - y_n}_{\Theta}^2 -
      \tfrac{\lambda L}{2}\norm{x_n - x_{n+1}}^2 \notag\\ 
    & \geq \norm{y_{n+1} - y_n}_{\Theta}^2 - \zeta
      \norm{y_{n+1} - y_n}_{\Theta}^2 \notag\\
    & = (1-\zeta) \norm{y_{n+1} - y_n}_{\Theta}^2 \,, \label{Fejer}
  \end{align}
  \ie, sequence $(y_n)_{n\in\IntegerP} \subset (\mathcal{X}^2, \innerp{\cdot}{\cdot}_{\Theta})$ is
  Fej\'{e}r monotone w.r.t.\ $\Upsilon_*^{(\lambda)}$ of \cref{prop:O*}.

  \textbf{(ii)} Due to Fej\'{e}r monotonicity, sequence $(y_n)_n$ is bounded [as well as $(x_n)_n$
  and $(v_n)_n$]~\cite[Prop.~5.4(i), p.~76]{Bauschke.Combettes.book} and possesses a non-empty set
  of weakly sequential cluster points $\mathfrak{W}[(y_n)_n]$~\cite[Lem.~2.37,
  p.~36]{Bauschke.Combettes.book}. Moreover, it can be verified by \eqref{Fejer}, that
  $\forall n\in\IntegerP$,
  \begin{align*}
    (1-\zeta) \sum\nolimits_{\nu=2}^n
    \norm{y_{\nu+1} - y_{\nu}}_{\Theta}^2 \leq
    \norm{y_{2} - y_*}_{\Theta}^2 -
    \norm{y_{n+1} - y_*}_{\Theta}^2 \leq
    \norm{y_{2} - y_*}_{\Theta}^2\,,
  \end{align*}
  and hence there exist $C', C\in \RealPP$ s.t.\ for any $n$,
  \begin{align}
    \sum\nolimits_{\nu=0}^n \norm{y_{\nu+1} -
    y_{\nu}}_{\Theta}^2 \leq \tfrac{C'}{1-\zeta} =:
    C\,, \label{y.as.regular}
  \end{align}
  which leads to $\lim_{n\to\infty} \norm{y_{n+1} - y_n}_{\Theta} = 0$, and which further implies
  that
  \begin{align}
    \lim\nolimits_{n\to\infty} (x_{n+1} -
    x_n) = 0\,, \quad \lim\nolimits_{n\to\infty}
    (v_{n+1} - v_n) = 0\,. \label{as.regularity.x.v}
  \end{align}

  Adding the following equations, which result from \eqref{basic.recursion.n.n+1},
  \begin{align}
    -\tfrac{1}{\lambda}(1-2\alpha)(T-\Id) x_{n+1}
    -\tfrac{1}{\lambda} Q_{\alpha} (x_{n+1} - x_n)
    - \tfrac{1}{\lambda} Uv_{n+1}
    - \nabla f(x_n)
    & = \xi_{n+1} \label{xi.n+1}\\
    \tfrac{1}{\lambda} (1-2\alpha)(T-\Id)x_{n}
    + \tfrac{1}{\lambda} Q_{\alpha} (x_{n}- x_{n-1})
    + \tfrac{1}{\lambda} Uv_n
    + \nabla
    f(x_{n-1})
    & = -\xi_{n} \notag
  \end{align}
  yields
  \begin{alignat}{2}
    \xi_{n+1} - \xi_n & = &&\
    \tfrac{1-2\alpha}{\lambda}(T-\Id)(x_{n} - x_{n+1}) +
    \tfrac{1}{\lambda} Q_{\alpha} (x_{n} - x_{n-1}) -
    \tfrac{1}{\lambda} Q_{\alpha}
    (x_{n+1} - x_{n}) \notag\\
    &&&\ + \tfrac{1}{\lambda} U(v_{n} - v_{n+1}) + [\nabla
    f(x_{n-1}) - \nabla f(x_{n})] \,. \label{diff.consecutive.xi}
  \end{alignat}
  By applying $\lim_{n\to\infty}$ to the previous equality, and by using the Lipschitz continuity of
  $\nabla f$, \ie, $\norm{\nabla f(x_n) - \nabla f(x_{n-1})} \leq L \norm{x_n- x_{n-1}}$,
  \eqref{as.regularity.x.v}, as well as the continuity of $\Id-T$, $Q_{\alpha}$ and $U$, it can be
  verified that
  \begin{align}
    \lim\nolimits_{n\to\infty} (\xi_{n+1} -
    \xi_n) = 0 \,. \label{as.regularity.xi}
  \end{align}

  Now, by \eqref{diff.consec.x},
  \begin{align*}
    x_{n+2} 
    & - x_{n+1} \\
    & = Tx_{n+1} -
      T_{\alpha} x_{n+1} + T_{\alpha} x_{n+1} -
      T_{\alpha}x_n - \lambda [\nabla f(x_{n+1}) - \nabla
      f(x_{n})] - \lambda [\xi_{n+2} - \xi_{n+1}] \\
    & = (T-T_{\alpha}) x_{n+1} + Q_{\alpha}
      (x_{n+1} - x_n) - \lambda [\nabla f(x_{n+1})
      - \nabla f(x_{n})] - \lambda [\xi_{n+2} - \xi_{n+1}] \,,
  \end{align*}
  which leads to
  \begin{align}
    (1-\alpha)(\Id-T)x_n
    \,\mathbin{=}\, 
    & (x_{n} - x_{n+1}) + Q_{\alpha} (x_{n} -
      x_{n-1}) \notag\\
    &\, - \lambda [\nabla f(x_{n}) - \nabla
      f(x_{n-1})] - \lambda [\xi_{n+1} - \xi_{n}]
      \,. \label{I-T.seq}
  \end{align}

  Choose any
  $\overline{y} := (\overline{x}, \overline{v})\in \mathfrak{W}[(y_n)_{n\in\IntegerP}]\neq \emptyset$,
  \ie, there exists a subsequence $(y_{n_k} := (x_{n_k}, v_{n_k}))_k$ s.t.\
  $x_{n_k} \rightharpoonup_{k\to\infty} \overline{x}$ and
  $v_{n_k}\rightharpoonup_{k\to\infty} \overline{v}$. Furthermore, by \eqref{as.regularity.x.v},
  \eqref{as.regularity.xi}, \eqref{I-T.seq}, and the Lipschitz continuity of $\nabla f$,
  \begin{alignat}{2}
    \lim\sup\nolimits_{n\to\infty} \norm*{(\Id-T) x_{n}}
    & \mathbin{\leq} && \tfrac{1}{1-\alpha}
    \lim\nolimits_{k\to\infty} 
    \norm*{x_{n} - x_{n+1}} +
    \lim\nolimits_{k\to\infty} \tfrac{1}{1-\alpha}
    \norm{Q_{\alpha}(x_{n} - x_{n-1})} \notag \\
    &&& + \tfrac{\lambda}{1-\alpha} \lim\nolimits_{k\to\infty} 
    \norm*{\nabla f(x_{n}) - \nabla f(x_{n-1})} \notag\\
    &&& +
    \tfrac{\lambda}{1-\alpha} \lim\nolimits_{k\to\infty}
    \norm*{\xi_{n+1} - \xi_{n}} \notag\\
    & \mathbin{\leq} && \tfrac{1}{1-\alpha}
    \lim\nolimits_{k\to\infty} 
    \norm*{x_{n} - x_{n+1}} +
    \lim\nolimits_{k\to\infty}
    \tfrac{\norm{Q_{\alpha}}}{1-\alpha} 
    \norm{x_{n} - x_{n-1}} \notag\\
    &&& + \tfrac{\lambda L}{1-\alpha} \lim\nolimits_{k\to\infty}
    \norm*{x_{n} - x_{n-1}} +
    \tfrac{\lambda}{1-\alpha} \lim\nolimits_{k\to\infty}
    \norm*{\xi_{n+1} - \xi_{n}} \notag\\
    & \mathbin{=} &&\, 0\,. \label{(I-T)xn.conv.0}
  \end{alignat}
  Hence, due to $x_{n_k} \rightharpoonup_{k\to\infty} \overline{x}$,
  $\lim_{k\to\infty} (\Id-T) x_{n_k}=0$, and the demiclosedness property of the nonexpansive mapping
  $T$~\cite[Thm.~4.17, p.~63]{Bauschke.Combettes.book}, it follows that
  \begin{align}
    \overline{x}\in \Fix T \,. \label{O*.1st.constraint}
  \end{align}

  Fix arbitrarily an $x_{\#}\in\mathcal{X}$. Since $(x_n)_n$ is bounded, there exist
  $C'', C_{\nabla f} \in \RealPP$ s.t.\ for any $n$,
  \begin{align}
    \norm{\nabla f(x_n)} 
    & \leq \norm{\nabla f(x_n) - \nabla
    f(x_{\#})} + \norm{\nabla f(x_{\#})} \notag\\
    & \leq L\norm{x_n-x_{\#}} + \norm{\nabla f(x_{\#})} \notag\\
    & \leq L(\norm{x_n} + \norm{x_{\#}}) + \norm{\nabla
      f(x_{\#})} \notag\\
    & \leq L(C'' + \norm{x_{\#}}) + \norm{\nabla
      f(x_{\#})} \leq C_{\nabla f} \,. \label{bounded.grad.f}
  \end{align}

  Now, according to the Baillon-Haddad theorem~\cite{Baillon.Haddad},
  \cite[Cor.~18.16, p.~270]{Bauschke.Combettes.book},
  \begin{subequations}\label{BH.for.strong.conv.gradient}
    \begin{align}
      \tfrac{2\lambda}{L}
      & \norm*{\nabla f(x_{n_k}) - \nabla
        f(\overline{x})}^2 \notag\\
      & \mathbin{\leq} 2\lambda \innerp*{x_{n_k} -
        \overline{x}}{\nabla f(x_{n_k}) - \nabla
        f(\overline{x})} \notag\\
      & \mathbin{=} 2\lambda \innerp*{x_{n_k+1} -
        \overline{x}}{\nabla f(x_{n_k})} \notag\\
      & \mathbin{\hphantom{\leq}} - 2\lambda \innerp*{x_{n_k+1} -
        \overline{x}}{\nabla f(\overline{x})}
        + 2\lambda \innerp*{x_{n_k} -
        x_{n_k+1}}{\nabla f(x_{n_k}) - \nabla f(\overline{x})}
        \notag\\
      & \mathbin{=} - 2\lambda \innerp*{x_{n_k+1} -
        \overline{x}}{\xi_{n_k+1}} - 2\innerp*{x_{n_k+1} -
        \overline{x}}{Uv_{n_k+1}} \notag\\
      & \mathbin{\hphantom{\leq}} - 2 \innerp*{x_{n_k+1} -
        \overline{x}}{Q_{\alpha}(x_{n_k+1} - x_{n_k})}
        -(1-2\alpha) \innerp*{x_{n_k+1} -
        \overline{x}}{(T-\Id)x_{n_k+1}} \notag\\
      & \mathbin{\hphantom{\leq}} - 2\lambda \innerp*{x_{n_k+1} -
        \overline{x}}{\nabla f(\overline{x})}
        + 2\lambda \innerp*{x_{n_k} -
        x_{n_k+1}}{\nabla f(x_{n_k}) - \nabla f(\overline{x})}
        \label{use.grad.f} \\
      & \mathbin{\leq} 2\lambda \left[g(\overline{x}) -
        g(x_{n_k+1}) \right] - 2\innerp*{U(x_{n_k+1} -
        \overline{x})}{v_{n_k+1}} \notag\\
      & \mathbin{\hphantom{\leq}} - 2 \innerp*{x_{n_k+1} -
        \overline{x}}{Q_{\alpha}(x_{n_k+1} - x_{n_k})}
        -(1-2\alpha) \innerp*{x_{n_k+1} -
        \overline{x}}{(T-\Id)x_{n_k+1}} \notag\\
      & \mathbin{\hphantom{\leq}} - 2\lambda \innerp*{x_{n_k+1} -
        \overline{x}}{\nabla f(\overline{x})}
        + 2\lambda \innerp*{x_{n_k} -
        x_{n_k+1}}{\nabla f(x_{n_k}) - \nabla f(\overline{x})}
      \label{lsc.U.is.adjoint}\\ 
      & \mathbin{\leq} 2\lambda \left[g(\overline{x}) -
        g(x_{n_k+1}) \right] - \tfrac{2}{1-\alpha} \innerp*{v_{n_k+1} -
        v_{n_k}}{v_{n_k+1}} \notag\\
      & \mathbin{\hphantom{\leq}} - 2 \innerp*{x_{n_k+1} -
        \overline{x}}{Q_{\alpha}(x_{n_k+1} - x_{n_k})}
        -(1-2\alpha) \innerp*{x_{n_k+1} -
        \overline{x}}{(T-\Id)x_{n_k+1}} \notag\\
      & \mathbin{\hphantom{\leq}} - 2\lambda \innerp*{x_{n_k+1} -
        \overline{x}}{\nabla f(\overline{x})}
        + 2\lambda \left(
        C_{\nabla f} + \norm{\nabla f(\overline{x})} \right)
        \norm{x_{n_k} - x_{n_k+1}}\,,
        \label{cons.v.bounded.grad.f}
    \end{align}
  \end{subequations}
  where \eqref{basic.recursion.n.n+1} is used in \eqref{use.grad.f}, the convexity of $g$,
  \eqref{existence.xi} and the self adjointness of $U$ in \eqref{lsc.U.is.adjoint}, and finally
  \eqref{consecutive.v} and \eqref{bounded.grad.f} in \eqref{cons.v.bounded.grad.f}. Since
  $\lim_{k\to\infty} (x_{n_k} - x_{n_k+1}) =0$ by \eqref{as.regularity.x.v}, the continuity of
  $Q_{\alpha}$ implies $\lim_{k\to\infty} Q_{\alpha}(x_{n_k+1} - x_{n_k})=0$, and
  \eqref{(I-T)xn.conv.0} yields $\lim_{k\to\infty}(T-\Id)x_{n_k+1}=0$. Notice again by
  \eqref{as.regularity.x.v} that $\lim_{k\to\infty} (v_{n_k+1} - v_{n_k}) = 0$. Further,
  \eqref{as.regularity.x.v}, together with
  $(x_{n_k} - \overline{x}) \rightharpoonup_{k\to\infty} 0$, yields
  $(x_{n_k+1} - \overline{x}) \rightharpoonup_{k\to\infty} 0$. Similarly,
  $(v_{n_k+1} -\overline{v}) \rightharpoonup_{k\to\infty} 0$ can be deduced from
  \eqref{as.regularity.x.v} and $(v_{n_k} - \overline{v}) \rightharpoonup_{k\to\infty} 0$. Due to
  \cite[Lem.~2.41(iii), p.~37]{Bauschke.Combettes.book}, all of the previous arguments result in
  $\lim_{k\to\infty} \innerp{v_{n_k+1} - v_{n_k}}{v_{n_k+1}}=0$,
  $\lim_{k\to\infty} \innerp{x_{n_k+1} - \overline{x}}{Q_{\alpha}(x_{n_k+1} - x_{n_k})} = 0$,
  $\lim_{k\to\infty} \innerp*{x_{n_k+1} - \overline{x}}{(T-\Id)x_{n_k+1}}=0$,
  $\lim_{k\to\infty} \innerp*{x_{n_k+1} - \overline{x}}{\nabla f(\overline{x})}=0$, and
  $\lim_{k\to\infty} \norm{x_{n_k} - x_{n_k+1}} =0$. Hence, the application of
  $\lim\sup_{k\to\infty}$ onto both sides of \eqref{cons.v.bounded.grad.f} yields
  \begin{align*}
    \lim\sup_{k\to\infty} \norm*{\nabla f(x_{n_k}) - \nabla
    f(\overline{x})}^2
    & \leq \lim\sup_{k\to\infty} L \left[g(\overline{x}) -
      g(x_{n_k+1}) \right] \\
    & = L \left[g(\overline{x}) - \lim\inf_{k\to\infty}
      g(x_{n_k+1})\right] \leq 0\,,
  \end{align*}
  where the last inequality is deduced from the fact that $g\in\Gamma_0(\mathcal{X})$ turns out to
  be also weakly sequentially lower semicontinuous~\cite[Thm.~9.1,
  p.~129]{Bauschke.Combettes.book}. In other words,
  \begin{align}
    \lim\nolimits_{k\to\infty} \nabla f(x_{n_k}) = \nabla
    f(\overline{x}) \,. \label{grad.f.conv.strongly}
  \end{align}

  Since $v_{n_k+1}\rightharpoonup_{k} \overline{v}$, \ie, $\forall z\in\mathcal{X}$,
  $\lim_{k\to\infty} \innerp{z}{v_{n_k+1}} = \innerp{z}{\overline{v}}$, it can be easily seen that
  $\forall z\in\mathcal{X}$,
  $\lim_{k\to\infty} \innerp{z}{Uv_{n_k+1}} = \lim_{k\to\infty} \innerp{Uz}{v_{n_k+1}} =
  \innerp{Uz}{\overline{v}} = \innerp{z}{U \overline{v}}$, \ie,
  $Uv_{n_k+1}\rightharpoonup_{k} U\overline{v}$. Hence, having this result and
  \eqref{grad.f.conv.strongly} plugged into \eqref{xi.n+1} yields that
  \begin{align}
    \xi_{n_k+1}\rightharpoonup_{k\to\infty}
    \overline{\xi}:= -\tfrac{1}{\lambda}U\overline{v} - \nabla
    f(\overline{x})\,. \label{def.xi.overline}
  \end{align} 
  Using \eqref{basic.recursion.n.n+1} once again,
  \begin{alignat}{2}
    \innerp{x_{n_k+1} - \overline{x}}{\xi_{n_k+1}}
    & = && -\innerp{x_{n_k+1} - \overline{x}}{\nabla f(x_{n_k})}
    - \tfrac{1}{\lambda} \innerp{x_{n_k+1} -
      \overline{x}}{Uv_{n_k+1}} \notag \\ 
    &&& -\tfrac{1}{\lambda} \innerp{x_{n_k+1} -
      \overline{x}}{Q_{\alpha}(x_{n_k+1} - x_{n_k})} \notag \\
    &&& - \tfrac{1}{\lambda} (1-2\alpha) \innerp{x_{n_k+1} -
      \overline{x}}{(T-\Id)x_{n_k+1}} \notag \\
    & = && -\innerp{x_{n_k+1} - \overline{x}}{\nabla f(x_{n_k})}
    - \tfrac{1}{\lambda} \innerp{U(x_{n_k+1} -
      \overline{x})}{v_{n_k+1}} \notag \\
    &&& - \tfrac{1}{\lambda}\innerp{x_{n_k+1} -
      \overline{x}}{Q_{\alpha}(x_{n_k+1} - x_{n_k})} \notag \\
    &&& - \tfrac{1}{\lambda} (1-2\alpha) \innerp{x_{n_k+1} -
      \overline{x}}{(T-\Id)x_{n_k+1}} \notag \\
    & = && -\innerp{x_{n_k+1} - \overline{x}}{\nabla f(x_{n_k})}
    - \tfrac{1}{\lambda(1-\alpha)} \innerp{v_{n_k+1} -
      v_{n_k}}{v_{n_k+1}} \notag\\
    &&& -\tfrac{1}{\lambda} \innerp{x_{n_k+1} -
      \overline{x}}{Q_{\alpha}(x_{n_k+1} - x_{n_k})} \notag\\ 
    &&& - \tfrac{1}{\lambda} (1-2\alpha) \innerp{x_{n_k+1} -
      \overline{x}}{(T-\Id)x_{n_k+1}}\,, \label{pre.Fitzpatrick}
  \end{alignat}
  where \eqref{consecutive.v} is used in \eqref{pre.Fitzpatrick}. Since
  $(x_{n_k+1} - \overline{x}) \rightharpoonup_{k} 0$ and
  $v_{n_k+1}\rightharpoonup_{k} \overline{v}$, and due to \eqref{as.regularity.x.v},
  \eqref{(I-T)xn.conv.0} and \eqref{grad.f.conv.strongly}, as well as the continuity of the linear
  mapping $Q_{\alpha}$, it turns out by \cite[Lem.~2.41(iii), p.~37]{Bauschke.Combettes.book} and
  \eqref{pre.Fitzpatrick} that
  $\lim_{k\to\infty} \innerp{x_{n_k+1} - \overline{x}}{\xi_{n_k+1}} = 0$. In other words,
  \begin{align}
    \lim\nolimits_{k\to\infty} \innerp*{x_{n_k+1}}{\xi_{n_k+1}} 
    & = \lim\nolimits_{k\to\infty} \left(\innerp*{x_{n_k+1} -
    \overline{x}}{\xi_{n_k+1}} +
      \innerp*{\overline{x}}{\xi_{n_k+1}} \right) \notag\\
    & = \lim\nolimits_{k\to\infty} \innerp*{x_{n_k+1} -
      \overline{x}}{\xi_{n_k+1}} + \lim\nolimits_{k\to\infty}
      \innerp*{\overline{x}}{\xi_{n_k+1}} \notag\\
    & = \lim\nolimits_{k\to\infty}
      \innerp*{\overline{x}}{\xi_{n_k+1}} =
      \innerp*{\overline{x}}{\overline{\xi}}\,.
      \label{Fitzpatrick}
  \end{align}
  Now, by $(x_{n_k+1}, \xi_{n_k+1})\in \graph\partial g$, the maximal monotonicity of
  $\partial g$~\cite[Thm.~20.40, p.~304]{Bauschke.Combettes.book} and the property manifested in
  \eqref{Fitzpatrick}, \cite[Cor.~20.49(ii), p.~306]{Bauschke.Combettes.book} suggests that
  $(\overline{x}, \overline{\xi})\in \graph \partial g \Leftrightarrow \overline{\xi}\in \partial
  g(\overline{x})$. Hence, according also to \eqref{def.xi.overline},
  $-U (\overline{v}/\lambda)\in \nabla f(\overline{x}) + \partial g(\overline{x})$, which together
  with \eqref{O*.1st.constraint} imply $(\overline{x}, \overline{v}) \in
  \Upsilon_*^{(\lambda)}$. Since $(\overline{x}, \overline{v})$ was arbitrarily chosen within
  $\mathfrak{W}[(y_n)_n]$, it follows that $\mathfrak{W}[(y_n)_n] \subset
  \Upsilon_*^{(\lambda)}$. Adding also to that the Fej\'{e}r monotonicity property \eqref{Fejer} of
  $(y_n)_{n\in\IntegerP}$ w.r.t.\ $\Upsilon_*^{(\lambda)}$ yields that $(y_n)_n$ converges weakly to a
  point in $\Upsilon_*^{(\lambda)}$~\cite[Thm.~5.5, p.~76]{Bauschke.Combettes.book}. According to
  \eqref{O.lambda.means.Gamma*}, the weak limit of $(x_n)_{n}$ solves
  $\text{VIP}(\nabla f + \partial g, \Fix T)$.
\end{proof}

\begin{definition}[\protect{\cite[(10.2),
    p.~144]{Bauschke.Combettes.book}}]\label{def:uniform.convex} A proper convex function $h:
  \mathcal{X} \to (-\infty, +\infty]$ is called \textit{uniformly convex}\/ on a non-empty subset
  $\mathcal{S}$ of $\domain h$, if there exists an increasing function
  $\varphi_{\mathcal{S}} : [0,+\infty]\to [0,+\infty]$, which vanishes only at $0$, s.t.\
  $\forall x,x'\in \mathcal{S}$ and $\forall \mu\in (0,1)$,
  \begin{align*}
    h(\mu x + (1-\mu)x') + \mu(1-\mu)
    \varphi_{\mathcal{S}}(\norm{x-x'}) \leq \mu h(x) + (1-\mu)h(x')\,.
  \end{align*}
  In the case where $\mathcal{S}:= \domain h$ and
  $\varphi_{\mathcal{S}} := (\beta_{\mathcal{S}}/2)(\cdot)^2$, for some
  $\beta_{\mathcal{S}}\in \RealPP$, then $h$ is called \textit{strongly convex}\/ with constant
  $\beta_{\mathcal{S}}$. Moreover, ``strong convexity'' $\Rightarrow$ ``uniform convexity''
  $\Rightarrow$ ``strict convexity.''
\end{definition}

\begin{assumption}\mbox{}
  \begin{asslist}

  \item\label{as:unif.convx.f} Function $f$ is uniformly convex on every non-empty bounded subset of
    $\mathcal{X}$.

  \item\label{as:unif.convx.g} Function $g$ is uniformly convex on every non-empty bounded subset of
    $\domain\partial g$.

  \end{asslist}
\end{assumption}

\begin{lemma}\label{lem:unif.convex}
  In addition to the setting of \cref{thm:basic}, if either \cref{as:unif.convx.f} or
  \cref{as:unif.convx.g} holds true, then sequence $(x_n)_{n\in\IntegerP}$ of \eqref{FM-HSDM}
  converges strongly to a point that solves $\text{VIP}(\nabla f + \partial g, \Fix T)$.
\end{lemma}

\begin{proof}
  As part \textbf{(ii)} of the proof of \cref{thm:basic} has demonstrated, sequences $(x_n)_n$ and
  $(Uv_n)_n$ converge weakly to $\overline{x}$ and $U\overline{v}$, respectively. Consequently,
  \eqref{as.regularity.x.v}, the continuity of $Q_{\alpha}$, \eqref{xi.n+1}, \eqref{(I-T)xn.conv.0},
  \eqref{grad.f.conv.strongly} and \eqref{def.xi.overline} suggest that $(\xi_n)_n$ converges weakly
  to $\overline{\xi}$.

  Let \cref{as:unif.convx.f} hold true. Then, according to \cite[Ex.~22.3(iii),
  p.~324]{Bauschke.Combettes.book}, given a bounded set $\mathcal{B}\subset \mathcal{X}$, there
  exists an increasing function $\varphi_{\mathcal{B}}: [0,+\infty)\to [0,+\infty]$, which vanishes
  only at $0$, s.t.\ $\forall x,x'\in\mathcal{B}$,
  \begin{align}
    \innerp*{x-x'}{\nabla f(x) - \nabla f(x')} \geq
    2\varphi_{\mathcal{B}}
    \left(\norm{x-x'}\right)\,. \label{uniformly.monotone.f} 
  \end{align}
  Define $\mathcal{B} := (x_n)_n \cup \Set{\overline{x}}$ (recall that $(x_n)_n$ is bounded). Set
  $x:= x_n$ and $x' := \overline{x}$ in \eqref{uniformly.monotone.f} to obtain
  \begin{align}
    \innerp*{x_n-\overline{x}}{\nabla f(x_n) - \nabla
    f(\overline{x})} \geq 2\varphi_{\mathcal{B}}
    \left(\norm{x_n-\overline{x}}\right)\,, \quad\forall
    n\,. \label{uniformly.monotone.f.xn}
  \end{align}
  Since $x_n\rightharpoonup_{n\to\infty} \overline{x}$ and
  $\lim_{n\to\infty}\nabla f(x_n) = \nabla f(\overline{x})$ by \eqref{grad.f.conv.strongly}, the
  application of $\lim_{n\to\infty}$ to \eqref{uniformly.monotone.f.xn} and \cite[Lem.~2.41(iii),
  p.~37]{Bauschke.Combettes.book} suggest that
  $\lim_{n\to\infty} \varphi_{\mathcal{B}} (\norm{x_n-\overline{x}})=0$, and thus
  $\lim_{n\to\infty}\norm{x_n-\overline{x}} = 0$, due to the properties of $\varphi_{\mathcal{B}}$.

  Let now \cref{as:unif.convx.g} hold true. Then, according to \cite[Ex.~22.3(iii),
  p.~324]{Bauschke.Combettes.book}, given a bounded set $\mathcal{B}\subset \domain\partial g$,
  there exists an increasing function $\varphi_{\mathcal{B}}: [0,+\infty)\to [0,+\infty]$, which
  vanishes only at $0$, s.t.\ $\forall x,x'\in\mathcal{B}$, and $\forall \xi\in\partial g(x)$,
  $\forall\xi'\in \partial g(x')$,
  \begin{align}
    \innerp*{x-x'}{\xi - \xi'} \geq
    2\varphi_{\mathcal{B}}
    \left(\norm{x-x'}\right)\,. \label{uniformly.monotone.g}
  \end{align}
  According to \eqref{existence.xi}, $x_n\in\domain \partial g$, $\forall n$. Moreover, as the
  discussion after \eqref{Fitzpatrick} demonstrated, $\overline{x}\in \domain \partial g$. Define
  thus the bounded set $\mathcal{B} := (x_n)_n \cup \Set{\overline{x}}\subset \domain\partial g$,
  and set $x:= x_n$, $x' := \overline{x}$, $\xi := \xi_n$ and $\xi' := \overline{\xi}$ in
  \eqref{uniformly.monotone.g} to obtain
  \begin{align}
    \innerp*{x_n-\overline{x}}{\xi_n - \overline{\xi}} \geq
    2\varphi_{\mathcal{B}} 
    \left(\norm{x_n-\overline{x}}\right)\,, \quad\forall
    n\,. \label{uniformly.monotone.g.xn}
  \end{align}
  Similarly to \eqref{Fitzpatrick}, it can be verified that
  $\lim_{n\to\infty}\innerp{x_n}{\xi_n} = \innerp{\overline{x}}{\overline{\xi}}$. Thus,
  \begin{align*}
    \lim_{n\to\infty} \innerp*{x_n-\overline{x}}{\xi_n -
    \overline{\xi}} 
    & =  \lim_{n\to\infty} \innerp*{x_n}{\xi_n} -
    \lim_{n\to\infty} \innerp*{x_n}{\overline{\xi}} 
    - \lim_{n\to\infty} \innerp*{\overline{x}}{\xi_n}
    + \innerp{\overline{x}}{\overline{\xi}} \\
    & = \innerp{\overline{x}}{\overline{\xi}} -
      \innerp{\overline{x}}{\overline{\xi}} -
      \innerp{\overline{x}}{\overline{\xi}} +
      \innerp{\overline{x}}{\overline{\xi}} = 0\,.
  \end{align*}
  Hence, the application of $\lim_{n\to\infty}$ to \eqref{uniformly.monotone.g.xn} yields
  $\lim_{n\to\infty} \varphi_{\mathcal{B}} (\norm{x_n-\overline{x}})=0$, and thus
  $\lim_{n\to\infty}\norm{x_n-\overline{x}} = 0$.
\end{proof}

\begin{corollary}\label{cor:f=0.g=0}
  Consider again the setting of \cref{thm:basic}. In the case where the non-smooth part of the
  composite loss becomes zero, \ie, $g:=0$, then \eqref{FM-HSDM} takes the special form
  \begin{subequations}\label{FM-HSDM.g=0}
    \begin{align}
      x_{1/2}
      & := T_{\alpha}x_0 -
        \lambda \nabla f(x_0)\,, \\
      x_{1}
      & := x_{1/2}\,,\\
      x_{n+3/2}
      & := x_{n+1/2} - \left[ T_{\alpha}x_n - \lambda \nabla
        f(x_n) \right] + \left[ Tx_{n+1} - \lambda \nabla
        f(x_{n+1})\right]\,,\\
      x_{n+2}
      & := x_{n+3/2}\,.
    \end{align}
  \end{subequations}
  Consider $\alpha\in [0.5,1)$ and $\lambda\in (0,2(1-\alpha)/L)$. Then the following hold true.
  \begin{corlist}
  \item For sequence $(x_n)_{n\in\IntegerP}$ of \eqref{FM-HSDM.g=0}, there exist a sequence
    $(v_n)_{n\in\IntegerP} \subset \mathcal{X}$ and a strongly positive operator
    $\Theta: \mathcal{X}^2 \to \mathcal{X}^2$ s.t.\ sequence
    $(y_n:= (x_n,v_n))_{n\in\IntegerPP\setminus \Set{1}}$ is Fej\'{e}r monotone~\cite[Def.~5.1,
    p.~75]{Bauschke.Combettes.book} w.r.t.\ $\Upsilon_*^{(\lambda)}$ of \cref{prop:O*} (under
    $g=0$) in the Hilbert space $(\mathcal{X}^2, \innerp{\cdot}{\cdot}_{\Theta})$.
  \item Sequence $(x_n)_{n\in\IntegerP}$ of \eqref{FM-HSDM.g=0} converges weakly to a point that
    solves $\text{VIP}(\nabla f, \Fix T)$.
  \end{corlist}

  In the case where $f:=0$, the FM-HSDM recursions take the form
  \begin{subequations}\label{FM-HSDM.f=0}
    \begin{align}
      x_{1/2}
      & := T_{\alpha} x_0\,, \\
      x_{1}
      & := \prox_{\lambda g} (x_{1/2})\,, \\
      x_{n+3/2}
      & := x_{n+1/2} - T_{\alpha}x_n + Tx_{n+1} \,,\\
      x_{n+2}
      & := \prox_{\lambda g} (x_{n+3/2})\,.
    \end{align}
  \end{subequations}
  Consider $\alpha\in [0.5,1)$ and $\lambda\in\RealPP$. Then the following hold true.  
  \begin{corlist}\setcounter{corlisti}{2}
  \item\label{test} For sequence $(x_n)_{n\in\IntegerP}$ of \eqref{FM-HSDM.f=0}, there exist a
    sequence $(v_n)_{n\in\IntegerP} \subset \mathcal{X}$ and a strongly positive operator
    $\Theta: \mathcal{X}^2 \to \mathcal{X}^2$ s.t.\ sequence
    $(y_n:= (x_n,v_n))_{n\in\IntegerPP\setminus \Set{1}}$ is Fej\'{e}r monotone~\cite[Def.~5.1,
    p.~75]{Bauschke.Combettes.book} w.r.t.\ $\Upsilon_*^{(\lambda)}$ of \cref{prop:O*} (under
    $f=0$) in the Hilbert space $(\mathcal{X}^2, \innerp{\cdot}{\cdot}_{\Theta})$.
  \item Sequence $(x_n)_{n\in\IntegerP}$ of \eqref{FM-HSDM.f=0} converges weakly to a point that
    solves $\text{VIP}(\partial g, \Fix T)$.
  \end{corlist}
  
\end{corollary}

\begin{proof}
  The proof becomes a special case of the one of \cref{thm:basic}, after setting $f:=0$ or
  $g:=0$. With regards to the reason behind the relaxation of $\lambda$ offered by
  \eqref{FM-HSDM.f=0}, notice that any $\lambda\in\RealPP$ can serve as the Lipschitz constant of
  $\nabla f = 0$.
\end{proof}

The following theorem draws even stronger links with the original form of HSDM.

\begin{theorem}\label{thm:plain.vanilla.HSDM}
  Consider $f\in \Gamma_0(\mathcal{X})$, with $L$ being the Lipschitz-continuity constant of
  $\nabla f$. Moreover, given the closed affine set $\mathcal{A}$, consider any
  $T\in\mathfrak{T}_{\mathcal{A}}$, and for $\lambda\in\RealPP$, an arbitrarily fixed
  $x_0\in\mathcal{X}$, and for all $n\in\IntegerP$ form the iterations:
  \begin{subequations}\label{FM-HSDM.g=0.original}
    \begin{align}
      x_{1/2}
      & := T_{\alpha}x_0 -
        \lambda \nabla f(T_{\alpha}x_0)\,, \\
      x_{1}
      & := x_{1/2}\,,\\
      x_{n+3/2}
      & := x_{n+1/2} - \left[ T_{\alpha}x_n - \lambda \nabla
        f(T_{\alpha}x_n) \right] + \left[ Tx_{n+1} - \lambda \nabla
        f(T_{\alpha}x_{n+1})\right]\,,\\
      x_{n+2}
      & := x_{n+3/2}\,,
    \end{align}
  \end{subequations}
  where $T_{\alpha}$ is defined in \eqref{def.Talpha}. Consider also $\alpha\in [0.5,1)$ and
  $\lambda\in (0,2(1-\alpha)^2/L)$. Then, the following hold true.

  \begin{enumerate}[label=\textbf{(\roman*)}]

  \item There exist a sequence $(v_n)_{n\in\IntegerP} \subset \mathcal{X}$ and a strongly positive
    operator $\Upsilon: \mathcal{X}^2 \to \mathcal{X}^2$ s.t.\ sequence
    $(y_n:= (x_n,v_n))_{n\in\IntegerPP\setminus \Set{1}}$ is Fej\'{e}r monotone~\cite[Def.~5.1,
    p.~75]{Bauschke.Combettes.book} w.r.t.\ $\Upsilon_*^{(\lambda)}$ of \cref{prop:O*} (under $g=0$)
    in the Hilbert space $(\mathcal{X}^2, \innerp{\cdot}{\cdot}_{\Upsilon})$.

  \item Sequence $(x_n)_n$ of \eqref{FM-HSDM.g=0.original} converges weakly to a point that solves
    $\text{VIP}(\nabla f, \Fix T)$.

  \item\label{thm:plain.vanilla.HSDM.strong.conv} If \cref{as:unif.convx.f} also holds true, then
    $(x_n)_n$ of \eqref{FM-HSDM.g=0.original} converges strongly to a point that solves
    $\text{VIP}(\nabla f, \Fix T)$.

  \end{enumerate}

\end{theorem}

\begin{proof}
  \textbf{(i)} \cref{prop:O*} takes the following special form in the present context: if
  $\exists v_*\in\mathcal{X}$ s.t.
  \begin{align}
    (x_*, v_*) \in \Upsilon_*^{(\lambda)} := 
    \Set*{(x, v) \in\Fix T\times \mathcal{X}\given -
    \tfrac{1}{\lambda}U v = \nabla f(x)}\,, \label{O*.orig.HSDM}
  \end{align}
  then $x_*$ solves $\text{VIP}(\nabla f, \Fix T)$. 

  By following the same steps which start from the beginning of the proof of \cref{thm:basic} till
  \eqref{sum.T.and.v}, it can be verified that
  \begin{align}
    -(1-2\alpha)(T-\Id) x_{n+1} 
    - Q_{\alpha} (x_{n+1} - x_n) 
    - Uv_{n+1} = \lambda \nabla f(T_{\alpha} x_n)
    \,, \label{basic.recursion.n.n+1.orig.HSDM}
  \end{align}
  and by considering any $(x_*, v_*)\in\Upsilon_*^{(\lambda)}$,
  \begin{align}
  \lambda & [\nabla
    f(T_{\alpha} x_n) - \nabla f(T_{\alpha} x_*)] \notag\\
          & = -(1-2\alpha)(Q-\Id)(x_{n+1} - x_*)
            - Q_{\alpha} (x_{n+1} - x_n)
            - U(v_{n+1} - v_*)
            \,. \label{describe.grads.orig.HSDM}
  \end{align}
  As in the proof of \cref{thm:basic}, the Baillon-Haddad theorem~\cite{Baillon.Haddad},
  \cite[Cor.~18.16, p.~270]{Bauschke.Combettes.book} suggests that
  % \begin{subequations}
  \begin{align}
    \tfrac{2\lambda}{L}
    & \norm{\nabla
      f(T_{\alpha}x_n) - \nabla
      f(T_{\alpha}x_*)}^2 \notag\\ 
    & \mathbin{\leq} 2\lambda
      \innerp{T_{\alpha} x_n -
      T_{\alpha} x_*}{\nabla f(T_{\alpha}x_n) 
      - \nabla f(T_{\alpha} x_*)} \notag\\
    & \mathbin{=} 2\lambda
      \innerp{Q_{\alpha}(x_n - x_*)}{\nabla
      f(T_{\alpha} x_n) - \nabla f(T_{\alpha} x_*)}
      \notag\\ 
    & \mathbin{=} 2\lambda
      \innerp{x_n - x_*}{Q_{\alpha}[\nabla 
      f(T_{\alpha}x_n) - \nabla
      f(T_{\alpha}x_*)]} \notag\\
    & \mathbin{=} 2\lambda
      \innerp{x_{n+1} - x_*}{Q_{\alpha}
      [\nabla f(T_{\alpha}x_n) - \nabla
      f(T_{\alpha}x_*)]} \notag\\
    & \mathbin{\hphantom{\leq}} + 2\lambda
      \innerp{x_{n} - x_{n+1}}{Q_{\alpha}[\nabla
      f(T_{\alpha}x_n) - \nabla
      f(T_{\alpha}x_*)]}\notag\\
    & \mathbin{=} -2(1-2\alpha) \innerp{x_{n+1} -
      x_*}{Q_{\alpha}
      (Q-\Id)(x_{n+1} - x_*)} \notag\\
    & \mathbin{\hphantom{=}} - 2\innerp{x_{n+1} -
      x_*}{Q_{\alpha}^2 (x_{n+1} -
      x_n)} -
      2\innerp{x_{n+1} -
      x_*}{Q_{\alpha} U(v_{n+1} - v_*)} \notag\\
    & \mathbin{\hphantom{\leq}} + 2\lambda \innerp{x_{n}
      - x_{n+1}}{Q_{\alpha}[\nabla
      f(T_{\alpha}x_n) - \nabla 
      f(T_{\alpha}x_*)]} \notag\\
    & \mathbin{=} -2(1-2\alpha)
      \innerp{x_{n+1} - x_*}{Q_{\alpha}(Q-\Id)
      (x_{n+1} - x_*)} \notag\\
    & \hphantom{\mathbin{=}} - 2\innerp{x_{n+1} -
      x_*}{Q_{\alpha}^2
      (x_{n+1} - x_n)} -
      2\innerp{U(x_{n+1} - x_*)}{Q_{\alpha}(v_{n+1} - 
      v_*)} \notag\\
    & \hphantom{\mathbin{=}} + 2\lambda \innerp{x_{n} -
      x_{n+1}}{Q_{\alpha}[\nabla f(T_{\alpha}
      x_n) - \nabla f(T_{\alpha} x_*)]}\notag\\
    & \mathbin{\leq} -2(1-2\alpha)
      \innerp{x_{n+1} - x_*}{Q_{\alpha}(Q
      -\Id)(x_{n+1} - x_*)} \notag\\
    & \hphantom{\mathbin{\leq}} - 2\innerp{x_{n+1} - x_*}{Q_{\alpha}^2
      (x_{n+1} - x_n)} - \tfrac{2}{1-\alpha}
      \innerp{v_{n+1} - v_n}{Q_{\alpha}
      (v_{n+1} - v_*)} \notag\\
    & \hphantom{\mathbin{\leq}} + \tfrac{\lambda
      L}{2}\norm{x_n - x_{n+1}}^2 + \tfrac{2\lambda}{L}
      \norm{Q_{\alpha}[\nabla f(x_n) - \nabla
      f(x_*)]}^2 \notag\\
    & \mathbin{\leq} -2(2\alpha-1)
      \innerp{x_{n+1} -
      x_*}{Q_{\alpha}(\Id-Q)(x_{n+1}  
      - x_*)} \notag\\
    & \hphantom{\mathbin{\leq}} - 2\innerp{x_{n+1} - x_*}{Q_{\alpha}^2
      (x_{n+1} - x_n)} - \tfrac{2}{1-\alpha}
      \innerp{v_{n+1} - v_n}{Q_{\alpha}
      (v_{n+1} - v_*)} \notag\\
    & \hphantom{\mathbin{\leq}} + \tfrac{\lambda
      L}{2}\norm{x_n - x_{n+1}}^2 + \tfrac{2\lambda}{L}
      \norm{\nabla f(T_{\alpha}x_n) - \nabla
      f(T_{\alpha} x_*)}^2 \notag\\
    & \mathbin{\leq} 
      2\innerp{x_* - x_{n+1}}{Q_{\alpha}^2
      (x_{n+1} - x_n)} 
      + \tfrac{2}{1-\alpha}
      \innerp{v_{n+1} - v_n}{Q_{\alpha}
      (v_*- v_{n+1})} \notag\\
    & \mathbin{\hphantom{\leq}} + \tfrac{\lambda
      L}{2}\norm{x_n - x_{n+1}}^2
      + \tfrac{2\lambda}{L} \norm{\nabla
      f(T_{\alpha}x_n) - \nabla
      f(T_{\alpha}x_*)}^2\,. \label{BH.orig.HSDM}
  \end{align}
  % \end{subequations}

  Mapping $Q_{\alpha}^2$ is strongly positive: indeed, if $U_{\alpha}$ denotes the square root of
  the strongly positive $Q_{\alpha}$ [\cf~\eqref{Qa.is.PD}], then $\forall x\in\mathcal{X}$,
  $\innerp{Q_{\alpha}^2x}{x} = \innerp{U_{\alpha} Q_{\alpha} U_{\alpha}x}{x} = \innerp{Q_{\alpha}
    U_{\alpha}x}{U_{\alpha}x} \geq (1-\alpha)\innerp{U_{\alpha}x}{U_{\alpha}x} =
  (1-\alpha)\innerp{Q_{\alpha}x}{x} \geq (1-\alpha)^2\norm{x}^2$. Define now the mapping
  $\Upsilon : \mathcal{X}^2 \to \mathcal{X}^2: (x,v) \mapsto (Q_{\alpha}^2x, [1/(1-\alpha)]
  Q_{\alpha}v)$. Mapping $\Upsilon$ turns out to be strongly positive, w.r.t.\ the standard inner
  product of $\mathcal{X}^2$: $\innerp{(x,v)}{(x,v')} := \innerp{x}{x'} + \innerp{v}{v'}$,
  $\forall (x,v), (x',v')\in \mathcal{X}^2$, due to the strong positivity of $Q_{\alpha}^2$ and
  $[1/(1-\alpha)] Q_{\alpha}$. Consequently, one can consider
  $(\mathcal{X}^2, \innerp{\cdot}{\cdot}_{\Upsilon})$ as a Hilbert space equipped with the inner
  product
  $\innerp{(x,v)}{(x,v')}_{\Upsilon} := \innerp{x}{Q_{\alpha}^2x'} + [1/(1-\alpha)]
  \innerp{v}{Q_{\alpha} v'}$, $\forall (x,v), (x',v')\in \mathcal{X}^2$. As such,
  \eqref{BH.orig.HSDM} becomes
  \begin{alignat}{2}
    0 & \,\mathbin{\leq}\, && 2 \innerp{y_{n+1} -
      y_{n}}{\Upsilon(y_*- y_{n+1})} +
    \tfrac{\lambda
      L}{2}\norm{x_n - x_{n+1}}^2 \notag\\
    & \,\mathbin{=}\, && 2 \innerp{y_{n+1} -
      y_{n}}{y_*- y_{n+1}}_{\Upsilon} +
    \tfrac{\lambda L}{2}\norm{x_n
      - x_{n+1}}^2 \notag\\
    & \,\mathbin{=}\, && \norm{y_n -
      y_*}_{\Upsilon}^2 - \norm{y_{n+1} -
      y_*}_{\Upsilon}^2 - \norm{y_{n+1} -
      y_n}_{\Upsilon}^2 + \tfrac{\lambda
      L}{2}\norm{x_n -
      x_{n+1}}^2\,. \label{pre.Fejer.orig.HSDM}
  \end{alignat}
  Choose, now, any $\zeta'$ with $\lambda L/[2(1-\alpha)^2] < \zeta' < 1$. Then, for any
  $y = (x, v)\in \mathcal{X}^2$,
  \begin{align*}
    \tfrac{\lambda L}{2} \norm{x}^2
    & < \zeta' (1-\alpha)^2 \norm{x}^2 \leq \zeta'
      \innerp{x}{Q_{\alpha}^2 x}\\
    & \leq \zeta' \innerp{x}{Q_{\alpha}^2 x} +
      \zeta' \tfrac{1}{1-\alpha}\innerp{v}{Q_{\alpha}
      v} = \zeta' \norm{y}^2_{\Upsilon} \,.
  \end{align*}
  This argument together with \eqref{pre.Fejer.orig.HSDM} yield
  \begin{align}
    \norm{y_n - y_*}_{\Upsilon}^2 -
    \norm{y_{n+1} - y_*}_{\Upsilon}^2
    & \geq \norm{y_{n+1} - y_n}_{\Upsilon}^2 -
      \tfrac{\lambda L}{2}\norm{x_n - x_{n+1}}^2 \notag\\ 
    & \geq \norm{y_{n+1} - y_n}_{\Upsilon}^2 - \zeta'
      \norm{y_{n+1} - y_n}_{\Upsilon}^2 \notag\\
    & = (1-\zeta') \norm{y_{n+1} - y_n}_{\Upsilon}^2 
      \,, \label{Fejer.orig.HSDM}
  \end{align}  
  \ie, sequence $(y_n)_{n\in\IntegerP} \subset (\mathcal{X}^2, \innerp{\cdot}{\cdot}_{\Upsilon})$ is
  Fej\'{e}r monotone w.r.t.\ $\Upsilon_*^{(\lambda)}$ of \eqref{O*.orig.HSDM}.

  \textbf{(ii)} Due to Fej\'{e}r monotonicity, $(y_n)$ is bounded~\cite[Prop.~5.4(i),
  p.~76]{Bauschke.Combettes.book} and possesses a non-empty set of weakly sequential cluster points
  $\mathfrak{W}[(y_n)_n]$~\cite[Lem.~2.37, p.~36]{Bauschke.Combettes.book}. Moreover, it can be
  readily verified, as in \eqref{as.regularity.x.v}, that $\lim_{n\to\infty} (y_{n+1} - y_n) = 0$,
  $\lim_{n\to\infty} (x_{n+1} - x_n) = 0$ and $\lim_{n\to\infty} (v_{n+1} - v_n) = 0$. The rest of
  the proof follows steps similar to those after \eqref{as.regularity.x.v} in the proof of
  \cref{thm:basic}, but with the following twist: $\nabla f(x_{n_k})$ is replaced by
  $\nabla f(T_{\alpha}x_{n_k})$, where all the asymptotic results of the proof of \cref{thm:basic}
  continue to hold due to the Lipschitz continuity of $\nabla f$ and the nonexpansiveness of
  $T_{\alpha}$, \eg, $\forall x,x'\in \mathcal{X}$,
  \begin{align*}
    \norm{\nabla f(T_{\alpha}x) - \nabla f(T_{\alpha}x')} \leq
    L \norm{T_{\alpha}x - T_{\alpha}x'} \leq L
    \norm{x - x'}\,.
  \end{align*}

  \textbf{(iii)} Part~\textbf{(ii)} of this proof has demonstrated that sequences $(x_n)_n$ and
  $(Uv_n)_n$ converge weakly to $\overline{x}$ and $U\overline{v}$, respectively. Consequently, in a
  way similar to part~\textbf{(ii)} of the proof of \cref{thm:basic}, it can be shown also here that
  $(\xi_n)_n$ converges weakly to $\overline{\xi}$.

  Let \cref{as:unif.convx.f} hold true. Then, according to \cite[Ex.~22.3(iii),
  p.~324]{Bauschke.Combettes.book}, given a bounded set $\mathcal{B}\subset \mathcal{X}$, there
  exists an increasing function $\varphi_{\mathcal{B}}: [0,+\infty)\to [0,+\infty]$, which vanishes
  only at $0$, s.t.\ $x, x'\in\mathcal{B}$,
  \begin{align}
    \innerp*{x- x'}{\nabla f(x) -
    \nabla f(x')} \geq 2\varphi_{\mathcal{B}}
    \left(\norm{x - x'}\right)\,. \label{uniform.convex.Talpha} 
  \end{align}
  Due to the nonexpansiveness of $T_{\alpha}$ and the boundedness of $(x_n)_n$, by part~\textbf{(i)}
  of the proof, it turns out that $(T_{\alpha}x_n)_n$ is also bounded:
  $\norm{T_{\alpha}x_n}\leq \norm{T_{\alpha}x_n - T_{\alpha}\overline{x}} +
  \norm{T_{\alpha}\overline{x}} \leq \norm{x_n - \overline{x}} + \norm{\overline{x}} \leq \norm{x_n}
  + 2\norm{\overline{x}}\leq C''+2\norm{\overline{x}}$, for some $C''\in\RealPP$ (recall that
  $\overline{x}\in \Fix T_{\alpha} = \Fix T$). Define, thus, the bounded set
  $\mathcal{B} := (T_{\alpha}x_n)_n \cup \Set{\overline{x}}$. As such, \eqref{uniform.convex.Talpha}
  yields
   \begin{align}
     & \innerp*{(T_{\alpha}-\Id)x_n}{\nabla
     f(T_{\alpha}x_n) - \nabla 
     f(\overline{x})} + \innerp*{x_n - \overline{x}}{\nabla
     f(T_{\alpha}x_n) - \nabla 
     f(\overline{x})} \notag\\
     & = \innerp*{T_{\alpha}x_n- \overline{x}}{\nabla
     f(T_{\alpha}x_n) - \nabla 
     f(\overline{x})} \geq 2\varphi_{\mathcal{B}}
     \left(\norm{T_{\alpha}x_n-\overline{x}}\right)\,, \quad
       \forall n\,. \label{Txn.converges}
   \end{align}
   Part~\textbf{(i)} of this proof has already showed that $\lim_{n\to\infty} (T-\Id)x_n = 0$. As
   such, $\lim_{n\to\infty} (T_{\alpha}-\Id)x_n = \alpha \lim_{n\to\infty}(T-\Id)x_n = 0$. Moreover,
   note that $x_n\rightharpoonup_{n\to\infty} \overline{x}$, and
   $\lim_{n\to\infty} \nabla f(T_{\alpha}x_n) = \nabla f(\overline{x})$. Hence, due also to
   \cite[Lem.~2.41(iii), p.~37]{Bauschke.Combettes.book}, an application of $\lim_{n\to\infty}$ to
   both sides of \eqref{Txn.converges} results in
   $\lim_{n\to\infty} \varphi_{\mathcal{B}} (\norm{T_{\alpha}x_n-\overline{x}}) = 0$, and thus
   $\lim_{n\to\infty} T_{\alpha}x_n = \overline{x}$. Using
   $\lim_{n\to\infty} (T_{\alpha}-\Id)x_n = 0$, one can easily verify that
   $\lim_{n\to\infty} x_n = \lim_{n\to\infty} (\Id - T_{\alpha})x_n + \lim_{n\to\infty}
   T_{\alpha}x_n = \overline{x}$, which establishes part~\ref{thm:plain.vanilla.HSDM.strong.conv} of
   \cref{thm:plain.vanilla.HSDM}.
\end{proof}

The following theorems present convergence rates on the sequence of FM-HSDM estimates.

\begin{theorem}\label{thm:Oh}
  For sequence $(x_n)_{n\in\IntegerP}$ of \eqref{FM-HSDM}, there exists $\xi_n\in \partial g(x_n)$,
  $\forall n$, s.t.\ for any $x_*\in \Fix T$,
  \begin{subequations}
    \begin{align}
      & \tfrac{1}{n+1} \sum\nolimits_{\nu=0}^{n}
        \innerp{x_{\nu+1} - x_*}{(\Id-
        Q)(x_{\nu+1} - x_*)} =
        O(\tfrac{1}{n+1})\,, \label{rate:I-Q}\\ 
      & \tfrac{1}{n+1} \sum\nolimits_{\nu=0}^{n}
        \norm{Uv_{\nu+1} + \lambda [\nabla f(x_{\nu}) +
        \xi_{\nu+1}]}^2 =
        O(\tfrac{1}{n+1})\,, \label{rate:qualif.cond}\\ 
      & \tfrac{1}{n+1}\sum\nolimits_{\nu=0}^{n} \norm{(\Id -
        T)x_{\nu+1}}^2 = O(\tfrac{1}{n+1})\,, \label{rate:I-T} 
    \end{align}
  \end{subequations}
  where the big-oh notation $a_n = O(b_n)$, $b_n>0$, means
  $\lim\sup_{n\to\infty}|a_n|/b_n < +\infty$. Regarding sequence $(x_n)_{n\in\IntegerP}$ of
  \eqref{FM-HSDM.g=0}, \eqref{rate:I-Q}--\eqref{rate:I-T} still hold true, but $\xi_{\nu+1}$ is set
  equal to $0$ in \eqref{rate:qualif.cond}. Similarly, for sequence $(x_n)_{n\in\IntegerP}$ of
  \eqref{FM-HSDM.g=0.original}, \eqref{rate:I-Q}, \eqref{rate:I-T} as well as
  \begin{align*}
    \tfrac{1}{n+1} \sum\nolimits_{\nu=0}^{n}
    \norm{Uv_{\nu+1} + \lambda \nabla f(T_{\alpha} x_{\nu}) }^2 =
    O(\tfrac{1}{n+1})
  \end{align*}
  hold true.
\end{theorem}

\begin{proof}
  First, notice by \eqref{Qa.is.PD}, \cref{prop:strongly.pos} and $\norm{Q_{\alpha}}\leq 1$ that
  $Q_{\alpha}^{-1}$ exists and it is strongly positive with
  \begin{align}
    \norm{Q_{\alpha}^{-1}} \leq \tfrac{1}{1-\alpha}\,;\qquad 
    (1-\alpha) \norm{x}^2 \leq
    \tfrac{(1-\alpha)}{\norm{Q_{\alpha}}^2} \norm{x}^2 \leq
    \innerp{Q_{\alpha}^{-1}x}{x}, \quad \forall
    x\in\mathcal{X}\,. \label{Qa.inverse.strongly.pos}
  \end{align}
  Then, going back to the discussion following \eqref{Qa.is.PD},
  \begin{subequations}
    \begin{align}
      & \norm{y_{n+1} - y_n}_{\Theta}^2 \notag\\
      & = \norm{x_{n+1} - x_n}_{Q_{\alpha}}^2 + \tfrac{1}{1-\alpha}
        \norm{v_{n+1} - v_n}^2 \label{consecutive.ys.1}\\ 
      & = \norm{Q_{\alpha} (x_{n+1} -
        x_n)}_{Q_{\alpha}^{-1}}^2 + \tfrac{1}{1-\alpha}
        \norm{(1-\alpha) U(x_{n+1} -
        x_*)}^2 \label{consecutive.ys.2} \\ 
      & = \norm{Uv_{n+1} + \lambda [\nabla f(x_n) +
        \xi_{n+1}] -  (1-2\alpha)(\Id-T) x_{n+1}}^2_{Q_{\alpha}^{-1}}
        \notag \\
      & \hphantom{=\ } +
        \tfrac{1}{1-\alpha} \norm{(1-\alpha) U(x_{n+1} -
        x_*)}^2 \label{consecutive.ys.3} \\
      & = \norm{Uv_{n+1} + \lambda [\nabla f(x_n) +
        \xi_{n+1}]}_{Q_{\alpha}^{-1}}^2 + (1-2\alpha)^2
        \norm{(\Id-T) x_{n+1}}^2_{Q_{\alpha}^{-1}}
        \notag \\ 
      & \hphantom{=\ } -2\innerp{Uv_{n+1} + \lambda [\nabla
        f(x_n) + \xi_{n+1}]}{(1-2\alpha)(\Id-T)
        x_{n+1}}_{Q_{\alpha}^{-1}} \notag\\
      & \hphantom{=\ } + \tfrac{1}{1-\alpha} \norm{(1-\alpha)
        U(x_{n+1} - x_*)}^2 \notag \\ 
      & \geq \tfrac{1}{\rho}
        \norm{Uv_{n+1} + \lambda [\nabla f(x_n) +
        \xi_{n+1}]}_{Q_{\alpha}^{-1}}^2 -
        \tfrac{(1-2\alpha)^2}{\rho-1}
        \norm{(\Id-T) x_{n+1}}^2_{Q_{\alpha}^{-1}}
        \notag\\
      & \hphantom{=\ } + \tfrac{1}{1-\alpha} \norm{(1-\alpha)
        U(x_{n+1} - x_*)}^2 \label{consecutive.ys.4}\\
      & = \tfrac{1}{\rho}
        \norm{Uv_{n+1} + \lambda [\nabla f(x_n) +
        \xi_{n+1}]}_{Q_{\alpha}^{-1}}^2 -
        \tfrac{(1-2\alpha)^2}{\rho-1} \norm{(\Id-T)x_{n+1} -
        (\Id-T) x_*}^2_{Q_{\alpha}^{-1}} \notag\\
      & \hphantom{=\ } + (1-\alpha)\innerp{x_{n+1} -
        x_*}{(\Id - Q)(x_{n+1} - x_*)} \notag\\
      & = \tfrac{1}{\rho} \norm{Uv_{n+1} + \lambda
        [\nabla f(x_n) + \xi_{n+1}]}_{Q_{\alpha}^{-1}}^2 -
        \tfrac{(1-2\alpha)^2}{\rho-1}
        \norm{(\Id-Q)(x_{n+1} - x_*)}^2_{Q_{\alpha}^{-1}} \notag\\
      & \hphantom{=\ }
        + (1-\alpha)\innerp{x_{n+1} - x_*}{(\Id-
        Q)(x_{n+1} - x_*)} \notag\\
      & = \tfrac{1}{\rho} \norm{Uv_{n+1} + \lambda [\nabla
        f(x_n) + \xi_{n+1}]}_{Q_{\alpha}^{-1}}^2 \notag\\
      & \hphantom{=\ } -
        \tfrac{(1-2\alpha)^2}{\rho-1} \innerp{x_{n+1} -
        x_*}{(\Id-Q) Q_{\alpha}^{-1}
        (\Id-Q)(x_{n+1} - x_*)} \notag\\
      & \hphantom{=\ } + (1-\alpha)\innerp{x_{n+1} -
        x_*}{(\Id- Q)(x_{n+1} - x_*)} \notag\\
      & \geq \tfrac{1}{\rho} \norm{Uv_{n+1} + \lambda
        [\nabla f(x_n) + \xi_{n+1}]}_{Q_{\alpha}^{-1}}^2 \\
      & \hphantom{\geq\ }-
        \tfrac{(2\alpha-1)^2}{(\rho-1)(1-\alpha)}
        \innerp{x_{n+1} - x_*}{(\Id-Q) (x_{n+1} - x_*)}\notag\\
      & \hphantom{\geq\ } + (1-\alpha)\innerp{x_{n+1} - x_*}{(\Id-
        Q)(x_{n+1} - x_*)} \label{consecutive.ys.5} \\
      & = \tfrac{1}{\rho} \norm{Uv_{n+1} + \lambda [\nabla
        f(x_n) + \xi_{n+1}]}_{Q_{\alpha}^{-1}}^2 +
        \theta \innerp{x_{n+1} - x_*}{(\Id-
        Q)(x_{n+1} - x_*)} \label{consecutive.ys.6} \\
      & \geq \tfrac{(1-\alpha)}{\rho} \norm{Uv_{n+1} + \lambda [\nabla
        f(x_n) + \xi_{n+1}]}^2 + \theta
        \innerp{x_{n+1} - x_*}{(\Id-
        Q)(x_{n+1} - x_*)}\,, \label{consecutive.ys.7}\\
      & \geq \tfrac{(1-\alpha)}{\rho}
        \norm{Uv_{n+1} + \lambda [\nabla f(x_{n}) +
        \xi_{n+1}]}^2 +
        \theta (1-\alpha) \norm{(\Id
        - Q)(x_{n+1} - x_*)}_{Q_{\alpha}^{-1}}^2
        \label{consecutive.ys.8}\\ 
      & = \tfrac{(1-\alpha)}{\rho} \norm{Uv_{n+1} + \lambda [\nabla
        f(x_{n}) + \xi_{n+1}]}^2 +
        \theta (1-\alpha) \norm{(\Id -
        T)x_{n+1}}_{Q_{\alpha}^{-1}}^2
        \label{consecutive.ys.9}\\
      & \geq \tfrac{(1-\alpha)}{\rho} \norm{Uv_{n+1} + \lambda [\nabla
        f(x_{n}) + \xi_{n+1}]}^2 +
        \theta(1-\alpha)^2\norm{(\Id -
        T) x_{n+1}}^2\,, \label{consecutive.ys.10}
    \end{align}
  \end{subequations}
  where the definition of $\Upsilon$, given after \eqref{BH.orig.HSDM}, is used in
  \eqref{consecutive.ys.1}, \eqref{consecutive.v} in \eqref{consecutive.ys.2},
  \eqref{basic.recursion.n.n+1} in \eqref{consecutive.ys.3}, \eqref{Young.ineq} with
  $\eta := \rho/(\rho-1)$, $a := Uv_{n+1} + \lambda [\nabla f(x_n) + \xi_{n+1}]$,
  $b := (1-2\alpha)(\Id-T) x_{n+1}$ and $\Pi := Q_{\alpha}^{-1}$, as well as $\rho>1$ in
  \eqref{consecutive.ys.4}, and
  \begin{subequations}
    \begin{align}
      & \innerp*{x_{n+1} - x_*}{(\Id-Q) Q_{\alpha}^{-1}
        (\Id-Q)(x_{n+1} - x_*)} \notag\\
      & \hspace{10ex} = \innerp*{x_{n+1} - x_*}{U^2Q_{\alpha}^{-1}
        U^2(x_{n+1} - x_*)} \notag \\
      & \hspace{10ex} = \innerp*{U(x_{n+1} -
        x_*)}{\left(UQ_{\alpha}^{-1}U\right) U(x_{n+1} - x_*)} \notag \\
      & \hspace{10ex} \leq \norm*{UQ_{\alpha}^{-1}U}\,
        \innerp*{U(x_{n+1} - x_*)}{U(x_{n+1} - x_*)} \label{UQinvU} \\
      & \hspace{10ex} = \norm*{UQ_{\alpha}^{-1}U}\,
        \innerp*{x_{n+1} - x_*}{(\Id-Q)(x_{n+1} - x_*)}  \notag\\
      & \hspace{10ex} \leq \norm*{U}^2 \norm*{Q_{\alpha}^{-1}}\,
        \innerp*{x_{n+1} - x_*}{(\Id-Q)(x_{n+1} - x_*)} \notag\\
      & \hspace{10ex} = \norm*{\Id-Q} \norm*{Q_{\alpha}^{-1}}\, 
        \innerp*{x_{n+1} - x_*}{(\Id-Q)(x_{n+1} - x_*)} \notag\\
      & \hspace{10ex} \leq \tfrac{1}{1-\alpha} \innerp*{x_{n+1} -
        x_*}{(\Id-Q)(x_{n+1} - x_*)} \label{apply.lemma}
    \end{align}
  \end{subequations}
  with \eqref{Qa.inverse.strongly.pos} and $\norm{\Id-Q}\leq 1$ in \eqref{consecutive.ys.5}. Note
  that \cite[Thm.~9.2-2, p.~466]{Kreyszig} is used in \eqref{UQinvU}. Moreover,
  $\theta := (1-\alpha) - {(2\alpha-1)^2}/[(1-\alpha)(\rho-1)]$ becomes positive for any
  $\rho > 1 + (2\alpha-1)^2/(1-\alpha)^2$ in \eqref{consecutive.ys.6},
  \eqref{Qa.inverse.strongly.pos} in \eqref{consecutive.ys.7}, \eqref{apply.lemma} in
  \eqref{consecutive.ys.8}, the fact
  $(\Id - Q)(x_{n+1} - x_*) = (\Id-T) x_{n+1} - (\Id-T) x_{*} = (\Id-T) x_{n+1}$ in
  \eqref{consecutive.ys.9}, and \eqref{Qa.inverse.strongly.pos} in \eqref{consecutive.ys.10}.

  Due to \eqref{y.as.regular}, the previous considerations suggest that there exists $C\in\RealPP$
  s.t.\ $\forall n$,
  \begin{alignat*}{2}
    \tfrac{C}{n+1}\,
    & \,\mathbin{\geq}\, && \tfrac{1}{n+1}
    \sum\nolimits_{\nu=0}^{n} \norm{y_{\nu+1} -
      y_{\nu}}_{\Theta}^2 \\
    & \,\mathbin{\geq}\, && \tfrac{1}{\rho(n+1)}
    \sum\nolimits_{\nu=0}^{n} \norm{Uv_{\nu+1}
      + \lambda [\nabla f(x_{\nu}) + \xi_{\nu+1}]}^2\\
    &&&\, + \tfrac{\theta}{n+1} \sum\nolimits_{\nu=0}^{n}
    \innerp{x_{\nu+1} - x_*}{(\Id-
      Q)(x_{\nu+1} - x_*)} \\
    & \,\mathbin{\geq}\, && \tfrac{1}{\rho(n+1)}
    \sum\nolimits_{\nu=0}^{n} \norm{Uv_{\nu+1} + \lambda [\nabla
      f(x_{\nu}) + \xi_{\nu+1}]}^2 +
    \tfrac{\theta(1-\alpha)^2}{n+1} \sum\nolimits_{\nu=0}^{n}
    \norm{(\Id - T) x_{\nu+1}}^2\,,
    \end{alignat*}
    which establishes the claim of \cref{thm:Oh} regarding the
    sequence of \eqref{FM-HSDM}. The proof of the claim with regards
    to the sequence of \eqref{FM-HSDM.g=0.original} follows the same
    steps as the previous one, but with the twist of replacing
    $\nabla f(x_n)$ by $\nabla f(T_{\alpha}x_n)$ and $g=0$.
\end{proof}

\begin{theorem}\label{thm:Oh.f=0}
  For the sequence $(x_n)_{n\in\mathbb{N}}$ of \eqref{FM-HSDM.f=0}, there exists
  $\xi_n\in \partial g(x_n)$, $\forall n$, s.t.\ for any $x_*\in \Fix T$,
  \begin{align*}
    & \innerp{x_{n+1} - x_*}{(\Id-Q)(x_{n+1} - x_*)} =
      O(\tfrac{1}{n+1})\,, \\
    & \norm{Uv_{n+1} + \lambda
      \xi_{n+1}}^2 = O(\tfrac{1}{n+1})\,, \\
    & \norm{(\Id - T) x_{n+1}}^2 =
      O(\tfrac{1}{n+1})\,.
  \end{align*}
\end{theorem}

\begin{proof}
  Define here $\Delta x_{n} := x_{n-1} - x_n$, $\Delta v_{n} := v_{n-1} - v_n$,
  $\Delta y_{n} := (\Delta x_{n}, \Delta v_{n})$, and $\Delta \xi_{n} := \xi_{n-1} - \xi_n$,
  $\forall n$. Under these definitions and in the case of $f=0$, \eqref{diff.consecutive.xi} yields
  \begin{align}
    & (1-2\alpha)(Q-\Id)(x_n -
      x_{n+1}) + Q_{\alpha} \left[ (x_n -
      x_{n+1}) - (x_{n-1} - x_{n})\right]
      \notag\\
    & \hspace{40ex} = 
      - U(v_n - v_{n+1}) - \lambda(\xi_n - \xi_{n+1}) \notag\\
    \Leftrightarrow\
    & (1-2\alpha)(Q-\Id)\Delta x_{n+1} +
      Q_{\alpha} (\Delta x_{n+1} - \Delta x_{n}) =
      -U\Delta v_{n+1} - \lambda\Delta \xi_{n+1}
      \notag\\
    \Leftrightarrow\
    & \lambda\Delta \xi_{n+1} =
      - U\Delta v_{n+1} - Q_{\alpha}
      (\Delta x_{n+1} - \Delta x_{n}) -
      (1-2\alpha)(Q-\Id)\Delta x_{n+1}\,.
      \label{basic.recursion.Delta.f=0}  
  \end{align}
  Moreover, \eqref{consecutive.v} suggests that $-\Delta v_{n+1} = (1-\alpha) U(x_{n+1} - x_*)$, and
  thus
  \begin{align}
    \tfrac{1}{1-\alpha}(\Delta v_{n+1} -
    \Delta v_{n}) = U\Delta x_{n+1}\,. \label{consecutive.Delta.v}
  \end{align}

  The monotonicity of $\partial g(\cdot)$, \eqref{basic.recursion.Delta.f=0},
  \eqref{consecutive.Delta.v}, and the definition of $\Theta$, introduced after \eqref{Qa.is.PD},
  imply that
  \begin{align}
    & 0\leq
      \innerp{\Delta x_{n+1}}{\lambda\Delta \xi_{n+1}}
      \notag\\
    \Leftrightarrow\
    & 0 \leq \innerp{\Delta x_{n+1}}{-U\Delta v_{n+1} -
      Q_{\alpha} (\Delta x_{n+1} - \Delta x_{n}) -
      (2\alpha-1)(\Id-Q)\Delta x_{n+1}}\notag\\
    \Leftrightarrow\ 
    & (2\alpha-1)
      \innerp{\Delta x_{n+1}}{(\Id-Q)\Delta x_{n+1}}  
      \notag\\
    & \hspace{10ex} \leq -
      \innerp{U\Delta x_{n+1}}{\Delta v_{n+1}} -
      \innerp{\Delta x_{n+1}}{Q_{\alpha}
      (\Delta x_{n+1} - \Delta x_{n})} \notag\\
    \Leftrightarrow\
    & (2\alpha-1)
      \innerp{\Delta x_{n+1}}{(\Id-
      Q)\Delta x_{n+1}} \notag\\
    & \hspace{10ex} \leq - \tfrac{1}{1-\alpha}
      \innerp{\Delta v_{n+1} -
      \Delta v_{n}}{\Delta v_{n+1}} - 
      \innerp{\Delta x_{n+1}}{Q_{\alpha}
      (\Delta x_{n+1} - \Delta x_{n})} \notag\\
    \Leftrightarrow\
    & (2\alpha-1) \innerp{\Delta x_{n+1}}{(\Id-
      Q)\Delta x_{n+1}} 
      \leq \innerp{\Delta y_{n+1}}{\Delta y_n - 
      \Delta y_{n+1}}_{\Theta} \notag\\ 
    \Leftrightarrow\
    & (2\alpha-1)
      \innerp{\Delta x_{n+1}}{(\Id-
      Q)\Delta x_{n+1}} 
      \leq \tfrac{1}{2} \left(\norm{\Delta y_n}^2_{\Theta} -
      \norm{\Delta y_{n+1}}^2_{\Theta} - \norm{\Delta y_n
      - \Delta y_{n+1}}^2_{\Theta}\right) \notag\\
    \Leftrightarrow\
    & 2(2\alpha-1)
      \innerp{\Delta x_{n+1}}{(\Id-Q)\Delta x_{n+1}} 
      + \norm{\Delta y_n - \Delta y_{n+1}}^2_{\Theta}
      \notag\\
    & \hspace{10ex} \leq \norm{\Delta y_n}^2_{\Theta} -
      \norm{\Delta y_{n+1}}^2_{\Theta}\,,
      \label{before.monotone.diff.y}
  \end{align}
  and due to $\alpha\geq 1/2$ as well as the positive-definiteness of $\Id - Q$,
  \eqref{before.monotone.diff.y} yields
  \begin{align}
    \norm{y_{n+1} - y_n}^2_{\Theta}
    \leq \norm{y_{n} - y_{n-1}}^2_{\Theta} \,, \quad
    \forall n\,. \label{monotone.diff.y}
  \end{align}

  Now, \eqref{y.as.regular} and \eqref{monotone.diff.y} imply that there exists $C>0$ s.t.\ for any
  $n$,
  \begin{align*}
    (n+1) \norm{y_{n+1} - y_n}_{\Theta}^2 \leq \sum\nolimits_{\nu=0}^n
    \norm{y_{\nu+1} - y_{\nu}}_{\Theta}^2 \leq C \,,
  \end{align*}
  and thus $\norm{y_{n+1} - y_n}_{\Theta}^2 \leq {C}/(n+1)$. This result applied to
  \eqref{consecutive.ys.7} and \eqref{consecutive.ys.10} establishes the claim of \cref{thm:Oh.f=0}.
\end{proof}

\section{Numerical tests}\label{sec:tests}

To validate the previous theoretical findings, tests are conducted on a simple scenario which is
motivated by~\cite[Prob.~4.1]{Iiduka.Math.Prog.15}. More elaborate tests, involving noisy real data,
are deferred to an upcoming publication where FM-HSDM is extended to a stochastic setting.

Given dimension $d\in\IntegerPP$, the real Euclidean space $\mathcalboondox{X}_0 := \Real^d$ is
considered. Upon defining the closed ball
$\mathcal{B}[\vect{u}_{\text{c}},r] := \Set{\vect{u}\in \mathcalboondox{X}_0 \given \norm{\vect{u}-
    \vect{u}_{\text{c}}}_2 \leq r}$, for center $\vect{u}_{\text{c}} \in\mathcalboondox{X}_0$ and
radius $r\in\RealPP$, let
$\mathcal{B}_1 := \mathcal{B}[\vect{u}_{\text{c}1},r_1] := \mathcal{B}[2\vect{e}_1,1]$ and
$\mathcal{B}_2 := \mathcal{B}[\vect{u}_{\text{c}2},r_2] := \mathcal{B}[\vect{0},2]$, where
$\vect{e}_1$ stands for the first column of the $d\times d$ identity matrix $\vect{I}_d$. In all
tests, $d := 10,000$. Let also $\vect{P}$ denote a $d\times d$ diagonal positive-definite matrix,
whose \textit{unique}\/ smallest entry $[\vect{P}]_{11}\leq 1$ is fixed at position $(1,1)$, and its
largest entry, placed at position $(d,d)$, is set to be equal to $10$. This setting is fixed across
all experiments. Each experiment in the sequel randomly draws numbers from the interval
$([\vect{P}]_{11},10)$, under the uniform distribution, and places them in the remaining $d-2$
entries of the diagonal of $\vect{P}$. Moreover, in all scenarios, parameter $\alpha$ of FM-HSDM is
set equal to $0.5$, since this value produced the best performance among all theoretically supported
values taken from $[0.5,1)$.

Along the lines of \cite[Prob.~4.1]{Iiduka.Math.Prog.15}, the following constrained quadratic
minimization task is considered:
\begin{align}
  \min_{\vect{u}\in \mathcal{B}_1\cap \mathcal{B}_2} \vect{u}^{\top}
  \vect{Pu} = 
  \min_{\vect{x} := (\vect{x}^{(1)}, \vect{x}^{(2)}, \vect{x}^{(3)})
  \in \mathcalboondox{X}_0^3 =: \mathcal{X}}
  &\ \tfrac{1}{2} \vect{x}^{(1)}{}^{\top}
    \vect{Px}^{(1)} + \iota_{\mathcal{B}_1}(\vect{x}^{(2)}) +
    \iota_{\mathcal{B}_2}(\vect{x}^{(3)}) \notag\\
  \text{s.to}
  &\ \vect{x}^{(1)} = \vect{x}^{(2)} = \vect{x}^{(3)}
    \,, \label{Iiduka.example} 
\end{align}
where
$\vect{x}:= (\vect{x}^{(1)}, \vect{x}^{(2)}, \vect{x}^{(3)}) := [\vect{x}^{(1)}{}^{\top},
\vect{x}^{(2)}{}^{\top}, \vect{x}^{(3)}{}^{\top}]^{\top} \in \mathcalboondox{X}_0^3$, and
$\mathcal{X} := \mathcalboondox{X}_0^3$ with inner product defined as the standard Euclidean
dot-vector product. The definition of the indicator functions
$\iota_{\mathcal{B}_1}, \iota_{\mathcal{B}_2}$ can be found in Sec.~\ref{sec:background}. Since
$\vect{P} \succ \vect{0}$ and the smallest entry of $\vect{P}$ is located at the $(1,1)$ position,
the unique solution to \eqref{Iiduka.example} is $\vect{x}_* :=
(\vect{e}_1,\vect{e}_1,\vect{e}_1)$. There are several ways of viewing \eqref{Iiduka.example} as a
special case of \eqref{the.problem}. For example,
$f(\vect{x}) := (1/2) \vect{x}^{(1)}{}^{\top} \vect{Px}^{(1)}$ and
$g(\vect{x}) := \iota_{\mathcal{B}_1}(\vect{x}^{(2)}) + \iota_{\mathcal{B}_2}(\vect{x}^{(3)})$, for
any $\vect{x} = (\vect{x}^{(1)}, \vect{x}^{(2)}, \vect{x}^{(3)})$. The Lipschitz coefficient of
$\nabla f$ is the largest entry of $\vect{P}$, \ie, $L = 10$, and
$\prox_{\lambda g}(\vect{x}) = (\vect{x}^{(1)}, P_{\mathcal{B}_1}(\vect{x}^{(2)}),
P_{\mathcal{B}_2}(\vect{x}^{(3)}))$. For any $\lambda\in\RealPP$, the proximal mapping of
$\iota_{\mathcal{B}_i}$ becomes $\prox_{\lambda \iota_{\mathcal{B}_i}} = P_{\mathcal{B}_i}$, where
$P_{\mathcal{B}_i}$ denotes the metric projection mapping onto the ball $\mathcal{B}_i$, given by
$P_{\mathcal{B}_i} (\vect{u}) = \vect{u}_{\text{c}i} + (\vect{u} - \vect{u}_{\text{c}i}) r_i/
\max\Set{\norm{\vect{u} - \vect{u}_{\text{c}i}}, r_i}$, for any $\vect{u}\in
\mathcalboondox{X}_0$. Furthermore,
$\mathcal{A} := \Set{\vect{x} = (\vect{x}^{(1)}, \vect{x}^{(2)}, \vect{x}^{(3)}) \in \mathcal{X}
  \given \vect{x}^{(1)} = \vect{x}^{(2)} = \vect{x}^{(3)}}$ is a closed linear subspace and thus an
affine set. According to \cref{ex:consensus}, a nonexpansive mapping $T$ with
$T\in \mathfrak{T}_{\mathcal{A}}$ is the metric projection mapping
$P_{\mathcal{A}}(\vect{x}) = (1/3)(\sum_{i=1}^3\vect{x}^{(i)}, \sum_{i=1}^3\vect{x}^{(i)},
\sum_{i=1}^3\vect{x}^{(i)})$,
$\forall \vect{x} := (\vect{x}^{(1)}, \vect{x}^{(2)}, \vect{x}^{(3)}) \in\mathcal{X}$.

\begin{figure}[!ht]
  \centering
  \includegraphics[width = 1\linewidth]{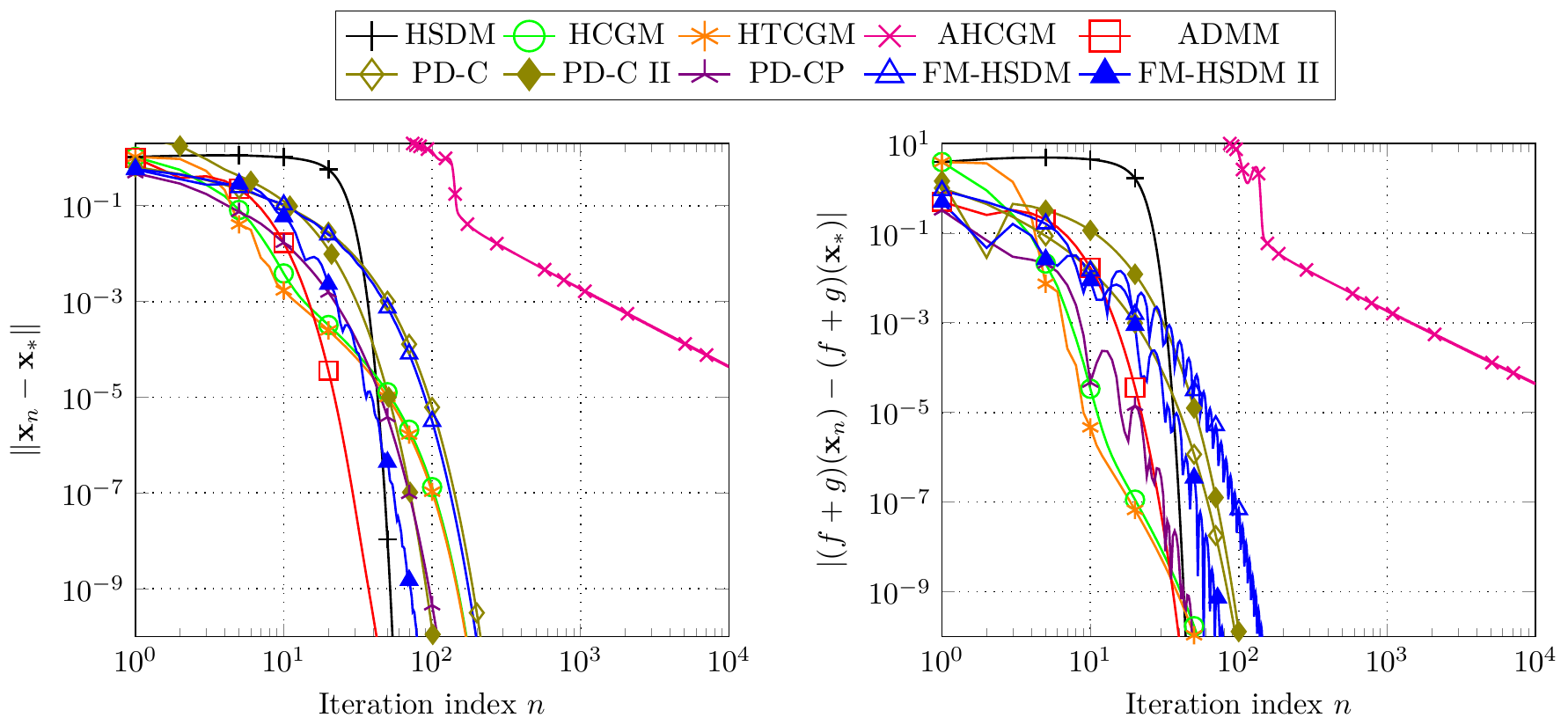}
  \caption{Deviation of the estimate $\vect{x}_n$ from the unique minimizer $\vect{x}_*$ of
    \eqref{Iiduka.example} and deviation of the loss-function value $(f+g)(\vect{x}_n)$ from the
    optimal $(f+g)(\vect{x}_*)$ vs.\ iteration index $n$, in the case where $[\vect{P}]_{11} := 1$
    and thus, the condition number of $\vect{P}$ equals $10$.}\label{fig:WellCond}
\end{figure}

Under the previous view of \eqref{Iiduka.example} as a special case of \eqref{the.problem}, FM-HSDM
is compared with other HSDM-family members such as the original HSDM~\cite{Yamada.HSDM.2001}, the
hybrid conjugate gradient method (HCGM)~\cite{Iiduka.Yamada.09}, the hybrid three-term conjugate
gradient method (HTCGM)~\cite{Iiduka.AMC.11} and the accelerated hybrid conjugate gradient method
(AHCGM)~\cite{Iiduka.Math.Prog.15}. Other competing methods include
ADMM~\cite{glowinski.marrocco.75, gabay.mercier.76, Bredies.Sun.DR.15, Sun.ADMM.16} in the standard
``scaled form''~\cite[\S3.1.1]{Boyd.admm}, and the primal-dual (PD) methods of \cite{Condat.JOTA.13}
(``CP-C'') and \cite{Chambolle.Pock.11} (``PD-CP''). Due to the strongly convex nature of
$\vect{x}^{(1)}{}^{\top} \vect{Px}^{(1)}$, the accelerated Alg.~2 of \cite{Chambolle.Pock.11} with
adaptive step sizes is used in ``PD-CP.''

To test \eqref{FM-HSDM.f=0} and address also the case where $[\vect{P}]_{11}\in\RealPP$ is close to
zero (\cf~Fig.~\ref{fig:IllCond}), \ie, $\vect{P}$ is ``nearly'' singular, $f$ and $g$ can be
considered in a different way than the previous setting: $f := 0$ and
$g(\vect{x}) := (1/2) \vect{x}^{(1)}{}^{\top} \vect{Px}^{(1)} +
\iota_{\mathcal{B}_1}(\vect{x}^{(2)}) + \iota_{\mathcal{B}_2}(\vect{x}^{(3)})$. Results that
associate with this take on \eqref{Iiduka.example} as a special case of \eqref{the.problem} and with
FM-HSDM are shown in the subsequent figures under the tag ``FM-HSDM~II.'' The PD method of
\cite{Condat.JOTA.13} is also adjusted to accommodate this view of \eqref{Iiduka.example}, and the
associated results are shown in Figs.\ \ref{fig:WellCond} and \ref{fig:IllCond} under the tag of
``PD-C II.'' It is worth stressing here that for this specific $g$, the proximal mapping
$\prox_{\lambda g}(\vect{x}) = ((\vect{I}_d + \lambda\vect{P})^{-1} \vect{x}^{(1)},
P_{\mathcal{B}_1}(\vect{x}^{(2)}), P_{\mathcal{B}_2}(\vect{x}^{(3)}))$. In other words, both PD-C II
and FM-HSDM~II use the resolvent $(\vect{I}_d + \gamma\vect{P})^{-1}$, for some adequate
$\gamma\in\RealPP$, similarly to the case of ADMM and PD-CP.

Parameters in all methods were tuned to yield best performance. In all tests, methods start from the
same initial point, randomly drawn from a unit-norm sphere and centered at the unique minimizer of
\eqref{Iiduka.example}. Each curve in Figs.~\ref{fig:WellCond} and \ref{fig:IllCond} is the uniform
average of the curves obtained from $100$ Monte-Carlo runs.

Fig.~\ref{fig:WellCond} considers $[\vect{P}]_{11} := 1$, and since the largest entry of $\vect{P}$
is $10$, the condition number of $\vect{P}$ is $10/1 = 10$. According to the developed theory,
parameter $\lambda$ of FM-HSDM is set equal to $\lambda := 0.99\cdot
2(1-\alpha)/L$. Fig.~\ref{fig:WellCond} shows that all methods, apart from AHCGM, perform
similarly. All HSDM-family members, excluding FM-HSDM~II, as well as PD-C score similar complexities
since they use $\nabla f$ once per iteration. On the contrary, ADMM, PD-CP, PD-C II and FM-HSDM~II
do not utilize $\nabla f$ but build around the resolvent
$(\vect{I}_d + \gamma\vect{P})^{-1}$~\cite{Bauschke.Combettes.book}, for appropriate
$\gamma\in\RealPP$.

The next set of tests follows that of Fig.~\ref{fig:WellCond}, but with
$[\vect{P}]_{11} := 10^{-2}$, which yields the condition number $10/10^{-2} = 10^3$ for
$\vect{P}$. As in the previous setting, parameter $\lambda$ of FM-HSDM is set equal to
$\lambda := 0.99\cdot 2(1-\alpha)/L$. Notice that since the theory which associates with HSDM, HCGM,
HTCGM and AHCGM offers guarantees of convergence in cases where $f$ is strongly convex, \ie,
$\vect{P}$ is positive definite, Fig.~\ref{fig:IllCond} shows that the performance of the
aforementioned algorithms degrades due to the fact that $\vect{P}$ was purposefully chosen to be
``nearly singular.'' Fig.~\ref{fig:IllCond} suggests also that FM-HSDM~II pays the price, by using
$(\vect{I}_d + \gamma\vect{P})^{-1}$, to achieve a performance similar to ADMM. The ``simpler''
FM-HSDM and PD-C, where no matrix inversion is required, face difficulties in following the ADMM,
FM-HSDM~II, PD-C~II and PD-CP curves for such an ill-conditioned minimization task. In theory, any
$\lambda\in \RealPP$ can serve FM-HSDM due to the fact that $f:=0$. In practice, tuning is
necessary, and the value of $\lambda=100$ is used. Fig.~\ref{fig:IllCond} underlines the flexibility
of FM-HSDM, where mappings and computational complexity can be tuned to suit the minimization task
at hand.

\begin{figure}[!ht]
  \centering
  \includegraphics[width = 1\linewidth]{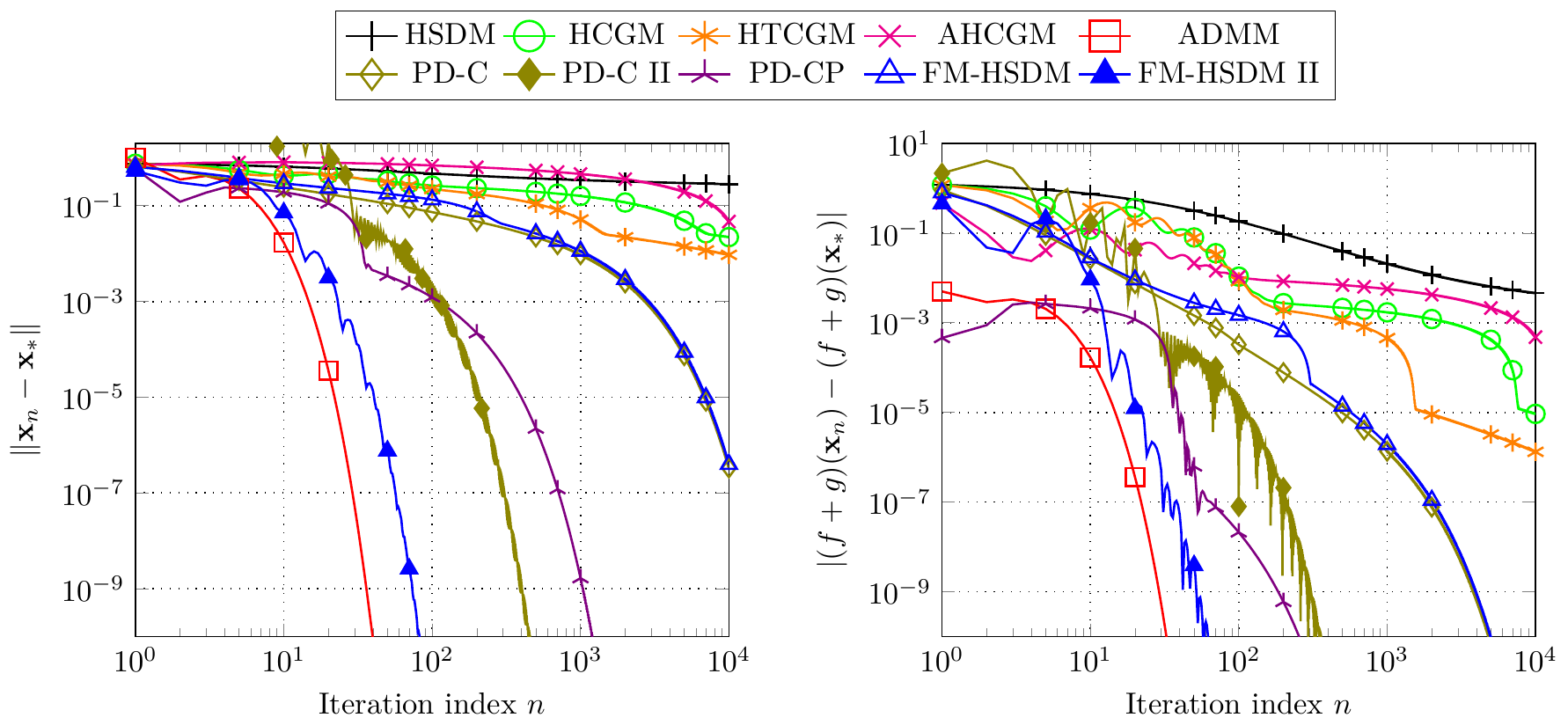}
  \caption{This setting follows that of Fig.~\ref{fig:WellCond}, but with
    $[\vect{P}]_{11} := 10^{-2}$, which results in a condition number $10/10^{-2} = 10^3$ for
    $\vect{P}$.}\label{fig:IllCond}
\end{figure}

To compare \eqref{FM-HSDM.g=0} with \eqref{FM-HSDM.g=0.original}, tests are performed on the
following task:
\begin{align}
  \min_{\vect{x}\in \mathcalboondox{X}_0}\ \vect{x}^{\top} \vect{Px}\quad \text{s.to}\ \vect{x}\in
  \mathcal{V} := \Set{\vect{u}\in \mathcalboondox{X}_0 \given \vect{e}_1^{\top} \vect{u} = 1}\,, \label{task.4.FISTA}
\end{align}
where $\mathcalboondox{X}_0$, $\vect{P}$ and $\vect{e}_1$ were defined earlier in this section, and
$\mathcal{V}$ is a hyperplane; hence, an affine set. Due to the construction of $\vect{P}$, it can
be verified that the minimizer of \eqref{task.4.FISTA} is $\vect{x}_* = \vect{e}_1$. Both
\eqref{FM-HSDM.g=0} and \eqref{FM-HSDM.g=0.original} are employed with $T := P_{\mathcal{V}}$, where
$P_{\mathcal{V}}$ stands for the metric projection mapping onto $\mathcal{V}$
(\cf~\cref{ex:proj.hyperplane}). The results of the application of \eqref{FM-HSDM.g=0} and
\eqref{FM-HSDM.g=0.original} are illustrated in Figs.~\ref{fig:FISTA.WellCond} and
\ref{fig:FISTA.IllCond} as ``FM-HSDM'' and ``FM-HSDM~III,'' respectively.

The state-of-the-art FISTA method~\cite[(4.1)--(4.3)]{FISTA} is also employed here after recasting
\eqref{task.4.FISTA} as
$\min_{\vect{x}\in \mathcalboondox{X}_0} (1/2) \vect{x}^{\top} \vect{Px} +
\iota_{\mathcal{V}}(\vect{x})$, where $\iota_{\mathcal{V}}$ stands for the indicator function of
$\mathcal{V}$. This take on \eqref{task.4.FISTA} opens also the door for \eqref{FM-HSDM.f=0}, under
$g(\vect{x}^{(1)}, \vect{x}^{(2)}) := g_1(\vect{x}^{(1)}) + g_2(\vect{x}^{(2)})$,
$\forall (\vect{x}^{(1)}, \vect{x}^{(2)}) \in \mathcalboondox{X}_0^2$, with
$g_1(\vect{x}^{(1)}) := (1/2) \vect{x}^{(1)}{}^{\top} \vect{P} \vect{x}^{(1)}$,
$\forall \vect{x}^{(1)}$, $g_2 := \iota_{\mathcal{V}}$, and
$\mathcal{A} := \Set{(\vect{x}^{(1)}, \vect{x}^{(2)}) \in \mathcalboondox{X}_0^2 \given
  \vect{x}^{(1)} = \vect{x}^{(2)}}$, similarly to the application of FM-HSDM~II to
\eqref{Iiduka.example}. Tag ``FM-HSDM~II'' is used also in Figs.~\ref{fig:FISTA.WellCond} and
\ref{fig:FISTA.IllCond} to indicate the performance of \eqref{FM-HSDM.f=0}. It is worth noticing
that \eqref{FM-HSDM.g=0.original} can be applied to \eqref{task.4.FISTA}, but not to
\eqref{Iiduka.example}, due to the limitation of $g=0$ in \eqref{FM-HSDM.g=0.original}. Moreover,
FISTA cannot be applied ``innocently'' to \eqref{Iiduka.example}, since its proximal-mapping
step~\cite[(4.1)]{FISTA} amounts to identifying the metric projection of a point onto the
intersection $\mathcal{B}_1\cap \mathcal{B}_2$, which is itself the outcome of an iterative
procedure, such as the projections-onto-convex-sets (POCS) algorithm~\cite[Cor.~5.23,
p.~84]{Bauschke.Combettes.book}. Such computational issues would have been surmounted, had FISTA the
ability to employ the convenient tool of ``splitting of variables,'' which is embedded in ADMM and
primal-dual methods, as well as in FM-HSDM via the affine constraint $\mathcal{A}$
[\cf~\eqref{Iiduka.example}].

The way to construct $\vect{P}$ is identical to that in the case of
\eqref{Iiduka.example}. Parameters $\alpha\, (:= 0.5)$ and $\lambda$ for FM-HSDM and FM-HSDM~II are
identical to those of the \eqref{Iiduka.example} scenario. The step size $\lambda'$ of FM-HSDM~III
is defined as $\lambda' := 0.99\cdot 2(1-\alpha)^2/L$, according to the specifications dictated by
\cref{thm:plain.vanilla.HSDM}. In all tests, methods start from the same initial point, randomly
drawn from a unit-norm sphere and centered at the unique minimizer of \eqref{task.4.FISTA}. Results
are depicted in Figs.~\ref{fig:FISTA.WellCond} and \ref{fig:FISTA.IllCond}, where each curve is the
uniform average of the curves obtained from $100$ Monte-Carlo runs. FM-HSDM~III demonstrates slower
convergence speed than that of the rest of the methods. Note that FISTA guarantees
\textit{optimal}\/ convergence rate $|(f+g)(\vect{x}_n) - (f+g)(\vect{x}_*)| = \mathcal{O}[1/(n+1)^2]$~\cite[Thm.~4.4]{FISTA}. The fast convergence speed of
FM-HSDM~II becomes prominent in the case of Fig.~\ref{fig:FISTA.IllCond}, where $\vect{P}$ suffers a
large condition number.

\begin{figure}[!ht]
  \centering
  \includegraphics[width=1\linewidth]{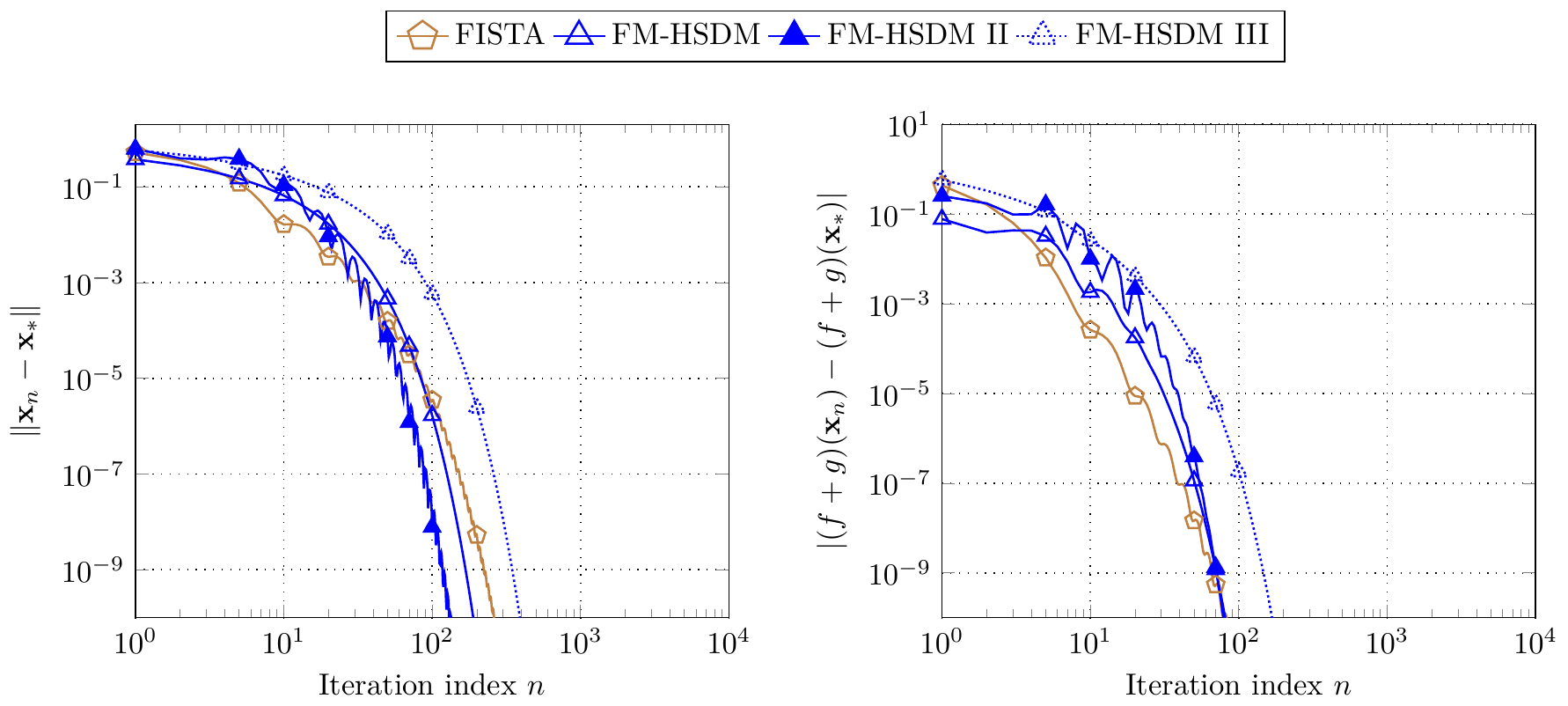}
  \caption{Deviation of the estimate $\vect{x}_n$ from the unique minimizer $\vect{x}_*$ of
    \eqref{task.4.FISTA} and deviation of the loss-function value $(f+g)(\vect{x}_n)$ from the
    optimal $(f+g)(\vect{x}_*)$ vs.\ iteration index $n$, in the case where $[\vect{P}]_{11} := 1$
    and thus, the condition number of $\vect{P}$ equals $10$.}\label{fig:FISTA.WellCond}
\end{figure}

\begin{figure}[!ht]
  \centering
  \includegraphics[width=1\linewidth]{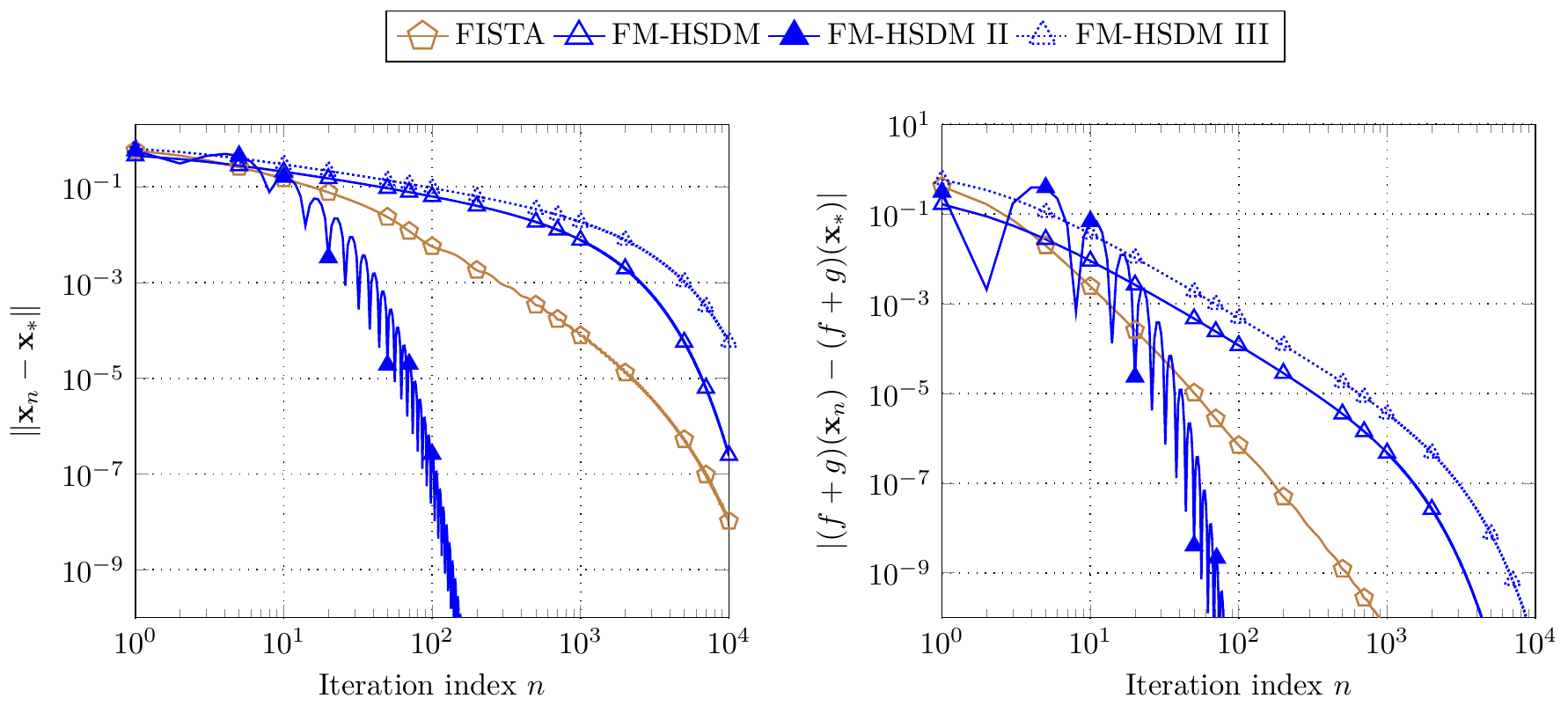}
  \caption{This setting follows that of Fig.~\ref{fig:FISTA.WellCond}, but with
    $[\vect{P}]_{11} := 10^{-2}$, which results in a condition number $10/10^{-2} = 10^3$ for
    $\vect{P}$.}\label{fig:FISTA.IllCond}
\end{figure}

\section{Conclusions}\label{sec:conclusions}

This paper introduced the \textit{Fej\'{e}r-monotone hybrid steepest descent method}\/ (FM-HSDM) for
solving affinely constrained composite minimization tasks in real Hilbert spaces. Only differential
and proximal mappings are used to provide low-computational-complexity recursions with enhanced
flexibility towards the accommodation of affine constraints. The advocated scheme enjoys Fej\'{e}r
monotonicity, a constant step-size parameter across iterations, and minimal presuppositions on the
smooth and non-smooth loss functions to establish weak, and under certain hypotheses, strong
convergence to an optimal point. Results on the rate of convergence of the FM-HSDM's sequence of
estimates were also presented. Numerical tests on synthetic data were also demonstrated to validate
the theoretical findings. Thorough tests on noisy real data, which showcase the flexibility of the
family of mappings $\mathfrak{T}_{\mathcal{A}}$ [\cf~\eqref{T.family}] in a stochastic setting, are
deferred to an upcoming publication.

\section*{Funding}\label{sec:Funds}

This work was partially supported by the NSF awards 1514056, 1525194 and 1718796.

\begin{appendices}
  \crefalias{section}{appsec}

  \section{}\label{sec:appendix}

  Several special cases of $\mathcal{A}$, of large interest in optimization tasks, together with
  members of the family of mappings $\mathfrak{T}_{\mathcal{A}}$ follow.

  \begin{example}\label{ex:consensus}
    Given a Hilbert space $\mathcalboondox{X}_0$ and $I\in\IntegerPP$, consider the Hilbert space
    $\mathcal{X} := \mathcalboondox{X}_0 \times \mathcalboondox{X}_0 \times \ldots \times
    \mathcalboondox{X}_0 = \Set{x := (x^{(1)}, x^{(2)}, \ldots, x^{(I)}) \given
      x^{(i)}\in\mathcalboondox{X}_0, \forall i\in\Set{1,\ldots, I}}$, equipped with the inner product
    $\innerp{x}{x'}_{\mathcal{X}} := \sum_{i=1}^I \innerp{x^{(i)}}{x'{}^{(i)}}$. Then, upon defining
    the (closed) linear subspace
    $\mathcal{S} := \Set{x\in \mathcal{X} \given x^{(1)} = x^{(2)} = \cdots = x^{(I)}}$, the metric
    projection mapping onto $\mathcal{S}$ satisfies
    \begin{align}
      P_{\mathcal{S}} (x) =
      \left(\tfrac{1}{I}\sum\nolimits_{i=1}^I x^{(i)},
      \tfrac{1}{I}\sum\nolimits_{i=1}^I x^{(i)}, \ldots,
      \tfrac{1}{I}\sum\nolimits_{i=1}^I x^{(i)}
      \right)\,, \quad \forall
      x \in \mathcal{X}\,, \label{proj.consensus}
    \end{align}
    and $P_{\mathcal{S}}\in \mathfrak{T}_{\mathcal{S}}$. 
  \end{example}

  \begin{proof}
    Formula \eqref{proj.consensus} can be easily derived by applying \cref{ex:projection} to the
    special cases of $\mathcal{X}$ and $\mathcal{S}$:
    $\norm{x - P_{\mathcal{S}} x}_{\mathcal{X}}^2 = \min_{z\in \mathcalboondox{X}_0} \sum_{i=1}^I
    \norm{x^{(i)} - z}^2$. Then, claim $P_{\mathcal{S}}\in \mathfrak{T}_{\mathcal{S}}$ is established
    by noticing that $\mathcal{S}$ is a closed affine set and by \cref{prop:T(A).family}.
  \end{proof}

  \begin{example}[Metric projection mapping onto a
    hyperplane]\label{ex:proj.hyperplane}
    For a non-zero $a\in\mathcal{X}$ and a real number $b$, consider the metric projection mapping
    onto the hyperplane $\mathcal{V} := \Set{x\in\mathcal{X} \given \innerp{a}{x} = b}$~\cite[(3.11),
    p.~49]{Bauschke.Combettes.book}
    \begin{align}
      P_{\mathcal{V}} = \Id -
      \tfrac{\innerp{a}{\Id}}{\norm{a}^2} a + \tfrac{b}{\norm{a}^2}
      a\,. \label{proj.hyperplane}
    \end{align}
    Then, $P_{\mathcal{V}}\in \mathfrak{T}_{\mathcal{V}}$.
  \end{example}

  \begin{proof}
    The claim follows by the observations that $\mathcal{V}$ is a closed affine set,
    $(b/\norm{a}^2)a\in \mathcal{V}$, and by introducing
    $\mathcal{V} = \Set{x\in\mathcal{X} \given \innerp{a}{x}=0}$, with
    $P_{\mathcal{V}} = \Id - \tfrac{\innerp{a}{\Id}}{\norm{a}^2} a$ and
    $P_{\mathcal{V}}[(b/\norm{a}^2)a] = 0$, in \cref{prop:T(A).family}.
  \end{proof}

  As the following fact states, affine sets obtain a specific form in Euclidean spaces.

  \begin{fact}[\protect{\cite[Thm.~1.4, p.~5]{Rockafellar.convex.analysis}}] \label{fact:Rockafellar}
    Given $\vect{b}\in \Real^M$ ($M\in\IntegerPP$) and $\vect{A}\in \Real^{M\times D}$
    ($D\in\IntegerPP$) the set $\Set{\vect{x}\in\Real^D \given \vect{Ax}=\vect{b}}$, if non-empty, is
    an affine set. Moreover, every affine set in $\mathcal{X}:= \Real^D$ can be represented in this
    way.
  \end{fact}

  Motivated by the previous fact and aiming at an algorithmic scheme with wide applicability in
  Euclidean spaces, where most of the minimization problems reside, the following example and
  proposition offer a view of affine sets via \textit{least-squares}\/ (LS) tasks and nonexpansive
  mappings.

  \begin{example}[Affinely constrained LS in Euclidean
    spaces]\label{ex:aff.constraint.LS} For 
    vector $\vect{b}$ and matrix $\vect{A}$ of \cref{fact:Rockafellar}, consider the following LS
    solution set \cite[Prop.~3.25, p.~50]{Bauschke.Combettes.book}:
    \begin{align}
      \mathcal{A}
      & := \Argmin\nolimits_{\vect{x}\in\Real^D} \tfrac{1}{2}
        \norm{\vect{Ax} - \vect{b}}^2 =
        \Set*{\vect{x}\in\Real^D \given \vect{A}^{\top}\vect{A}
        \vect{x} = \vect{A}^{\top} \vect{b}}\,. \label{LS}
    \end{align}
    Now, considering the $D\times 1$ vectors $\Set{\bm{\alpha}_m}_{m=1}^M$, defined by the rows of
    $\vect{A}$, \ie, $[\bm{\alpha}_1, \bm{\alpha}_2, \ldots, \bm{\alpha}_M] := \vect{A}^{\top}$, as
    well as the $D\times 1$ vectors $\Set{\vect{g}_d}_{d=1}^D$ defined via
    $[\vect{g}_1, \ldots, \vect{g}_D] := \vect{G}$, where $\vect{G}:= \vect{A}^{\top} \vect{A}$ and
    $\vect{c}:= [c_1, c_2, \ldots, c_D]^{\top} := \vect{A}^{\top}\vect{b}$, let the hyperplanes
    $\mathcal{A}_m := \Set{\vect{x}\in\Real^D \given \innerp{\bm{\alpha}_m} {\vect{x}} = b_m}$,
    $(m=1,\ldots,M)$, as well as
    $\mathcal{G}_d := \Set{\vect{x}\in\Real^D \given \innerp{\vect{g}_d}{\vect{x}} = c_d}$,
    $(d=1,\ldots,D)$,
    % \begin{alignat*}{2}
    %   \mathcal{A}_m & := \Set*{\vect{x}\in\mathcal{X} \given
    %   \innerp{\bm{\alpha}_m} {\vect{x}} = b_m} \,, \qquad &&
    %   (m=1,\ldots,M)\,,\\
    %   \mathcal{G}_d & :=
    %   \Set*{\vect{x}\in\mathcal{X} \given
    %   \innerp{\vect{g}_d}{\vect{x}} = c_d} \,, &&
    %   (d=1,\ldots,D)\,,
    % \end{alignat*}
    with associated metric projection mappings $P_{\mathcal{A}_m}$ and $P_{\mathcal{G}_d}$,
    respectively [\cf~\eqref{proj.hyperplane}]. Then, any of the following mappings, with $\dagger$
    denoting the Moore-Penrose pseudoinverse operation~\cite{Ben.Israel},
    \begin{subequations}
      \begin{numcases}{T=}
        \left(\vect{I} - \tfrac{\mu}{\varrho}
          \vect{A}^{\top} \vect{A} \right)\Id +
        \tfrac{\mu}{\varrho}
        \vect{A}^{\top}\vect{b}\,, &
        $\varrho\geq \norm{\vect{A}}^2$,
        $\mu\in(0,1]$\,, \label{affine.set.AtA}\\
        (\vect{I} - \vect{A}^{\top}\vect{A}^{\dagger\top}) \Id +
        \vect{A}^{\dagger}\vect{b}\,, \label{affine.set.proj}\\
        (\vect{I} - \vect{G}\vect{G}^{\dagger}) \Id +
        \vect{G}^{\dagger}
        \vect{A}^{\top}\vect{b}\,, & \label{affine.set.proj.via.gram}\\ 
        (\vect{I} + \gamma\vect{A}^{\top}\vect{A})^{-1} \Id + \gamma 
        (\vect{I} + \gamma\vect{A}^{\top}\vect{A})^{-1}\vect{A}^{\top}
        \vect{b}\,, &
        $\gamma\in\RealPP$\,, \label{affine.set.prox}\\
        (1-\beta) \Id + \beta \sum\nolimits_{m=1}^M
        \tfrac{\norm{\bm{\alpha}_m}^2}
        {\norm{\vect{A}}_{\text{F}}^2}
        P_{\mathcal{A}_m}\,, &
        $\beta\in(0,1]$\,, \label{affine.set.grad}\\ 
        (1-\theta) \Id + \theta \sum\nolimits_{d=1}^D \omega_d
        P_{\mathcal{G}_d}\,, & $\begin{cases} \theta\in(0,1]\,,
          \,\omega_d\in(0,1)\,, \\ \sum\nolimits_{d=1}^D \omega_d =
          1\,, \end{cases}$
        \label{affine.set.normal}
      \end{numcases}
    \end{subequations}
    satisfies $T\in \mathfrak{T}_{\mathcal{A}}$.

    Further, given also the $M_0 \times 1$ ($M_0\in \IntegerPP$) vector $\vect{b}_0$, the
    $M_0\times D$ matrix $\vect{A}_0$, let the non-empty affine constraint set
    $\mathcal{K} := \Set{\vect{x}\in\Real^D \given \vect{A}_0 \vect{x} = \vect{b}_0}$, with metric
    projection mapping
    $P_{\mathcal{K}} = (\vect{I} - \vect{A}_0^{\top}\vect{A}_0^{\dagger\top}) \Id +
    \vect{A}_0^{\dagger}\vect{b}_0$~\cite[Prop.~3.17, p.~47]{Bauschke.Combettes.book}. Then, according
    to~\cite[Ex.~34, p.~120]{Ben.Israel},
    \begin{align}
      & \vect{x}\in \mathcal{A}_{\mathcal{K}} :=
        \Argmin\nolimits_{\vect{z}\in\mathcal{K}}\,
        \tfrac{1}{2} \norm{\vect{A}\vect{z} - \vect{b}}^2 \notag\\
      & \Leftrightarrow \exists \bm{\mu} \in\Real^{M_0}\
        \text{s.t.}\ (\vect{x}, 
        \bm{\mu}) \in \overline{\mathcal{A}} :=
        \Set*{(\vect{x}', \bm{\mu}')\in\Real^D\times \Real^{M_0}
        \given \overbrace{\left[\begin{smallmatrix}
              \vect{A}^{\top}\vect{A}
              & \vect{A}_0^{\top} \\
              \vect{A}_0 & \vect{0}
            \end{smallmatrix}\right]}^{\vect{L}\,:=}
                           \left[\begin{smallmatrix} \vect{x}' \\
                               \bm{\mu}'
                             \end{smallmatrix}\right] = 
      \overbrace{\left[\begin{smallmatrix}
            \vect{A}^{\top} 
            \vect{b}\\
            \vect{b}_0 
          \end{smallmatrix}\right]}^{\vect{e}\,:=}}
      \,, \label{constrained.LS.characterize} 
    \end{align}
    or, in other words, $\mathcal{A}_{\mathcal{K}} = \Pi_{\Real^D} \overline{\mathcal{A}}$, where
    $\Pi_{\Real^D}$ denotes the mapping
    $\Pi_{\Real^D}: \Real^D\times \Real^{M_0} \to \Real^D: (\vect{x}, \bm{\mu}) \mapsto
    \vect{x}$. Define also the $(D+M_0)\times 1$ vectors
    $[\vect{l}_1, \ldots, \vect{l}_{D+M_0}] := \vect{L}$, as well as the hyperplanes
    $\mathcal{L}_d := \Set{(\vect{x}', \bm{\mu}')\in\Real^D \times \Real^{M_0} \given
      \innerp{\vect{l}_d}{(\vect{x}', \bm{\mu}')} = e_d}$,
    % \begin{align*}
    %   \mathcal{L}_d := \Set*{(\vect{x},
    %   \bm{\mu})\in\mathcal{X}\times \Real^{M_0} \given
    %   \innerp{\vect{l}_d}{(\vect{x}, \bm{\mu})} = e_d}\,,
    % \end{align*}
    with $P_{\mathcal{L}_d}$ denoting the associated metric projection mapping
    [\cf~\eqref{proj.hyperplane}]. Then, any of the following mappings
    $\overline{T}: \Real^{D+M_0}\to \Real^{D+M_0}$:
    \begin{subequations}
      \begin{numcases}{\overline{T}=}
        \left( \vect{I} -
          \tfrac{\overline{\mu}}{\overline{\varrho}}
          \vect{L}^{\top} \vect{L} \right)\Id +
        \tfrac{\overline{\mu}}{\overline{\varrho}}
        \vect{L}^{\top}\vect{b}\,, & $\overline{\varrho}\geq
        \norm{\vect{L}}^2$, $\overline{\mu}\in(0,1]$\,,
        \label{constrained.affine.set.LtL}\\
        \left(\vect{I} - \vect{L}^{\top}\vect{L}^{\dagger\top}
        \right) \Id + \vect{L}^{\dagger}\vect{e}\,,
        \label{constrained.affine.set.proj}\\
        \left(\vect{I} +
          \overline{\gamma}\vect{L}^{\top}\vect{L} \right)^{-1} \Id
        + \overline{\gamma} \left(\vect{I} +
          \overline{\gamma}\vect{L}^{\top}\vect{L}\right)^{-1}
        \vect{L}^{\top} \vect{e}\,, & $\overline{\gamma}\in
        \RealPP$\,,
        \label{constrained.affine.set.prox}\\
        (1 - \overline{\theta}) \Id + \overline{\theta}
        \sum\nolimits_{d=1}^{D+M_0} \overline{w}_d
        P_{\mathcal{L}_d}\,, &
        $\begin{cases} \overline{\theta} \in (0,1]\,,
          \,\overline{\omega}_d\in(0,1)\,, \\
          \sum\nolimits_{d=1}^{D+M_0} \overline{\omega}_d =
          1\,, \end{cases}$ \label{constrained.LS.intersection.hyperplanes}
      \end{numcases}
      satisfies
      $\overline{T}\in
      \mathfrak{T}_{\overline{\mathcal{A}}}$. Moreover,
      the mapping $T:\Real^D\to \Real^D$, defined by
      \begin{align}
        T := (1- \overline{\beta}) P_{\mathcal{K}} +
        \overline{\beta} 
        P_{\mathcal{K}} \sum\nolimits_{m=1}^M
        \tfrac{\norm{\bm{\alpha}_m}^2} 
        {\norm{\vect{A}}_{\text{F}}^2} P_{\mathcal{A}_m}
        P_{\mathcal{K}}\,, \quad \overline{\beta}\in (0,1]\,,
        \label{constrained.affine.set.grad}
      \end{align}
    \end{subequations}%
    satisfies $T\in \mathfrak{T}_{\mathcal{A}_{\mathcal{K}}}$.
  \end{example}

  \begin{proof}
    For $\delta\in\RealPP$, define
    \begin{align}
      \varphi_{\delta}(\vect{x}) := \tfrac{1}{2\delta}
      \norm{\vect{Ax} - \vect{b}}^2, \quad
      \forall\vect{x}\in\Real^D\,, \label{phi.delta}
    \end{align}
    and verify that
    $\nabla\varphi_{\delta}= (1/\delta)\vect{A}^{\top} \vect{A}\Id - (1/\delta) \vect{A}^{\top}
    \vect{b}$. According to \eqref{LS}, all points $\vect{x}\in\Real^D$ s.t.\
    $\nabla\varphi_{\delta}(\vect{x}) = \vect{0}$ constitute $\mathcal{A}$. Moreover, for any
    $\varrho\geq \norm{\vect{A}}^2/\delta$,
    $\norm{\nabla\varphi_{\delta}(\vect{x}) - \nabla\varphi_{\delta}(\vect{x}')} \leq
    (\norm{\vect{A}}^2/\delta) \norm{\vect{x} - \vect{x}'} \leq \varrho \norm{\vect{x} - \vect{x}'}$,
    $\forall\vect{x}, \vect{x}'\in\Real^D$, since
    $\norm{\vect{A}^{\top} \vect{A}} = \norm{\vect{A}}^2$. In other words, $\nabla\varphi_{\delta}$ is
    $\varrho$-Lipschitz continuous, which, according to the Baillon-Haddad
    theorem~\cite{Baillon.Haddad}, \cite[Cor.~18.16, p.~270]{Bauschke.Combettes.book}, is equivalent
    to that $(1/\varrho) \nabla\varphi_{\delta}$ is firmly nonexpansive iff
    $\Id - (1/\varrho) \nabla\varphi_{\delta}$ is firmly nonexpansive [\cf~\cref{ex:firmly.nonexp}]
    with fixed-point set equal to $\mathcal{A}$. By utilizing once again \cref{ex:firmly.nonexp},
    $R := 2[\Id - (1/\varrho ) \nabla\varphi_{\delta}] - \Id$ is nonexpansive, and for any
    $\zeta\in (0,1]$,
    $R' := \zeta R + (1-\zeta) \Id = \Id - (2\zeta/\varrho)\nabla\varphi_{\delta} = \{\vect{I}-
    [2\zeta/(\varrho\delta)]\vect{A}^{\top} \vect{A}\}\Id + [2\zeta/(\varrho\delta)] \vect{A}^{\top}
    \vect{b}$ is nonexpansive with $\Fix(R') = \mathcal{A}$. Due to the nonexpansiveness of $R'$,
    $\norm{\vect{I}- [2\zeta/(\varrho\delta)]\vect{A}^{\top} \vect{A}} \leq 1$
    (\cf~\cref{fact:affine.nonexp.T}). Constraining $\zeta\in (0,1/2]$ guarantees that
    $\vect{I}- [2\zeta/(\varrho\delta)]\vect{A}^{\top} \vect{A} \succeq \vect{0}$. By defining
    $\mu:= 2\zeta$ and $\delta:=1$, the claim regarding \eqref{affine.set.AtA} is established.

    The metric projection mapping $P_{\ker\vect{A}}$ onto $\ker\vect{A}$ is
    $P_{\ker\vect{A}} = (\vect{I}- \vect{A}^{\top}\vect{A}^{\top\dagger})\Id$~\cite[Prop.~3.28(iii),
    p.~51]{Bauschke.Combettes.book}. Since $\mathcal{A} = \ker\vect{A} + \vect{A}^{\dagger}\vect{b}$
    \cite[Prop.~3.28(i), p.~51]{Bauschke.Combettes.book}, \cite[Prop.~3.17,
    p.~47]{Bauschke.Combettes.book} suggests that the metric projection mapping $P_{\mathcal{A}}$ onto
    $\mathcal{A}$ becomes
    $P_{\mathcal{A}} = P_{\ker\vect{A}} + \vect{A}^{\dagger}\vect{b} - P_{\ker\vect{A}}
    (\vect{A}^{\dagger}\vect{b}) = P_{\ker\vect{A}} + \vect{A}^{\dagger}\vect{b}$, due to
    $P_{\ker\vect{A}} (\vect{A}^{\dagger}\vect{b}) = \vect{0}$~\cite[Prop.~3.28(i),
    p.~51]{Bauschke.Combettes.book}. Hence, \eqref{affine.set.proj} is an immediate consequence of
    \cref{prop:T(A).family}. By \cite[Ex.~18(d), p.~49]{Ben.Israel},
    $\vect{A}^{\top}\vect{A}^{\top\dagger} = \vect{A}^{\top}\vect{A} (\vect{A}^{\top}
    \vect{A})^{\dagger} = \vect{GG}^{\dagger}$ and
    $\vect{A}^{\dagger} \vect{b} = (\vect{A}^{\top} \vect{A})^{\dagger}\vect{A}^{\top} \vect{b} =
    \vect{G}^{\dagger} \vect{A}^{\top} \vect{b}$. Hence, \eqref{affine.set.proj.via.gram} follows
    easily from \eqref{affine.set.proj}.

    Now, for any $\gamma'\in \RealPP$,
    $\prox_{\gamma' \varphi_{\delta}} = (\vect{I} + (\gamma'/\delta) \vect{A}^{\top} \vect{A})^{-1}\Id
    + (\gamma'/\delta) (\vect{I} + (\gamma'/\delta) \vect{A}^{\top} \vect{A})^{-1} \vect{A}^{\top}
    \vect{b}$. Setting $\gamma:= \gamma'/\delta$, the nonexpansiveness of
    $\prox_{\gamma\delta\varphi_{\delta}}$, stated by \cref{ex:prox}, suggests that
    $\norm{(\vect{I} + \gamma\vect{A}^{\top} \vect{A})^{-1}}\leq 1$ (\cf~\cref{fact:affine.nonexp.T}),
    and that $\Fix (\prox_{\gamma\delta\varphi_{\delta}}) = \mathcal{A}$. Due also to the fact that
    $(\vect{I} + \gamma\vect{A}^{\top} \vect{A})^{-1}$ is positive, the claim regarding
    \eqref{affine.set.prox} is established.
    
    Let $\delta:= \norm{\vect{A}}_{\text{F}}^2$ in \eqref{phi.delta}, so that
    \begin{align*}
      \varphi_{\norm{\vect{A}}_{\text{F}}^2}(\vect{x})
      & = \tfrac{1}{2\norm{\vect{A}}_{\text{F}}^2} \norm{\vect{Ax} -
        \vect{b}}^2 = \tfrac{1}{2\norm{\vect{A}}_{\text{F}}^2}
        \sum\nolimits_{m=1}^M \left(\innerp{\bm{\alpha}_m}{\vect{x}} -
        b_m\right)^2 \\
      & = \tfrac{1}{2} \sum\nolimits_{m=1}^M
        \tfrac{\norm{\bm{\alpha}_m}^2}{\norm{\vect{A}}_{\text{F}}^2}
        \norm{\vect{x} - P_{\mathcal{A}_m}(\vect{x})}^2 = \tfrac{1}{2}
        \sum\nolimits_{m=1}^M w_m \norm{\vect{x} -
        P_{\mathcal{A}_m}(\vect{x})}^2\,,
    \end{align*}
    where the explicit expression of $P_{\mathcal{A}_m}$ is given in \eqref{proj.hyperplane}, and the
    non-negative weights $\Set{w_m := \norm{\bm{\alpha}_m}^2/ \norm{\vect{A}}_{\text{F}}^2}_{m=1}^M$
    satisfy $\sum_{m=1}^M w_m = 1$. It can be also verified by the Fr\'{e}chet-gradient definition
    \cite[Def.~2.45, p.~38]{Bauschke.Combettes.book} that
    $\nabla \norm{(\Id - P_{\mathcal{A}_m})\vect{x}}^2 = 2(\Id - P_{\mathcal{A}_m})\vect{x}$, which
    yields
    \begin{align*}
      \nabla\varphi_{\norm{\vect{A}}_{\text{F}}^2}
      & = \sum\nolimits_{m=1}^M w_m (\Id -
        P_{\mathcal{A}_m}) = \Id - \sum\nolimits_{m=1}^M w_m
        P_{\mathcal{A}_m}\,.
    \end{align*}
    Hence, all minimizers of $\varphi_{\norm{\vect{A}}_{\text{F}}^2}$, \ie, $\mathcal{A}$, constitute
    the fixed-point set of $\sum\nolimits_m w_m P_{\mathcal{A}_m}$, which is equal to the fixed-point
    set of the mapping in \eqref{affine.set.grad}. Hence, by utilizing the trivial fact
    $\Id\in\mathfrak{T}$ and by applying also \cref{prop:convex.comb.affine} to
    $(1-\beta)\Id + \beta\sum\nolimits_m w_m P_{\mathcal{A}_m}$, the claim of \eqref{affine.set.grad}
    is established.

    Regarding \eqref{affine.set.normal}, notice first that $\mathcal{A} = \cap_{d=1}^D
    \mathcal{G}_d$. According to \cref{ex:convex.comb.maps},
    $\mathcal{A} = \Fix(\sum_d \omega_d P_{\mathcal{G}_d})$. Since $P_{\mathcal{G}_d} \in\mathfrak{T}$
    (\cf~\cref{ex:proj.hyperplane}), \cref{prop:convex.comb.affine} yields
    $\sum_d \omega_d P_{\mathcal{G}_d}\in\mathfrak{T}$. As a result, fact $\Id\in\mathfrak{T}$ and
    \cref{prop:convex.comb.affine} yield
    $(1-\theta) \Id + \theta \sum\nolimits_d \omega_d P_{\mathcal{G}_d}\in\mathfrak{T}$, which
    establishes the claim of \eqref{affine.set.normal}. Due to
    $\overline{\mathcal{A}} = \Argmin_{(\vect{x}, \bm{\mu})} \norm{\vect{L} [\vect{x}^{\top},
      \bm{\mu}^{\top}]^{\top} - \vect{e}}^2$, arguments similar to those developed for
    \eqref{affine.set.AtA}, \eqref{affine.set.proj} and \eqref{affine.set.prox} yield
    \eqref{constrained.affine.set.LtL}, \eqref{constrained.affine.set.proj} and
    \eqref{constrained.affine.set.prox}, respectively. Further, notice that since
    $\overline{\mathcal{A}} = \cap_{d=1}^{D+M_0} \mathcal{L}_d$,
    \eqref{constrained.LS.intersection.hyperplanes} is deduced in a way similar to the derivation of
    \eqref{affine.set.normal} from \eqref{LS}.

    Regarding \eqref{constrained.affine.set.grad}, notice that
    $\mathcal{A}_{\mathcal{K}} = \Fix
    T_{\mathcal{A}_{\mathcal{K}}}$~\cite[Prop.~4.2(a)]{Yamada.HSDM.2001}, where
    \begin{align*}
      T_{\mathcal{A}_{\mathcal{K}}} := (1-
      \overline{\beta}) \Id + \overline{\beta} P_{\mathcal{K}}
      \sum\nolimits_{m=1}^M \tfrac{\norm{\bm{\alpha}_m}^2}
      {\norm{\vect{A}}_{\text{F}}^2} P_{\mathcal{A}_m}
    \end{align*}
    is nonexpansive for $\overline{\beta}\in (0, 3/2]$. Since
    $\mathcal{A}_{\mathcal{K}} = \Fix T_{\mathcal{A}_{\mathcal{K}}} = \Fix
    T_{\mathcal{A}_{\mathcal{K}}} \cap \mathcal{K} = \Fix T_{\mathcal{A}_{\mathcal{K}}} \cap \Fix
    P_{\mathcal{K}}$, \cref{ex:composition.maps} suggests that $\mathcal{A}_{\mathcal{K}}$ can be seen
    also as the fixed-point set of the nonexpansive mapping
    $T_{\mathcal{A}_{\mathcal{K}}}P_{\mathcal{K}}$, which is nothing but the mapping appearing at
    \eqref{constrained.affine.set.grad}. Now, due to \cref{prop:convex.comb.affine} and
    \cref{ex:proj.hyperplane}, $\sum_m w_mP_{\mathcal{A}_m}\in\mathfrak{T}$, with
    $w_m := \norm{\bm{\alpha}_m}^2/ \norm{\vect{A}}_{\text{F}}^2$. Hence, \cref{prop:compose.affine}
    suggests also that
    $P_{\mathcal{K}} (\sum_m w_mP_{\mathcal{A}_m}) P_{\mathcal{K}}\in \mathfrak{T}$. Once again, since
    $P_{\mathcal{K}} \in\mathfrak{T}$ (\cf~\cref{prop:T(A).family}), \cref{prop:convex.comb.affine}
    guarantees
    $(1- \overline{\beta})P_{\mathcal{K}} + \overline{\beta} P_{\mathcal{K}}\sum_m
    w_mP_{\mathcal{A}_m} P_{\mathcal{K}}\in \mathfrak{T}$, for $\overline{\beta} \in (0,1]$, which
    establishes the claim of \eqref{constrained.affine.set.grad}.
  \end{proof}

  An auxiliary proposition, used in \cref{thm:Oh}, follows.

  \begin{prop}\label{prop:strongly.pos}
    Given the surjective and strongly positive mapping $\Pi\in\mathfrak{B}(\mathcal{X})$, \ie, there
    exists $\delta\in\RealPP$ s.t. $\innerp{\Pi x}{x}\geq \delta\norm{x}^2$,
    $\forall x\in\mathcal{X}$, the inverse $\Pi^{-1}$ exists and
    $\Pi^{-1}\in \mathfrak{B}(\mathcal{X})$ with $\norm{\Pi^{-1}}\leq 1/\delta$. Moreover, $\Pi^{-1}$
    is strongly positive and
    $(\delta/\norm{\Pi}^2)\norm{x}^2\leq \innerp{\Pi^{-1}x}{x} \leq (1/\delta) \norm{x}^2$,
    $\forall x\in\mathcal{X}$.
  \end{prop}

  \begin{proof}
    \cite[\S2.7, Prob.~7, p.~101]{Kreyszig} guarantees the existence of $\Pi^{-1}$ and
    $\Pi^{-1}\in \mathfrak{B}(\mathcal{X})$. By the strong positivity of $\Pi$,
    $\forall x\in \mathcal{X}\setminus (\Set{0} = \ker \Pi^{-1})$,
    $\norm{\Pi^{-1}x}^2 \leq (1/\delta) \innerp{\Pi^{-1}x}{\Pi(\Pi^{-1}x)} = (1/\delta)
    \innerp{\Pi^{-1}x}{x} \leq (1/\delta) \norm{\Pi^{-1}x}\norm{x} \Rightarrow \norm{\Pi^{-1}x} \leq
    (1/\delta) \norm{x}\Rightarrow \norm{\Pi^{-1}} \leq (1/\delta)$. By \cite[Thm.~9.4-2,
    p.~476]{Kreyszig} and the previous result, $\forall x\in\mathcal{X}$,
    $\innerp{\Pi^{-1}x}{x}\leq \norm{\Pi^{-1}} \norm{x}^2\leq (1/\delta)\norm{x}^2$. Moreover,
    $\forall x'\in\mathcal{X}$,
    $\innerp{\Pi x'}{\Pi^{-1}\Pi x'} = \innerp{\Pi x'}{x'} \geq \delta \norm{x'}^2 \geq
    (\delta/\norm{\Pi}^2) \norm{\Pi x'}^2$, which yields, under $x:= \Pi x'$, that
    $\forall x\in\mathcal{X}$, $(\delta/\norm{\Pi}^2)\norm{x}^2\leq \innerp{\Pi^{-1}x}{x}$.
  \end{proof}

\end{appendices}

%\bibliographystyle{IEEEtran}
%\bibliography{biblio}


\end{document}